\documentclass[oneside,10pt,final]{amsart}
\makeatletter
\@namedef{subjclassname@2020}{%
  \textup{2020} Mathematics Subject Classification}
\makeatother
\usepackage{tikz-cd}
\usepackage[notcite,notref]{showkeys}
\usepackage[a4paper,width=160mm,top=27mm,bottom=27mm]{geometry}
\usepackage{cases}
\usepackage{amsmath,amssymb}
\usepackage{comment}
\usepackage{xcolor}
\usepackage[hidelinks]{hyperref}

\newtheorem{theorem}{Theorem}[section]
\newtheorem{lemma}[theorem]{Lemma}
\newtheorem{definition}[theorem]{Definiton}
\newtheorem{proposition}[theorem]{Proposition}
\newtheorem{corollary}[theorem]{Corollary}

\theoremstyle{definition}
\newtheorem{remx}[theorem]{Remark}
\newenvironment{remark}
  {\pushQED{\qed}\remx}
  {\popQED\endremx}

\newcommand{\PP}{\mathbb{P}}

\newcommand{\R}{\mathbb{R}}
\newcommand{\C}{\mathbb{C}}

\newcommand{\abs}[1]{\lvert#1\rvert}

\numberwithin{equation}{section}

\begin{document}
\address{Filone G. Longmou-Moffo
\newline \indent African Institute for Mathematical Sciences(AIMS-Rwanda)
\newline \indent 17 KN 16 Ave, Nyarugenge District |Kigali, Rwanda
\newline\indent P.O Box,7150, Kigali, Rwanda.
\newline\indent Université Cheikh Anta Diop (UCAD)
\newline\indent Dakar, Sénégal}
\email{longmoumoffofilonegilson@gmail.com, gilson.longmou@aims.ac.rw}

\address{Mouhamadou Sy
\newline \indent African Institute for Mathematical Sciences(AIMS-Rwanda)
\newline \indent 17 KN 16 Ave, Nyarugenge District |Kigali, Rwanda
\newline\indent P.O Box,7150, Kigali, Rwanda}
\email{mouhamadous314@gmail.com, mouhamadou.sy@aims.ac.rw}

\title[NLS with exponential nonlinearity on compact surfaces]{NLS with exponential nonlinearity on compact surfaces}
\author{Filone G. Longmou-Moffo and Mouhamadou Sy}

\begin{abstract}
In this paper, we establish a probabilistic global theory (existence for all times, uniqueness and continuity) in $H^1$ for the NLS equation with a Moser-Trudinger  nonlinearity posed on compact surfaces. This equation is known to be the two dimensional counterpart to the classical energy-critical Schrödinger equations \cite{CollianderIbrahimMajdoubMasmoudi2009}. The authors of \cite{CollianderIbrahimMajdoubMasmoudi2009} also identified a trichotomy around the criticality of the equation based on the size of the total energy. In particular, for supercritical regimes (large energy), the equation is known to exhibit instabilities : the (uniform) continuity of the flow fails to hold. Large data distributional non unique probabilistic solutions have been obtained in \cite{CasterasMonsaingeon2024}. The setting of \cite{CasterasMonsaingeon2024} does not handle  the uniqueness issue for the $H^1$-data and therefore could not define a flow for this regularity. Our main focus here is to build a single probabilistic framework that provides both existence, uniqueness, and continuity with respect to the initial data in $H^1$. Our uniqueness and continuity are based on the so-called Yudovich argument \cite{Judovic1963}, and the probabilistic estimates are derived through the IID limit procedure \cite{Sy2019}. Beyond the difficulties related to the borderline nature of the context, the major challenge resides in the need to satisfy two features that tend to play against each other : obtaining both continuity property of the flow and large data in the support of the reference measure. This made the design of the dissipation operator inherent in the method, as well as the analysis of the resulting quantities, particularly difficult. Regarding the supercritical regime identified in \cite{CollianderIbrahimMajdoubMasmoudi2009}, we show that a modified energy, with regularity similar to the original total energy, admits values as high as desired, suggesting that the constructed set of data contains supercritical ones.
\end{abstract}
\bigskip


\maketitle


\section{Introduction}
\subsection{Context}In this present work, we consider the nonlinear Schr\"odinger equation (NLS) with Moser-Trudinger nonlinearity
\begin{equation}\label{Intro_Trudinger}
    \partial_tu-i\left(\Delta u-(e^{\beta\abs{u}^2}-1)u\right)=0,\quad \quad (u(t,x),t,x)\in \mathbb{C}\times \mathbb{R}\times{M}
\end{equation}
supplemented with an initial condition
\begin{equation}\label{Intro_IC}
    u|_{t=0}=u_0\in H^1(M)
\end{equation}
where $M=(M,g)$ is a compact Riemannian manifold of dimension $2$, and $H^s(M)$ is the Sobolev space of order $s\in\R$ defined on $M$. The operator $\Delta$ is the classical Laplace-Beltrami operator associated to $g$ and $\beta>0$ is a fixed parameter.\\
 When $M$ has boundary, which will be assumed to be smooth enough, we will supplement \eqref{Intro_Trudinger}, \eqref{Intro_IC} with the Dirichlet condition
\begin{equation}
    u|_{\partial M}=0.
\end{equation}    
This equation was introduced in the context of nonlinear optics in \cite{LamLippmannTappert1977} to describe a self-focusing laser beam whose radius is much larger than the vacuum wavelength. This equation also has several applications, notably in quantum mechanics and fluid dynamics, where it is used to describe, respectively, the evolution of self-interacting quantum states in the critical two-dimensional regime and the dynamics of nonlinear waves as well as the concentration of coherent structures such as vortices in two-dimensional flows.
Interpreting the exponential nonlinearity with coefficient $\beta$ as a power series expansion, namely
\[
(e^{\beta |u|^{2}}-1) u = \sum_{k=1}^{\infty} \frac{\beta^{k}}{k!} |u|^{2k} u,
\]
one observes that the equation can be regarded as involving a polynomial nonlinearity of infinite degree. Consequently, the associated critical Sobolev exponent corresponds to the limit case $k = \infty$, leading to the scaling-critical regularity
\[
s_c =\frac{d}{2}.
\] 
From a heuristic point of view, this number indicates the level of regularity above which one typically expects local well-posedness of \eqref{Intro_Trudinger}. In other words, local well-posedness of \eqref{Intro_Trudinger} is generally anticipated for initial data in $H^s, s>s_{c}$ which corresponds to subcritical regime. However for $s=s_c$ (critical regime), irregularities may appear, while for $s<s_c$ (supercritical regime), the equation is expected to be illposed.
Smooth enough solutions of \eqref{Intro_Trudinger} admits the two following conservation laws:
\begin{align}
    M(u) &=\frac{1}{2}\int_{M}|u(t,x)|^2dx, &\text{(Mass)}\\
    E(u) &=\frac{1}{2}\int_{M}|\nabla u(t,x)|^2dx+\frac{1}{2\beta}\int_{M}\left(e^{\beta|u(t,x)|^2}-1-\beta|u(t,x)|^2\right) dx. &\text{(Energy)}
\end{align}
The Cauchy problem \eqref{Intro_Trudinger}, \eqref{Intro_IC} is highly challenging from a mathematical  point of view when the initial data is considered in  $H^1(M)$, where $M$ is of dimension $2$; this indeed corresponds to a double critical regime. Serious obstructions appear in both local and global analysis. Indeed, the critical nature of the regime precludes the control of the exponential nonlinearity in $H^1$, due to the absence of the Sobolev embedding $H^1\hookrightarrow L^\infty$. This lack of embedding obstructs the construction of local solutions to \eqref{Intro_Trudinger} via standard fixed-point arguments and the Duhamel formulation. Moreover, the fact that the nonlinearity is energy-critical prevents the potential part of the energy from being controlled by the kinetic part for large initial data. Consequently, it is impossible to guarantee that any local solution, should it exist, remains bounded in $H^1$
 solely through energy conservation.  Apart from the fact that the notion of local well-posedness becomes extremely delicate to handle in this regime, another major difficulty is that it completely breaks down for large initial data. Indeed, in \cite{CollianderIbrahimMajdoubMasmoudi2009}, Colliander, Ibrahim, Majdoub and Masmoudi investigated this equation on $\mathbb{R}^2$ under the restriction $\beta = 4\pi$ (corresponding to the optimal constant in the Moser-Trudinger inequality). After introducing the notions of subcritical, critical, and supercritical Cauchy problems corresponding respectively to the cases $E(u_0) < \frac{1}{2}$, $E(u_0) = \frac{1}{2}$, and $E(u_0) > \frac{1}{2}$ they established global well-posedness results in $H^1$ for both the subcritical and critical regimes. In contrast, in the supercritical case $E(u_0) > \frac{1}{2}$, they exhibited strong instability via a loss of uniform continuity of the solution, which ultimately leads to the failure of well-posedness theory for large initial data. From a physical standpoint, such instability might indicate that the model loses its predictive power at high energy levels: in this regime, arbitrarily small uncertainties in the initial state can give rise to dramatic deviations in the evolution, suggesting that the equation no longer provides a "computable" description of the underlying dynamics beyond the critical energy threshold. Nonetheless, adopting a probabilistic perspective offers a complementary and powerful insight. By endowing the space of initial data with appropriate probability measures, one can often show that the configurations responsible for such instabilities are of negligible probability. In this sense, the equation may remain well-posed almost surely, thereby restoring a form of predictive validity within a probabilistic framework. This approach not only overcomes the limitations imposed by deterministic instabilities but also highlights the effectiveness of probabilistic methods in validating and extending the applicability of nonlinear models in regimes where classical deterministic analysis fails.\\
 In \cite{CasterasMonsaingeon2024}, Casteras and Monsaingeon  established existence of probabilistic solutions of \eqref{Intro_Trudinger} for large size of data relying on arguments involving invariant measures and the IID limit method (see below). Although this work ensures existence, it left open the uniqueness problem and the continuity issue listed above in \cite{CollianderIbrahimMajdoubMasmoudi2009}, which the authors described as very delicate problem that could potentially fail. This prevented them from being able to properly define a flow. In what follows, we review some recent developments concerning this equation and aim to address the aforementioned issues in order to construct well-behaved global, unique and continuous solutions in $H^1$ while extending over compact surfaces and without usual technical restrictions on the parameter $\beta$ of nonlinearity. 
 
\subsection{Background and earliers results}
The equation of interest has been extensively studied in \cite{CollianderIbrahimMajdoubMasmoudi2009,IbrahimMajdoubMasmoudiNakanishi2012,Nakamura2015,NakamuraOzawa1998,WangHudzik2007,Cazenave1979} in the Euclidean space $\mathbb{R}^2$ with $\beta=4\pi$, particularly in \cite{CollianderIbrahimMajdoubMasmoudi2009}, the authors  Colliander, Ibrahim, Majdoub and Masmoudi have proven the global existence result for small data. Their approach relied on a fixed-point argument, combined with the Moser-Trudinger inequality, Brézis-Gallouët-type logarithmic and Strichartz estimates. In fact the Strichartz estimate \cite{BurqGerardTzvetkov2004,KeelTao1998,Strichartz1977} \[
\| e^{it\Delta}f \|_{L^r_t(\R;L^p_x(\mathbb{R}^2))} \leq C \| f \|_{L^2_x(\mathbb{R}^2)}, \quad \text{with} \quad r \geq 2 \quad\text{and $(r,p)\neq (2,\infty)$}\quad\text{satisfying}\ \ \frac{2}{r}+\frac{2}{p}=1;
\]
help to control the $L^4C^{\frac{1}{2}}$ norm of $e^{it\Delta}f$ by the $H^1$ norm of $f$ which allows to perform the contraction argument in $X(T) = C([0,T]; H^1(\mathbb{R}^2)) \;\cap\; L^4([0,T]; C^{1/2}(\mathbb{R}^2))
$. The difficulty here is to control the $L^{\frac{4}{3}}_TL^4_x$ norm of $e^{(4\pi+\epsilon)|u|^2}-1$ which can be handle using the estimate:
\begin{align}
\left\| e^{4\pi(1+\varepsilon)|u(t,\cdot)|^2} - 1 \right\|_{L^{4/3}_T(L^4_x)}
&\le
\left\| e^{3\pi(1+\varepsilon)||u(t,\cdot)||_{L^\infty}^2} \right\|_{L^\frac{4}{3}_T}
\;
\left\| e^{4\pi(1+\varepsilon)|u(t,\cdot)|^2} - 1 \right\|_{L^\infty_T(L^1_x)}^{1/4}.
\end{align}
The first term can be controlled by using the Brézis–Gallouët-type logarithmic estimate \cite{IbrahimMajdoubMasmoudi2007}:\begin{align}
\|u\|_{L^\infty}^2 \;\le\; \lambda \, \|u\|_{H^1}^2 \,
\log\Biggl( C_\lambda + \frac{\sqrt{8} \, \|u\|_{C^\frac{1}{2}}}{ \, \|u\|_{H^1}} \Biggr) ~\text{for any}~ \lambda>\frac{1}{\pi}, u\in H^1(\mathbb{R}^2)\cap C^\frac{1}{2}(\mathbb{R}^2).
\end{align}
while the second is handled by using the Moser-Trudinger inequality \cite{AdachiTanaka1999,Moser1971,Trudinger1967}:
\begin{align}
 \sup_{u\in H^1,\lVert\nabla u\rVert_{L^2}\leq 1 }\|e^{\beta |u|^2} - 1 \|_{L^1(\mathbb{R}^2)} \;\le\; C_\beta \, \|u\|_{L^2(\mathbb{R}^2)}^2 ~\text{for}~\beta\in[0,4\pi)~
\end{align}
The use of the Moser-Trudinger inequality here conditions this local theory for restricted size data $E(u_0)\leq \frac{1}{2}$. For subcritical initial data, the local theory can be extended globally rather straightforwardly by using energy conservation and the fact that the local existence time depends only on 
$\lVert u_0\rVert_{L^2}^2$ and \(\eta\), with \(\eta\) such that 
\(\|\nabla u_0\|_{L^2}^2 < 1 - \eta\) while
in the critical case, the situation is more delicate, but it can be circumvented by employing a finite-time concentration phenomenon argument.
 For supercritical initial data, it has been proven that the equation is highly unstable: infinitesimal differences in the initial state can lead to large deviations in arbitrarily short times. This demonstrates a fundamental loss of continuity and the breakdown of classical local well-posedness.
The same authors demonstrated in \cite{IbrahimMajdoubMasmoudiNakanishi2012} that scattering holds for global solutions arising from small initial data, with $\beta = 4\pi$. However, in the compact manifold, scattering cannot occur, so we can ask what the long time behavior of the flow looks like. To answer these questions, several studies have been developed using probabilistic approaches.\\\\
The pioneering work in this direction has been done by Bourgain \cite{Bourgain1994,Bourgain1996}
in the context of 1-D NLS.The principal obstacle to establishing global solutions at such low levels of regularity lies in the absence of conventional conservation laws. To overcome this difficulty, Bourgain employed the Gibbs measure constructed by Lebowitz, Rose, and Speer \cite{LebowitzRoseSpeer1988}, which is supported on $H^{1/2-}$, to derive uniform bounds for certain $N$-dimensional Galerkin projections of the equation projections that had already played a role in the construction of the measure. Crucially, these bounds were independent of the dimension $N$. By carefully comparing the full infinite-dimensional problem with its finite-dimensional approximations, Bourgain succeeded in proving global existence of solutions. In essence, the use of invariant measures enables the persistence of conservation laws at low regularity, albeit in a statistical or probabilistic sense.\\
In the fractional version of our model \eqref{Intro_Trudinger}, the Gibbs measure was constructed by Robert \cite{Robert2024InvariantGibbs}. He demonstrated that the measure is well-defined for $\alpha > \frac{d}{2}$ and is supported on $H^{(\alpha-\frac{d}{2})^-}$. Subsequently, this Gibbs measure was employed to construct global solutions. Specifically, for $\alpha > \frac{d}{2}$, $0 < \beta < \beta_0$, and $0 < s < \alpha - \frac{d}{2}$, he established the existence of a random variable taking values in $C(\mathbb{R},H^s)$ that solves the equation. Moreover, in the regime $\alpha > d$, $0 < \beta < \beta_0$, and $\frac{d}{2} < s < \alpha - \frac{d}{2}$, the local well-posedness results allowed him to extend the local flow globally and to prove that the Gibbs measure is invariant under this flow. Naturally, this framework cannot address the case $\alpha = \frac{d}{2}$, which is the focus of the present work. A key limitation of this Gibbs measure-based approach is that its support lies in a rough function space, resulting in a lack of spatial regularity.\\\\
An alternative framework for the construction measures is provided by the fluctuation–dissipation approach, introduced by Kuksin \cite{Kuksin2004Eulerian} and further developed with Shirikyan \cite{KuksinShirikyan2004}. This method relies on adding a suitable stochastic perturbation to the equation designed to enhance dissipation and to admit a stationary measure for each fixed viscosity. Passing to the vanishing viscosity limit via compactness arguments then leads to invariant measures for the original deterministic dynamics which can subsequently be exploited to establish long-time existence results. We refer to \cite{FoldesSy2021,KuksinShirikyan2012,Latocca2022InvariantEuler,Sy2018BenjaminOno} for further developments along these lines. Although this method generally does not encounter regularity issues of the support, it is limited by a lack of time integrability.\\\\
To overcome these challenges, M. S. \cite{Sy2019} developed a method referred to as the IID limit, which combines the fluctuation–dissipation approach with Gibbs measure techniques. The central idea is to apply the fluctuation–dissipation method to finite-dimensional Galerkin approximations in order to construct a stationary measure and subsequently pass to the inviscid limit. The limiting measure is carefully designed to be invariant for the approximated deterministic system, while the dissipation is chosen to provide uniform estimates that enable the implementation of Bourgain’s globalization procedure in regimes where local theory is well-posed. In less favorable regimes, compactness arguments combined with Skorokhod’s theorem are employed to construct probabilistic solutions . We refer to \cite{gueye2025probabilisticglobalwellposednessenergysupercriticalschrodinger,Sy2019,SyYu2020,SyYu2021Fractional,CasterasMonsaingeon2024}
for related results. In \cite{CasterasMonsaingeon2024} Casteras and Monsaingeon have studied the fractionnal case of our model \eqref{Intro_Trudinger} by employing this IID limit method. They constructed invariant measures and established global-in-time solutions in regimes of strong regularity, namely for $\alpha \leq \frac{d}{2}$ and $s > \frac{d}{2}$. In contrast, in less favorable regularity, specifically when $\alpha \leq 1$ and $s \leq \min\big(\frac{d}{2},1+\alpha\big)$, they proved the existence of probabilistic solutions. Although these latter solutions concern large size of data, they lack sufficient regularity belonging only to $C(\mathbb{R},H^{s-\alpha})$ and raise a serious issue of uniqueness, which may in fact fail.\\
In this paper, we address this uniqueness problem for the model \eqref{Intro_Trudinger} and go further by also addressing the continuity issue with respect to the initial data in $H^1$ listed above in \cite{CollianderIbrahimMajdoubMasmoudi2009}. To overcome these problems, we use a valuable argument from Yudovich developed in \cite{Judovic1963} in the context of 2-D Euler (for further use of the method, one can refer to Foldes and Sy \cite{FoldesSy2021} in the context of 2-D SQG and by Burq, Gerard, and Tzvetkov \cite{BurqGerardTzvetkov2004} in the context of 3-D cubic NLS.) Below, we present a detailed exposition of the original Yudovich argument.
\subsection{Yudovich argument for the uniqueness problem of the 2D Euler}
A deep understanding of the behavior of incompressible fluids particularly through the lens of the Euler equations is a cornerstone of both fluid mechanics and the analysis of nonlinear partial differential equations. In two space-dimensions, the structure of these equations becomes especially amenable to mathematical investigation due to the conservation of vorticity. However, once one moves beyond the realm of smooth solutions, the study of weak solutions raises profound questions of well-posedness, and in particular of uniqueness.
The famous Yudovich argument\cite{Judovic1963} allows to achieve this by a quite elegant fashion. This powerful technique plays the role of a substitute for the classical Gronwall inequality, particularly adapted for such borderline situations.\\
Yudowich’s analysis focuses on the two-dimensional Euler equations in stream function form:
\[
\partial_t \Delta\psi + \frac{\partial\psi}{\partial x_2}\frac{\partial \Delta\psi}{\partial x_1}- \frac{\partial\psi}{\partial x_1}\frac{\partial \Delta\psi}{\partial x_2} = -\text{curl}~f(t,x), \quad  \psi|_{\partial\Omega}=0,~~~\psi(0,x)=\psi_0(x),\quad u = \nabla^\perp \psi~~ \text{and}~~ \text{curl}~ u=-\Delta\psi.
\]
where \( \psi \) denotes the stream function, \( u \) is the velocity field, and $w=\text{curl}~ u=-\Delta\psi$ is the vorticity. \\
Let $\psi_1,\psi_2$ be two solutions starting at $\psi_0$, by taking the difference of theses two solutions $\psi=\psi_1-\psi_2$ and the inner product with $\psi$,{while using Green's formula and the homogeneous Dirichlet boundary condition}, we have :
\begin{align}
    -\int_{\Omega}|\nabla\psi(x,t)|^2dx+\int_0^t\int_\Omega\left[-\Delta\psi\frac{\partial\psi}{\partial \tau}+\Delta\psi\left(\frac{\partial\psi_1}{\partial x_1}\frac{\partial\psi}{\partial x_2}-\frac{\partial\psi_1}{\partial x_2}\frac{\partial\psi}{\partial x_1}\right)\right]dxd\tau=0
\end{align}
By using the integration by parts, we obtain:
\begin{align}
    \frac{1}{2}\int_\Omega|\nabla \psi|^2dx=\int_0^t\int_\Omega\left[\frac{\partial\psi_1}{\partial x_1\partial x_2}\left(\left(\frac{\partial\psi}{\partial x_1}\right)^2-\left(\frac{\partial\psi}{\partial x_2}\right)^2\right)+\left(\frac{\partial^2\psi_1}{\partial x_2^2}-\frac{\partial^2\psi_1}{\partial x_1^2}\right)\frac{\partial\psi}{\partial x_1}\frac{\partial\psi}{\partial x_2}\right]dxd\tau
\end{align}
At this stage, with the aim of applying Gronwall’s lemma, we write
\begin{align}\label{Gronwall Euler }
    \frac{1}{2}\int_\Omega|\nabla \psi|^2dx\leq 4\int_0^t\int_\Omega |D^2\psi_1 | |\nabla \psi|^2dxd\tau.
\end{align}
However, in the context of interest,  $D^2\psi_1$ does not belong to $L^1_tL^\infty_{x}$. This not only prevents the nonlinearity  from being Lipschitz monotone but also makes Gronwall’s lemma inapplicable due to this lack of regularity, which constitutes the main difficulty in the analysis. Nevertheless,
by using the $L^\infty$ conservation of the vorticity, one has $\Delta\psi_1\in L^\infty_{t,x}$. This suggests that Gronwall’s lemma could apply if we could expect an estimate of the form $\left\lVert \frac{\partial^2\psi_1}{\partial_ {x_j}\partial_{x_k}}\right\rVert_{L^\infty_x}\leq C \lVert \Delta \psi_1\rVert_{L^\infty_x} \quad\forall j,k=1,2$; which is not the case. However, by the theory of Calderón–Zygmund operators (see lemma 2.2 in \cite{Judovic1963}), one has $\left\lVert \frac{\partial^2\psi_1}{\partial_ {x_j}\partial_{x_k}}\right\rVert_{L^p_x}\leq C ~p\lVert \Delta \psi_1\rVert_{L^\infty_x}$  $\forall p<\infty$ and $\frac{\partial\psi}{\partial x_1},\frac{\partial\psi}{\partial x_2}$ are bounded. \\
Yudovich’s argument precisely consists in bypassing this lack of regularity. Indeed, our situation lies at the boundary between an  available a priori estimate and the kind of estimate we actually need to control.
 The idea is to slightly relax the problem by introducing a small parameter $\varepsilon\to 0$, in order to make it possible to apply the available a priori estimate, even at the cost of losing the linear Gronwall estimate above, and replacing it with a new nonlinear Osgood-type estimate depending on $\varepsilon$, which converges back to the linear Gronwall estimate as $\varepsilon\to 0$. 
 To this end, setting $z^2(t)=\int_\Omega|\nabla\psi|^2dx$ and applying the Hölder inequality while using the boundedness of $\nabla\psi$ yields:
\begin{align*}
    \frac{d\,z^2}{dt}&\leq 8M^{\varepsilon}\int_\Omega\left|D^2\psi_1\right|\left|\nabla\psi\right|^{2-\varepsilon}dx
    \leq 8 M^{\varepsilon} \left\lVert D^2\psi_1\right\rVert_{L^{\frac{2}{\varepsilon}}}z^{2-\varepsilon}
 \leq M^\varepsilon C ~{\tfrac{2}{\varepsilon}}~M_1z^{2-\varepsilon}.
\end{align*}
Hence $z(t)\leq M(2CM_1t)^\frac{1}{\varepsilon}$. Then, letting $\varepsilon\to 0$ for sufficiently small $t\leq t_0$, we get $z(t)=0$ and by repeating the same argument step by step , we obtain $z(t)=0 \quad\forall t\geq 0$ which conclude the uniqueness.

\subsection{Problematic, Methodology and Exposition of the main result}

{
Let us set
\[
u=e^{-it}v.
\]
We see that if \(v\) solves \eqref{Intro_Trudinger}, then \(u\) solves the following equation:
\begin{equation}\label{equation without zero frequency}
\partial_t u
=
i\left((\Delta-1)u-(e^{\beta|u^2|}-1)u\right).
\end{equation}
This formulation is more appropriate for dealing with the zero frequency.
Its energy is given by
\begin{equation}
E(u)
=
\int_{M}
\left(
\frac12 |u|^2
+
\frac12 |\nabla u|^2
+
\frac{1}{2\beta}\left(e^{\beta|u|^2}-1-\beta|u|^2\right)
\right)\,dx.
\end{equation}
}
Let us formally present the major questions involved in the present work :
\begin{enumerate}
    \item Dropping the restriction on the parameter $\beta$ ;
    \item The uniqueness problem in the context of probabilistic settings for $H^1$ data is unknown. A major aspect of our work is to establish such uniqueness ;
    \item {Addressing the continuity issue and the problem about the size of data: As established by deterministic arguments in \cite{CollianderIbrahimMajdoubMasmoudi2009}, the equation displays a significant instability phenomenon. More precisely, the flow map is not uniformly continuous in $H^1$
 for initial data of large energy, namely when $E(u_0)>1$}. Such tension between continuity (not necessarily uniform) and large data property will be at the heart of the design of our dissipation operator. Namely, it has to ensure both properties at the same time. 
\end{enumerate}

We firstly present a general framework of Yudovich's argument for PDEs and we then employ the IID limit method outlined above, highlighting the delicate and technically demanding construction of the dissipation mechanism. As this stage of the work is primarily exploratory, the full details of this construction are generally omitted, but it is one of the most tricky steps  in this method. Once the Galerkin approximation of the equation is considered, once perform the following dissipation that will be used in the fluctuation-dissipation approximation :For any $u\in E_N$
{\begin{align}\label{dissipation}
    \mathcal{L}(u)= Ce^{\gamma C_1\lVert P_N(e^{\beta\abs{u}^2}u)\rVert_{L^2}^2}\left(u+P_N(e^{\beta\abs{u}^2}u)\right)+C_2P_N(e^{\gamma\abs{u}^2}u)+P_N\left((1-\Delta)^{-\frac{1}{2}}\left(|(1-\Delta)^{\frac{1}{2}}u|^\delta(1-\Delta)^{\frac{1}{2}}u    \right)\right).
\end{align}}
The dissipation must provide sufficient regularity and integrability to the system to produce crucial estimates. In a certain extent, these estimates will survive the passage to the limits, and are instrumental for:
\begin{itemize}
    \item 
Constructing solutions using the invariant measure and Skorokhod’s theorem.

\item Controlling the nonlinearity to establish global existence.
\item Ensure that large size of data are included.
\item Allowing the use of Yudovich’s argument to prove uniqueness and continuity with respect to the initial data.
\end{itemize}
{In fact the role of the term $u$ in dissipation is to enable the  construction of an invariant measure supported in $H^1$
 , through the dissipative energy. The second term $P_N(e^{\beta|u|^2}u)$ ensures control of the system's non-linearity, which allows to establish the global existence of the solution. The third term $P_N(e^{\gamma|u|^2}u)$ is crucial in ensuring uniqueness by allowing the use of  Yudovich's argument. Indeed, through the dissipative energy, we have the control $\sum_{p\in\mathbb{N}}\frac{\gamma^p}{p!}\lVert \nabla u^{p+1}\rVert_{L^2}^2$, which is crucial to control\\ $\lVert\nabla(e^{\beta|u|^2}-1)\rVert_{L^2}$. The term $e^{\gamma C_1\lVert P_N(e^{\beta|u|^2}u)}\rVert_{L^2}^2$ manages the interaction between $P_N(e^{\beta|u|^2}u)$  and $P_N(e^{\gamma|u|^2}u)$, in $\langle P_N(e^{\beta|u|^2}u),P_N(e^{\gamma|u|^2}u)  \rangle$. The last term $P_N\left((1-\Delta)^{-\frac{1}{2}}\left(|(1-\Delta)^{\frac{1}{2}}u|^\delta(1-\Delta)^{\frac{1}{2}}u    \right)\right)$ is crucial in producing a control of the term $\lVert u\rVert_{W^{1,2+\delta}}^{2+\delta}$ which allows to bound the term $\lVert \nabla((e^{\beta|u|^2}-1)u)\rVert_{L^2}$ by applying the Kato-Ponce inequality and to show by using the Duhamel's formulation that the solution belongs to $C(\mathbb{R},H^1)$, and is continuous with respect to the initial data in $H^1$ via Yudovich's argument}. 
Moreover, if we focus on the estimate gained from the evolution of the mass functional, one uniformly controls the following finite dimensional quantity (the dissipation of the mass):
\begin{align*}
\mathcal{M}_N(u_N)&={Ce^{\gamma C_1\lVert P_N(e^{\beta\abs{u_N}^2}u_N)\rVert_{L^2}^2}}\lVert u_N\rVert_{L^2}^2+{Ce^{\gamma C_1\lVert P_N(e^{\beta\abs{u_N}^2}u_N)\rVert_{L^2}^2}}\sum_{p\in\mathbb{N}}\frac{\beta^p}{p!}\lVert u_N\rVert_{L^{2p+2}}^{2p+2}\\
&\quad \quad
+C_2\sum_{p\in \mathbb{N}}\frac{\gamma^p}{p!}\lVert u_N\rVert_{L^{2p+2}}^{2p+2}+\left\langle u_N,P_N\left((1-\Delta)^{-\frac{1}{2}}\left(|(1-\Delta)^{\frac{1}{2}}u_N|^\delta(1-\Delta)^{\frac{1}{2}}u_N    \right)\right)\right\rangle.
\end{align*}
The dissipation of the mass can morally be considered as a modified energy, as it involves terms that have similar regularity as the total energy. We then define :
\begin{itemize}
    \item Modified potential energy (finite-dimensional)
    \begin{align}
        \mathcal{V}_N(u_N)&={Ce^{\gamma C_1\lVert P_N(e^{\beta\abs{u_N}^2}u_N)\rVert_{L^2}^2}}\lVert u_N\rVert_{L^2}^2+{Ce^{\gamma C_1\lVert P_N(e^{\beta\abs{u_N}^2}u_N)\rVert_{L^2}^2}}\sum_{p\in\mathbb{N}}\frac{\beta^p}{p!}\lVert u_N\rVert_{L^{2p+2}}^{2p+2}\\
&\quad \quad
+C_2\sum_{p\in \mathbb{N}}\frac{\gamma^p}{p!}\lVert u_N\rVert_{L^{2p+2}}^{2p+2}
\end{align}
    \item  Modified kinetic energy (finite-dimensional)
    \begin{align*}
    \mathcal{K}_N(u_N)&=\left\langle u_N,P_N\left((1-\Delta)^{-\frac{1}{2}}\left(|(1-\Delta)^{\frac{1}{2}}u_N|^\delta(1-\Delta)^{\frac{1}{2}}u_N    \right)\right)\right\rangle=\langle u_N,\, g_{\delta,N}\rangle.
\end{align*}
\end{itemize}
Exploiting the dissipation of the energy, which has a fundamental role to ensure compactness, on can pass to the infinite-dimensional limit to obtain the following quantities :
\begin{itemize}
  \item Modified potential energy (explicit in $u$)
    \begin{align}
        \mathcal{V}(u)&={Ce^{\gamma C_1\lVert e^{\beta|u|^2}u\rVert_{L^2}^2}}\lVert u\rVert_{L^2}^2+{Ce^{\gamma C_1\lVert e^{\beta|u|^2}u\rVert_{L^2}^2}}\sum_{p\in\mathbb{N}}\frac{\beta^p}{p!}\lVert u\rVert_{L^{2p+2}}^{2p+2}
+C_2\sum_{p\in \mathbb{N}}\frac{\gamma^p}{p!}\lVert u\rVert_{L^{2p+2}}^{2p+2}.\label{Modif_Pot_En}
\end{align}
    \item  Modified kinetic energy (defined only as a random variable : $\Omega:\to\R$)
    \begin{align*}
    \mathcal{K}^\omega&=\langle u^\omega,\, g_{\delta}^\omega\rangle,\label{Modif_Kin_En}
\end{align*}
where $g_\delta^\omega$ is a random variable $\Omega\to L^2$, defined as  the limit (w.r.t. $L^2$-norm) of $g_{\delta,N}$ (this limit is guaranteed to exist by estimates on the energy dissipation, although we fail to explicitly express it in terms of $u$).
\item  Modified total energy (as a random variable)
\begin{align}
    \mathcal{M} &= \mathcal{K}+\mathcal{V}.
\end{align}
\end{itemize}

Here is our main result.
\begin{theorem}
    Let $M$ be a compact  surface endowed with a Riemannian metric, let $\beta>0, ~\delta\in(0,1)$ and $\gamma >4\beta$.\\ There exist a probability space $(\Omega,\mathcal{F},\mathbb{P})$, a probability measure $\mu \in P(L^2)$, a set $\Sigma\subset H^1$ and a stochastic process $u$ valued in $C(\mathbb{R},H^1)\cap L^{2+\delta}_{loc}W^{1,2+\delta}$ such that:
    \begin{itemize}
    \item $\mu(\Sigma)=1$.
    \item $\mu$ is a invariant law for $u(t)$.
        \item $u=u^\omega$ is a global strong solution of \eqref{equation without zero frequency} for $\mathbb{P}$ a.a ~~ $\omega\in \Omega$.
        \item $u$ satisfies a uniqueness property in the sense that : for any two realizations $\omega$ and $\omega'$, if $u^\omega(t_0)=u^{\omega'}(t_0)$, then $u^\omega(t)=u^{\omega'}(t)$ for all $t$. This, in particular, defines a flow $\phi_t$ on $\Sigma$:
        \begin{align}
            \phi_t:\Sigma\to \Sigma : \quad u_0\mapsto \phi_t(u_0):=u(t)_{|u(0)=u_0}.
        \end{align}
        \item $\phi_\cdot\, :\Sigma\to \Sigma$ is $H^1$ continuous : $\forall\, T>0$ and all $u_0,\, v_0$ in $\Sigma$,
        \begin{align}
            \lim_{\lVert u_0-v_0\rVert_{H^1}\to 0}\sup_{t\in [-T,T]}\lVert \phi_t(u_0)-\phi_t(v_0)\rVert_{H^1}=0.
        \end{align}
        \item We have
            \begin{align*}
\mathbb{E}\mathcal{M}&=\frac{A^0}{2},
\end{align*}

where $\mathcal{M}$ is the modified total energy defined above, and $A^0$ is an arbitrarily prescribed constant.
        \item The support of $\mu$ contains data of arbitrary modified potential energy level \[\forall R>0,~\mu(\mathcal{V}(u)>R)>0.\] \label{Support property}
 \item The distribution under $\mu$ of the mass functional has density with respect to the Lebesgue measure on $\R$
 \item The measure $\mu$ is of at least two-dimensional nature.
 \end{itemize}
\end{theorem}
\begin{remark}
    As the modified potential energy has similar regularity as the modified potential energy, the large modified potential  energy property suggests that the support of the measure contains supercritical data in the sense of \cite{CollianderIbrahimMajdoubMasmoudi2009}. Further properties of the support (absolute continuity of the mass and lower bound on the dimension) establish nontriviality of the measure.
\end{remark}
\begin{remark}
It has been established that the support of $\mu$ is neither concentrated at $0$ nor on the set of nontrivial stationary solutions, since the latter do not exist for the present equation. Indeed, if $\phi(x)$ were a nontrivial stationary solution, then it would satisfy
 $(\Delta-1) \phi=(e^{\beta|\phi|^2} -1)\phi$. Multiplying this equation by $\overline{\phi}$, integrating over $M$, and using Green's formula together with the homogeneous Dirichlet boundary condition in the case of a compact manifold with boundary, one immediately arrives at a contradiction.
   It would also be of considerable interest to determine whether the support of $\mu$ consists exclusively of nontrivial standing waves or not. However, this issue is substantially more delicate in the present setting, as even the existence of nontrivial standing waves remains an open problem. To the best of our knowledge, there is currently no result establishing the existence of nontrivial standing waves for the nonlinear Schrödinger equation with Moser Trudinger type nonlinearity in the regime of arbitrary values of $\beta$ and arbitrarily large data. More precisely, a nontrivial standing wave of the form $u=e^{iwt}\phi(x)$ solution to the equation, if it existed, would necessarily satisfy $ -\Delta\phi+(1+w)\phi+(e^{\beta|\phi|^2}-1)\phi=0$\quad with $w<-1$.  The existing constructions of standing waves rely therefore on variational methods for the associated elliptic problem, which fundamentally use Moser Trudinger type inequality and typically require restrictions on $\beta$ and the data to guarantee the coercivity of the potential energy.\\
    In this sense, our results reveal the existence of a genuinely new global dynamical regime, corresponding to arbitrary values of $\beta$ and arbitrarily large initial data, for which the existence of nontrivial standing waves is presently unknown, since the standard approaches based on variational methods to the associated elliptic problem appear to be insufficient to describe this regime.
\end{remark}
\subsection{Difficulties}
    It is worth mentioning the following noticeable difficulties :
    \begin{itemize}
        \item (The dissipation building) The main focus in this article is to construct a probabilistic framework which ensures both global regularity (global existence, uniqueness and continuity) and some large data property. These two properties appear to play against each other (as can be expected from the trichotomy of \cite{CollianderIbrahimMajdoubMasmoudi2009}). This makes to dissipation operator - which is the source of estimates - very tricky to build. The analysis of quantities that it produces introduces several difficulties throughout the paper.
        \item (The uniqueness challenge) Combining the estimates obtained from the dissipation operator and the 'pointwise' boundedness property of the energy, we show that the solution has regularity $C_tH^1\cap L^{2+\delta}_tW^{1,2+\delta}_x$ (alongside some $L^p_{t,x}$). However, to prove uniqueness, one nearly requires  $L^\infty_{t,x}$ bound (as a matter of fact, exponential integrability in time of the $L^\infty_x$ norm). While the $L^p_{t,x}$ does not reach any of them, $C_tH^1$ reaches only $L^\infty_t$ but not $L^\infty_x$ and $L^{2+\delta}_tW^{1,2+\delta}_x$ reaches only $L^\infty_x$  but is far from $L^\infty_t$.
        \item $H^1$ continuity will also face similar challenges.
        \item Nontriviality, from which follows the large data property, requires a delicate passage to the limit procedure to properly define all terms involved in the modified energy $\mathcal{M}(u)$ and to prevent a loss of mass. 
    \end{itemize}

\subsection{General notations}
\leavevmode\par
We denote by $M:=(M,g)$ a compact Riemannian manifold of dimension $2$ with or without boundary, $\Delta$  the Laplace Beltrami operator on $M$ and $(e_n)_{n\in \mathbb{N}}$ a normalized eigenfunctions basis of $-\Delta$, orthonormal in $L^2(M)$. The associated eigenvalues are noted $(\lambda_n)_{n\in \mathbb{N}}$ such that $0=\lambda_0 <\lambda_1\leq \lambda_2 \leq .....\leq \lambda_n \leq ...$.\\
For any $u\in L^2(M)$, we have $u(x)=\sum_{n\in \mathbb{N}}u_ne_n(x)$ and,
by Parseval identity, $$\lVert u \rVert^2_{L^2}:=\int_{M}{\abs{u(x)}^2 dx}=\sum_{n\in \mathbb{N}}\abs{u_n}^2.$$
We denote by $\lVert~~ \rVert_{L^p}$ the norm of the Lebesgue space $L^p(M),\ p\in[1,\infty]$, $\lVert \quad \rVert_{H^s}$ the norm of Sobolev space $H^s(M),\ s\in\R$ and  $|~|$ the modulus of a complex number.\\
The Sobolev norm of $H^s$ is defined by $$\lVert u \rVert_{H^s}^2:=\lVert{(1-\Delta)^{\frac{s}{2}}u}\rVert_{L^2}^2={\sum_{n\in \mathbb{N}}(1+\lambda_n)^s\abs{u_n}^2}.$$
For a Banach space \( X \) and an interval \( I \subset \mathbb{R} \), we denote by \( C(I,X) = C_I^tX \) the space of continuous functions \( f : I \to X \). The corresponding norm is 
$\|f\|_{C^t_IX} = \sup_{t \in I} \|f(t)\|_X.$\\
For \( p \in [1,\infty) \), we also denote by \( L^p(I,X) = L^{p}_IX \), the spaces given by the norm
\[
\|f\|_{L^{p}_IX} = \left( \int_I \|f(t)\|_X^p \, dt \right)^{\frac{1}{p}}.
\]
We denote by $E_N$ the subspace of $L^2$ generated by the finitely many family $\{e_n ;n\leq N\}$, the operator $P_N$ is the projector onto $E_N$ and $E_{\infty}$ refers to $L^2$, and the projector $P_{\infty}$ to the identity operator.\\
The inequality $A \lesssim B$ between two positive quantities $A$ and $B$ means
$A \le C B$ for some $C > 0$.
\section{On a general framework of the Yudovich argument}
In this section, we formulate a framework for Yudovich’s argument for the uniqueness of PDEs posed in $H^{\frac{d}{2}}$, where $d$ is the dimension of the compact manifold $M$. Here, the lack of embedding $H^{\frac{d}{2}}\hookrightarrow L^\infty$
 may render the nonlinearity non-Lipschitz monotone, which prevents the use of standard tools to prove the uniqueness. Nevertheless, using Yudovich’s argument, we are able to overcome this problem of regularity.
\begin{theorem}\label{yudowich theorem}
   Let $M$ be a compact manifold of dimension $d$; let us consider the general abstract PDE of the form: \begin{equation}\label{abstract equation}
\left\{
\begin{aligned}
    \frac{du}{dt} + Au(t) &= F(u(t)) + f(t,x), \\
    u(0) &= u_0\in H^{\frac{d}{2}}
\end{aligned}
\right.
\end{equation}

       where  $u:\mathbb{R}\to L^2(M,\mathbb{K})$, $A:V\to V^*$ is a linear and coercive or positive semi-definite operator such that we have the Gelfand triple $\left(V\hookrightarrow L^2\hookrightarrow V^*\right)  $ and $F$ is a non-linear function or operator such that for all $u_1, u_2$ 
       \begin{align*}
           \langle F(u_1)-F(u_2),u_1-u_2\rangle\lesssim \langle H(u_1-u_2)G(u_1,u_2),u_1-u_2\rangle.
           \end{align*}
           Let us assume that $\forall r\in (1,2),~ \lVert H(u_1-u_2)\rVert_{L^{2r}}\lesssim \lVert u_1-u_2\rVert_{L^{2r}}$ and  there exist $p\geq 1$ such that  
           \begin{align}\label{Yudovich estimate}
           \forall t_0>0, ~~\int_0^{t_0}\left(1+\lVert u_1(\tau)\rVert_{\dot{H}^{\frac{d}{2}}}^p+\lVert u_2(\tau)\rVert_{\dot{H}^\frac{d}{2}}^p\right)\lVert G(u_1(\tau),u_2(\tau))\rVert_{\dot{H}^\frac{d}{2}}d\tau<\infty.
           \end{align}
       Then the solution of \eqref{abstract equation} is unique.
\end{theorem}
\begin{proof}
   Let $u_1, u_2$ two solutions of \eqref{abstract equation} starting at $u_0$  then we have:
   \begin{equation}
\left\{
\begin{aligned}
    \frac{du_1}{dt} + Au_1(t) &= F(u_1(t)) + f(t,x), \\
     \frac{du_2}{dt} + Au_2(t) &= F(u_2(t)) + f(t,x).
\end{aligned}
\right.
\end{equation}
Let $u=u_1-u_2$; by talking the equation in $u$ and performing the scalar product with $u$, we have:
\begin{align*}
    \langle \frac{du}{dt},u(t)\rangle+\langle Au(t),u(t)\rangle=\langle F(u_1)-F(u_2),u_1-u_2\rangle\leq C\langle H(u_1-u_2)G(u_1,u_2),u_1-u_2\rangle
\end{align*}

 $\forall r\in (1,2),~2<2r<\frac{2r}{r-1}$ and by applying the Holder inequality with $1=\frac{1}{2}+\frac{1}{2r}+\frac{r-1}{2r}$, we have:
 \begin{align*}
    \frac{1}{2}\frac{d}{dt}\lVert u\rVert_{L^2}^2\leq C\lVert H(u)\rVert_{L^{2r}}\lVert G(u_1,u_2)\rVert_{L^{\frac{2r}{r-1}}}\lVert u\rVert_{L^2}\leq C\lVert u\rVert_{L^{2r}}\lVert G(u_1,u_2)\rVert_{L^{\frac{2r}{r-1}}}\lVert u\rVert_{L^2}.
\end{align*}
Since
$
    \lVert u\rVert_{L^{2r}}\leq C'\lVert u\rVert_{L^{2}}^{2-r}\lVert u\rVert_{L^{\frac{2r}{r-1}}}^{r-1},
$
then we have 
$
    \frac{d}{dt}\lVert u\rVert_{L^2}^2\leq C_1\lVert u\rVert_{L^2}^{2(\frac{3-r}{2})}  \lVert u\rVert_{L^{\frac{2r}{r-1}}}^{r-1}\lVert G(u_1,u_2)\rVert_{L^{\frac{2r}{r-1}}}
    $.\\
    Let us denote by $\lambda=\frac{3-r}{2}<1$.\\ We then have by using the fact that in dimension d, the norm $H^\frac{d}{2}$controls all the $L^q$ norm with $q<\infty$, the Sobolev inequality with precise constant $\lVert v\rVert_{L^\delta}\leq C\sqrt{\delta}\lVert v\rVert_{\dot{H}^\frac{d}{2}}$ and the fact $\left(\frac{2r}{r-1}\right)^\frac{r-1}{2}$ is bounded \\for $1<r<2$, $\sqrt{\frac{2r}{r-1}}=\frac{\sqrt{r}}{\sqrt{1-\lambda}}$: 
\begin{align}
\frac{d}{dt}\lVert u\rVert_{L^2}^2&\leq C_2\lVert u\rVert_{L^2}^{2\lambda} \left( \lVert u_1\rVert_{\dot{H}^{\frac{d}{2}}}+\lVert u_2\rVert_{\dot{H}^{\frac{d}{2}}}\right)^{r-1} \lVert G(u_1,u_2)\rVert_{L^{\frac{2r}{r-1}}} 
\end{align}
Therefore we have for all $p\geq 1$, the Yudovich type estimate for abstract PDE: 
\begin{align}\label{Yudowich type estimate}
\frac{d}{dt}\lVert u\rVert_{L^2}^2&\leq C_3\frac{1}{\sqrt{1-\lambda}}\lVert u\rVert_{L^2}^{2\lambda} \left( 1+\lVert u_1\rVert_{\dot{H}^{\frac{d}{2}}}^p+\lVert u_2\rVert_{\dot{H}^{\frac{d}{2}}}^p\right) \lVert G(u_1,u_2)\rVert_{\dot{H}^\frac{d}{2}}\quad \quad\text{with} ~~\lambda\xrightarrow[<]{}{1}.
\end{align}

Thus
$$
    \frac{1}{1-\lambda} \lVert u(t) \rVert_{L^2}^{2(1-\lambda)}\leq C_3\frac{1}{\sqrt{1-\lambda}}  \int_0^t \left(1+\lVert u_1\rVert_{\dot{H}^\frac{d}{2}}^p+\lVert u_2\rVert_{\dot{H}^\frac{d}{2}}^p\right) \lVert G(u_1,u_2) \rVert_{{\dot{H}^\frac{d}{2}}} d\tau.
$$

\begin{align}\label{estimate}
     \lVert u(t) \rVert_{L^2}^{2} &\leq \left(C_3\sqrt{1-\lambda}  \int_0^t \left(1+\lVert u_1\rVert_{\dot{H}^\frac{d}{2}}^p+\lVert u_2\rVert_{\dot{H}^\frac{d}{2}}^p \right) \lVert G(u_1,u_2) \rVert_{\dot{H}^\frac{d}{2}}  d\tau \right)^{\frac{1}{1-\lambda}}.
\end{align}

     According the estimate \eqref{Yudovich estimate}, since there exist $p\geq 1$, such that $$\forall t_0>0, \int_0^{t_0}\left(1+\lVert u_1(\tau)\rVert_{\dot{H}^{\frac{d}{2}}}^p+\lVert u_2(\tau)\rVert_{\dot{H}^\frac{d}{2}}^p\right)\lVert G(u_1(\tau),u_2(\tau))\rVert_{\dot{H}^\frac{d}{2}}d\tau<\infty,$$
     Therefore $\forall t_0>0$, one has $C_3\sqrt{1-\lambda}\int_0^{t_0}\left(1+\lVert u_1(\tau)\rVert_{\dot{H}^{\frac{d}{2}}}^p+\lVert u_2(\tau)\rVert_{\dot{H}^\frac{d}{2}}^p\right)\lVert G(u_1(\tau),u_2(\tau))\rVert_{\dot{H}^\frac{d}{2}}d\tau<1$ for $\lambda<1$ sufficently close to 1. By letting now $r\to 1$(or $\lambda \to 1$) in \eqref{estimate}, we arrive at $\lVert u(t)\rVert_{L^2}=0 ~~\forall t\leq t_0$.\\
Since $t_0$ is arbitrary, $\lVert u(t)\rVert_{L^2}=0 ~~\forall t\geq 0$ which gives the uniqueness of the solution.
\end{proof}

\begin{remark}
 Let $X$ be a Banach (or Hilbert) space and $u_1, u_2 ~:\mathbb{R^+}\to X$ be two solutions of a PDE with the initial condition $u_1(0)=u_2(0)=u_0$. Let $y(t)=\lVert u_1(t)-u_2(t)\rVert_X$, $f(t)=:F(t,u_1(t),u_2(t))$ with $F:\mathbb{R^+}\times X\times X \to \mathbb{R^+}$ and $g:\mathbb{R^+}\to \mathbb{R^+}$ be a continuous function that is strictly positive except at $0$\\ where $g(0)=0$. Let $h$ be any function. The summary table showing the different scenarios and the appropriate lemmas to demonstrate the uniqueness.
\begin{align}\label{tableau}
\begin{array}{|c||c|c|c|}
\hline
\text{\bfseries } & \text{Form of the estimate} & \text{Condition} & \text{Conclusion} \\
\hline
\text{Gronwall's lemma} & y'(t) \le f(t) y(t) & f\in L^1_{\text{loc},t} & y(t)=0 \\
\hline
\text{Osgood's lemma} & y'(t) \le g(y(t)) & \exists t_0>0,~\int_0^{t_0} \frac{1}{g(s)} ds = +\infty & y(t)=0 \\
\hline
\text{Yudovich's argument} & y'(t) \le h(\lambda) f(t) y(t)^\lambda & \lambda\xrightarrow[<]{}{1},~h(\lambda)(1-\lambda) \to 0,~f\in L^1_{\text{loc},t} & y(t)=0 \\
\hline
\end{array}
\end{align}
\end{remark}

\section{Stochastic dynamic and Probabilistic estimates}
Let us start by introducing some settings and notation that will be used throughout this section.
We denote by $(\Omega, \mathcal{F},\mathbb{P})$ a complete probability space. For a Banach space $E$, we consider random variables $X:\Omega \to E$~~being Bochner measurable functions with respect to $\mathcal{F}$ and $\mathcal{B}(E)$, where $\mathcal{B}(E)$ is Borel $\sigma$-algebra of $E$.
 For any positive integer $N$, we define the $N$-dimension Brownian motion by 
 \begin{align*}
\mathcal{W}_N(t,x)=\sum_{n=1}^Na_ne_n(x)\mathcal{B}_n(t).
\end{align*}
where $\mathcal{B}_n(t)$ is a one-dimensional independent Brownian motion with filtration $(\mathcal{F}_t)_{t\geq 0}$.\\
 The numbers $(a_n)_{n\geq 1}$ are such that $A_0=\sum_{n\geq 1}a_n^2<\infty$. \\
 We denote by $A_N^s=\sum_{n=1}^N \lambda_n^sa_n^2$ and $A^s=\sum_{n\geq1} \lambda_n^sa_n^2$ where $(\lambda_n)_{n\geq 1}$ are the eigenvalues of $-\Delta_g$ such that $A^{1}<\infty$ and $(e_n)_{n\in \mathbb{N}}$ a normalized eigenfunctions basis of $-\Delta$, orthonormal in $L^2(M)$.\\
 Recall that $\langle,\rangle$ denotes the real inner product in $L^2(M;\C)$ define by $\langle u,v\rangle = \text{Re}\left( \int_M u(x) \overline{v(x)} \, dx \right).$ \\ We have the important property $(u,iu)=0$. We denote by $s^+=s+\varepsilon$ for some $\varepsilon>0$ enough to $0$(we use $s^-$ in similar way).
 Let $F: E \to \mathbb{R}, u\mapsto F (u) $ be a smooth function. We denote by $F'(u;v)$ and $F''(u;v,w)$, respectively, the first and second derivative of $F$ at $u$.\\ 
Let us consider the Galerkin projected NLS with Moser Trudinger nonlinearity 
\begin{equation}\label{Galerkin_Trudinger}
\left\{
\begin{aligned}
    \partial_t P_Nu&=i\left((\Delta-1) P_Nu-P_N\left(\left(e^{\beta\abs{P_Nu}^2}-1\right)P_Nu\right)\right),\\
    P_Nu(0) &= P_N u_0 \in E^N.
\end{aligned}
\right.    
\end{equation}
\\
In this section, our goal is to employ the fluctuation–dissipation method to construct global solutions to the Trudinger model via invariant measure arguments and the Skorokhod representation theorem.

The key and essential tool in this method lies in identifying a suitable and effective dissipation operator that can provide both global existence, uniqueness and continuity of the solution with respect to the initial data. In other words, the chosen dissipation must yield sufficient control over the solution in order to successfully implement the Yudovich-type argument to prove uniqueness and continuity of the solution with respect to the initial data.
\subsection{Analysis of Dissipation} We introduce the following dissipation function
\begin{align}\label{dissipation}
    \mathcal{L}(u)= Ce^{\gamma C_1\lVert P_N(e^{\beta\abs{u}^2}u)\rVert_{L^2}^2}\left(u+P_N(e^{\beta\abs{u}^2}u)\right)+C_2P_N(e^{\gamma\abs{u}^2}u)+P_N\left((1-\Delta)^{-\frac{1}{2}}\left(|(1-\Delta)^{\frac{1}{2}}u|^\delta(1-\Delta)^{\frac{1}{2}}u    \right)\right).
\end{align}

Here $C,C_1,C_2,\gamma $ are positive constants  that will be determined later and $\gamma$ will be taken so that uniqueness works; also $\delta \in(0,1)$ is a positive constant.
We denote by $\mathcal{M}(u)$ and $\mathcal{E}(u)$, the dissipation rate of the mass and energy under $\mathcal{L}(u)$.
Let us recall the mass and energy quantities:
\begin{align*}
M(u)&=\frac{1}{2}\lVert u\rVert_{L^2}^2,\\
E(u)&=\frac{1}{2}\lVert u\rVert_{L^2}^2+\frac{1}{2}\lVert \nabla u\rVert_{L^2}^2+\frac{1}{2\beta}\int_M\left(e^{\beta|u|^2}-1-\beta|u|^2\right)dx=\frac{1}{2}\lVert u\rVert_{L^2}^2+\frac{1}{2}\lVert \nabla u\rVert_{L^2}^2+\frac{1}{2\beta}\sum_{p\in \mathbb{N}^*}\frac{\beta^{p+1}}{(p+1)!}\lVert u\rVert_{L^{2p+2}}^{2p+2}.
\end{align*}
We can easily see that $M'(u,v)=(u,v)$ and $E'(u,v)=\left((1- \Delta) u +(e^{\beta\abs{u}^2}-1)u,v\right)$.\\
The dissipation rate of the mass is given by
\begin{align*}
\mathcal{M}(u)&:=\left(u ~,~ Ce^{\gamma C_1\lVert P_N(e^{\beta\abs{u}^2}u)\rVert_{L^2}^2}\left(u+P_N(e^{\beta\abs{u}^2}u)\right)+C_2P_N(e^{\gamma\abs{u}^2}u)+P_N\left((1-\Delta)^{-\frac{1}{2}}\left(|(1-\Delta)^{\frac{1}{2}}u|^\delta(1-\Delta)^{\frac{1}{2}}u    \right)\right)\right)\\
&={Ce^{\gamma C_1\lVert P_N(e^{\beta\abs{u}^2}u)\rVert_{L^2}^2}}\lVert u\rVert_{L^2}^2+{Ce^{\gamma C_1\lVert P_N(e^{\beta\abs{u}^2}u)\rVert_{L^2}^2}}\sum_{p\in\mathbb{N}}\frac{\beta^p}{p!}\lVert u\rVert_{L^{2p+2}}^{2p+2}\\
&\quad \quad
+C_2\sum_{p\in \mathbb{N}}\frac{\gamma^p}{p!}\lVert u\rVert_{L^{2p+2}}^{2p+2}+\left\langle u,P_N\left((1-\Delta)^{-\frac{1}{2}}\left(|(1-\Delta)^{\frac{1}{2}}u|^\delta(1-\Delta)^{\frac{1}{2}}u    \right)\right)\right\rangle.
\end{align*}

In the finite-dimensional step below, we will need the following technical inequalities
\begin{align}\label{3}
    \left|\left\langle u,P_N\left((1-\Delta)^{-\frac{1}{2}}\left(|(1-\Delta)^{\frac{1}{2}}u|^\delta(1-\Delta)^{\frac{1}{2}}u    \right)\right)\right\rangle\right| \leq \lVert u\rVert_{L^2}\lVert (1-\Delta)^\frac{1}{2}u\rVert_{L^2}^{1+\delta}\leq \frac{C}{2}\beta\lVert u\rVert_{L^4}^4+C_N.
\end{align}
Thus
\begin{align}\label{4}
    \mathcal{M}(u)
&\geq \frac{1}{2}\bigg[C{e^{\gamma C_1\lVert P_N(e^{\beta\abs{u}^2}u)\rVert_{L^2}^2}}\left(\lVert u\rVert_{L^2}^2+\sum_{p\in\mathbb{N}}\frac{\beta^p}{p!}\lVert u\rVert_{L^{2p+2}}^{2p+2}\right)
+C_2\sum_{p\in \mathbb{N}}\frac{\gamma^p}{p!}\lVert u\rVert_{L^{2p+2}}^{2p+2}\bigg]-C_N=\mathcal{M}_0(u)-C_N.
\end{align}
\begin{remark}
 The estimate  \eqref{3} could be misinterpreted, as one might naturally expect uniform controls with a view to passing to the infinite-dimensional limit. However, this is not the purpose of this inequality. Indeed, the estimate \eqref{3} is strictly employed at the finite-dimensional level, and only for the construction of the finite-dimensional invariant measure $\mu_N^\alpha$. This construction is based on the classical Bogolyubov–Krylov argument combined with Ito’s formula $$
\mathbb{E} M(u) + \alpha \int_0^t \mathbb{E}\,\mathcal{M}(u)\,d\tau
=
\mathbb{E} M(u_{0,N}) + \alpha \frac{A_{0,N}}{2}\, t.$$ The role of \eqref{3} is to guaranty that the dissipation functional $\mathcal{M}(u)$ provides sufficient control of the $L^2$–mass $\lVert u\rVert_{L^2}^2$ at the finite dimension level, as shown in estimate \eqref{4}. Below, we will establish uniform estimates on relevant quantities that will be used in the passage to infinite-dimensional limit.
\end{remark}
Let us now turn to the dissipation rate of the energy. We have {by using Green's formula and the homogeneous Dirichlet boundary condition for the case with boundary}: 
\begin{align*}
    \mathcal{E}&(u)=\bigg\langle (1- \Delta) u +(e^{\beta\abs{u}^2}-1)u~; \quad {Ce^{\gamma C_1\lVert P_N(e^{\beta\abs{u}^2}u)\rVert_{L^2}^2}}\left(u+P_N(e^{\beta\abs{u}^2}u)\right)+C_2P_N(e^{\gamma\abs{u}^2}u)\\
    &\quad \quad \quad \quad \quad \quad \quad \quad \quad \quad \quad \quad \quad\quad+P_N\left((1-\Delta)^{-\frac{1}{2}}\left(|(1-\Delta)^{\frac{1}{2}}u|^\delta(1-\Delta)^{\frac{1}{2}}u    \right)\right)   \bigg\rangle  \\
    &= {Ce^{\gamma C_1\lVert P_N(e^{\beta\abs{u}^2}u)\rVert_{L^2}^2}}\lVert u\rVert_{H^1}^2+{Ce^{\gamma C_1\lVert P_N(e^{\beta\abs{u}^2}u)\rVert_{L^2}^2}}\sum_{p\in\mathbb{N}}\frac{\beta^p}{p!}\langle -\Delta u,\abs{u}^{2p}u\rangle+C_2\sum_{p\in \mathbb{N}}\frac{\gamma^p}{p!}\langle -\Delta u,\abs{u}^{2p}u\rangle\\
    &\quad+\lVert u\rVert_{W^{1,2+\delta}}^{2+\delta}+2{Ce^{\gamma C_1\lVert P_N(e^{\beta\abs{u}^2}u)\rVert_{L^2}^2}}\sum_{p\in \mathbb{N}}\frac{\beta^p}{p!}\lVert u\rVert_{L^{2p+2}}^{2p+2}+C_2\sum_{p\in \mathbb{N}}\frac{\gamma^p}{p!}\lVert u\rVert_{L^{2p+2}}^{2p+2}+{Ce^{\gamma C_1\lVert P_N(e^{\beta\abs{u}^2}u)\rVert_{L^2}^2}}{\lVert P_N(e^{\beta\abs{u}^2}u)\rVert_{L^2}^2}\\
    &+C_2\langle P_N(e^{\beta\abs{u}^2}u),P_N(e^{\gamma\abs{u}^2}u)\rangle
     + \left\langle P_N(e^{\beta|u|^2}u), P_N\left((1-\Delta)^{-\frac{1}{2}}\left(|(1-\Delta)^{\frac{1}{2}}u|^\delta(1-\Delta)^{\frac{1}{2}}u    \right)\right) \right\rangle -\mathcal{M}(u).
\end{align*}
Let us remark that
\begin{align*}
\langle -\Delta u, |u|^{2p}u\rangle&=\langle |\nabla u|^2,|u|^{2p}\rangle+\langle u \nabla ({\bar u}^p\, u^p), \nabla u \rangle\\
&= \langle |\nabla u|^2,|u|^{2p}\rangle+p \langle u\, u^{p-1} (\nabla u)\, \bar u^p, \nabla u \rangle
+ p \langle u\, \bar u^{p-1} (\nabla \bar u)\, u^p, \nabla u \rangle\\
&\ge \langle |\nabla u|^2,|u|^{2p}\rangle+ p \langle |u|^{2p}, |\nabla u|^2 \rangle
- p \langle |u|^{2p}, |\nabla u|^2 \rangle =\langle |\nabla u|^2,|u|^{2p}\rangle.
\end{align*}
Thus 
\begin{align*}
    \mathcal{E}&(u)\geq {Ce^{\gamma C_1\lVert P_N(e^{\beta\abs{u}^2}u)\rVert_{L^2}^2}}\lVert u\rVert_{H^1}^2+{Ce^{\gamma C_1\lVert P_N(e^{\beta\abs{u}^2}u)\rVert_{L^2}^2}}\sum_{p\in\mathbb{N}}\frac{\beta^p}{p!}\langle \abs{\nabla u}^2,\abs{u}^{2p}\rangle+C_2\sum_{p\in \mathbb{N}}\frac{\gamma^p}{p!}\langle\abs{\nabla u}^2,\abs{u}^{2p}\rangle\\
    &+\lVert u\rVert_{W^{1,2+\delta}}^{2+\delta} +{2Ce^{\gamma C_1\lVert P_N(e^{\beta\abs{u}^2}u)\rVert_{L^2}^2}}\sum_{p\in \mathbb{N}}\frac{\beta^p}{p!}\lVert u\rVert_{L^{2p+2}}^{2p+2}+C_2\sum_{p\in \mathbb{N}}\frac{\gamma^p}{p!}\lVert u\rVert_{L^{2p+2}}^{2p+2}+{Ce^{\gamma C_1\lVert P_N(e^{\beta\abs{u}^2}u)\rVert_{L^2}^2}}{\lVert P_N(e^{\beta\abs{u}^2}u)\rVert_{L^2}^2}\\
    &+C_2\langle P_N(e^{\beta\abs{u}^2}u),P_N(e^{\gamma\abs{u}^2}u)\rangle
    + \left\langle P_N(e^{\beta|u|^2}u), P_N\left((1-\Delta)^{-\frac{1}{2}}\left(|(1-\Delta)^{\frac{1}{2}}u|^\delta(1-\Delta)^{\frac{1}{2}}u    \right)\right) \right\rangle -\mathcal{M}(u).
\end{align*}
Let us observe that $\forall \lambda>0,~~ \sum_{p\in\mathbb{N}}\frac{\lambda^p}{p!}\lVert \nabla u^{p+1}\rVert_{L^2}^2=\sum_{p\in \mathbb{N}}\frac{\lambda^p}{p!}\langle\abs{\nabla u}^2,\abs{u}^{2p}\rangle$.\\
In order to establish a lower bound for 
$\mathcal{E}$, we want to absorb the term $\langle P_N(e^{\beta\abs{u}^2}u),P_N(e^{\gamma\abs{u}^2}u)\rangle$ by $e^{\gamma C_1\lVert P_N(e^{\beta|u|^2}u)\rVert_{L^2}^2}\lVert P_N(e^{\beta|u|^2}u)\rVert_{L^2}^2$ and $\sum_{p\in \mathbb{N}}\frac{\gamma^p}{p!}\lVert \nabla u^{p+1}\rVert_{L^2}^2$. To achieve this, we write the following estimates by introducing a new constant $C_3$ that will be determined later:
\begin{align}
    \langle P_N(e^{\beta\abs{u}^2}u),P_N(e^{\gamma\abs{u}^2}u)\rangle &=\sum_{p\in N}\frac{\gamma^p}{p!}\left\langle |u|^{2p}u, P_N(e^{\beta|u|^2}u)\right\rangle \label{Non linearity dissipation}\\
    &\leq \sum_{p\in N}\frac{\gamma^p}{p!}2(C_3)^{\frac{2p+1}{2p+2}}\lVert P_N(e^{\beta|u|^2}u)\rVert_{L^2}\frac{1}{2}(C_3)^{-\frac{2p+1}{2p+2}}\lVert u\rVert_{L^{4p+2}}^{2p+1}\nonumber. 
\end{align}  
By applying the Young's inequality with $1=\frac{1}{2p+2}+\frac{2p+1}{2p+2}$, we obtain:
\begin{align*}
     \langle P_N(e^{\beta\abs{u}^2}u),P_N(e^{\gamma\abs{u}^2}u)\rangle &\leq \sum_{p\in \mathbb{N}}\frac{\gamma^p}{p!}\left(\frac{2^{2p+2}}{2p+2}(C_3)^{2p+1}\lVert P_N(e^{\beta|u|^2}u)\rVert_{L^2}^{2p+2}+\frac{1}{2^{\frac{2p+2}{2p+1}}}\frac{1}{C_3}\lVert u\rVert_{L^{4p+2}}^{2p+2}\right) \\
     &\leq 2C_3e^{\gamma (4C_3^2)\lVert P_N(e^{\beta|u|^2}u)\rVert_{L^2}^2}\lVert P_N(e^{\beta|u|^2}u)\rVert_{L^2}^2+\frac{1}{2}  \sum_{p\in \mathbb{N}} \frac{\gamma^p}{p!}\frac{1}{C_3}\lVert u\rVert_{L^{4p+2}}^{2p+2}.
\end{align*}
We just need to show that $\frac{1}{C_3}\lVert u\rVert_{L^{4p+2}}^{2p+2}\leq \lVert \nabla u^{p+1}\rVert_{L^2}^2$
to have the claim.\\

Let $q\geq 4$ and let us set $\alpha(p)=\frac{2p+2}{4p+2}-\frac{2p+2}{(p+1)q}$, we have by applying the Hölder's inequality and Sobolev embedding $H^1\hookrightarrow L^q$ with a precise constant: $$\lVert u\rVert_{L^{4p+2}}^{2p+2}\leq |M|^{\alpha(p)}\lVert u\rVert_{L^{(p+1)q}}^{2p+2}=|M|^{\alpha(p)}\lVert u^{p+1}\rVert_{L^q}^2\leq |M|^{\alpha(p)} C'^2q\lVert \nabla u^{p+1}\rVert_{L^2}^2$$
Let us set
$C_M=\sup_{p\geq 0}|M|^{\alpha(p)}$, Since $\alpha(0)=1-\frac{2}{q}$, $ \lim_{p\to \infty}\alpha(p)=\frac{1}{2}-\frac{2}{q}$ and $\alpha(p)\geq 0, \alpha'(p)<0$; It is easy to see that $C_M=\max(1, |M|^{1-\frac{2}{q}})$ which does not depend on $p$.
Therefore by taking\\ $    C_MC'^2q\leq C_3\leq \frac{C}{8C_2}$ where $C$ is large enough and  $C_1=4C_3^2$; we have $\frac{1}{C_3}\lVert u\rVert_{L^{4p+2}}^{2p+2}\leq \lVert \nabla u^{p+1}\rVert_{L^2}^2$ and obtain the following: 

\begin{align*}
    \langle P_N(e^{\beta\abs{u}^2}u),P_N(e^{\gamma\abs{u}^2}u)\rangle 
    &\leq \frac{C}{4C_2}e^{\gamma C_1\lVert P_N(e^{\beta|u|^2}u)\rVert_{L^2}^2}\lVert P_N(e^{\beta|u|^2}u)\rVert_{L^2}^2+\frac{1}{2}\sum_{p\in \mathbb{N}}\frac{\gamma^p}{p!}\lVert \nabla u^{p+1}\rVert_{L^2}^2.
\end{align*}
Also, 
\begin{align*}
  \bigg| \bigg\langle P_N(e^{\beta|u|^2}u), P_N\left((1-\Delta)^{-\frac{1}{2}}\left(|(1-\Delta)^{\frac{1}{2}}u|^\delta(1-\Delta)^{\frac{1}{2}}u    \right)\right) \bigg\rangle \bigg|
  &\leq \bigg\lVert P_N(e^{\beta|u|^2}u)\bigg\rVert_{L^2}\bigg\lVert |(1-\Delta)^{\frac{1}{2}}u|^\delta(1-\Delta)^{\frac{1}{2}}u \bigg\rVert_{H^{-1}}.
\end{align*}
Since we are in two-space dimension, $L^{1+\frac{1-\delta}{1+\delta}}\hookrightarrow H^{-1}$, we have by using Holder and Young's inequalities:
\begin{align*}
\bigg| \bigg\langle P_N(e^{\beta|u|^2}u), P_N\left((1-\Delta)^{-\frac{1}{2}}\left(|(1-\Delta)^{\frac{1}{2}}u|^\delta(1-\Delta)^{\frac{1}{2}}u    \right)\right) \bigg\rangle \bigg|&\leq\lVert P_N(e^{\beta|u|^2}u)\rVert_{L^2}\lVert u\rVert_{H^1}^{1+\delta}
\\
&\leq \frac{C}{2}e^{\gamma C_1\lVert P_N(e^{\beta|u|^2}u)\rVert_{L^2}^2}\lVert u\rVert_{H^1}^2+\tilde{C}.
\end{align*}

Therefore, we obtain 

\begin{align*}
\mathcal{E}(u) \geq \frac{1}{2} \bigg[  
& Ce^{\gamma C_1\lVert P_N(e^{\beta |u|^2} u) \rVert_{L^2}^2}\bigg(\lVert u \rVert_{H^1}^2 
  +  
\sum_{p \in \mathbb{N}} \frac{\beta^p}{p!} \lVert \nabla u^{p+1} \rVert_{L^2}^2  + \sum_{p \in \mathbb{N}} \frac{\beta^p}{p!} \lVert  u \rVert_{L^{2p+2}}^{2p+2} 
+  \lVert P_N(e^{\beta |u|^2} u) \rVert_{L^2}^2 \bigg)\\
&+ C_2\sum_{p \in \mathbb{N}} \frac{\gamma^p}{p!} \lVert u^{p+1} \rVert_{H^1}^2 +\lVert u\rVert_{W^{1,2+\delta}}^{2+\delta}
\bigg] -\mathcal{M}(u)-\tilde{C}=\mathcal{E}_0(u)-\mathcal{M}(u)-\tilde{C}.
\end{align*}
\subsection{Stochastic Global well-posedness}
Let us now consider the following-fluctuation dissipation model.
\begin{equation}\label{eq3}
\left\{
\begin{aligned}
du &= \bigg[i \left((\Delta-1) u - P_N ((e^{\beta\abs{u}^2}-1)u)\right)-\alpha \left( Ce^{\gamma C_1\lVert P_N(e^{\beta\abs{u}^2}u)\rVert_{L^2}^2}\left(u+P_N(e^{\beta\abs{u}^2}u)\right)+C_2P_N(e^{\gamma\abs{u}^2}u)\right)\\
&\quad\quad\quad \quad\quad\quad-\alpha P_N\left((1-\Delta)^{-\frac{1}{2}}\left(|(1-\Delta)^{\frac{1}{2}}u|^\delta(1-\Delta)^{\frac{1}{2}}u    \right)\right)\bigg]dt+\sqrt{\alpha} \sum_{n=1}^N a_ne_n(x)d\mathcal{B}_n(t),\\
u(0) &= P_N u_0:=u_{0,N} \in E_N.
\end{aligned}
\right.
 \end{equation}
 Let us start by defining what we mean by a stochastically globally well-posed problem.
  \begin{definition}(Stochastic global well-posedness).
  \leavevmode\\
  We will say that the problem(\ref{eq3}) is globally stochastically well posed on $E_N$if the following properties hold: \begin{itemize}
    \item For any $E_N$-valued random variable $u_{0,N}$ independent of $\mathcal{F}_t$, 
    \begin{itemize}
        \item for $\mathbb{P}$-a.e. $\omega$, there exists a solution $u = u^\omega \in C(\mathbb{R}^+; E_N)$ of \eqref{eq3} with initial datum $u_{0,N} = u_{0,N}^\omega$ in the integral sense, i.e.,
        \begin{align*}
            u(t) &= u_{0,N} + \int_0^t i \big((\Delta-1) u - P_N ((e^{\beta\abs{u}^2}-1)u) \big)d\tau\\
            &\quad\quad- \alpha\int_0^t \left( Ce^{\gamma C_1\lVert P_N(e^{\beta\abs{u}^2}u)\rVert_{L^2}^2}\left(u+P_N(e^{\beta\abs{u}^2}u)\right)+C_2P_N(e^{\gamma\abs{u}^2}u)\right) \, d\tau\\ &\quad \quad  - \alpha\int_0^t P_N\left((1-\Delta)^{-\frac{1}{2}}\left(|(1-\Delta)^{\frac{1}{2}}u|^\delta(1-\Delta)^{\frac{1}{2}}u    \right)\right)d\tau + \sqrt{\alpha} \mathcal{W}_N(t).
        \end{align*}
        with equality as elements of $C(\mathbb{R}^+; E_N)$.
    \end{itemize}
    \item For any $u_1^\omega, u_2^\omega \in C(\mathbb{R}^+; E_N)$ two solutions with respective initial data $u_{1,0,N}^\omega, u_{2,0,N}^\omega$, then\\ \[\lim_{u_{1,0,N}\to u_{2,0,N}} u_1^\omega(\cdot,u_{1,0,N}) = u_2^\omega(\cdot,u_{2,0,N}) ~\text{and if}~ u_{1,0,N}^{\omega}=u_{2,0,N}^{\omega} ~\text{then}~ u_1^\omega \equiv u_2^\omega.\]
    \item The process $(\omega, t) \mapsto u^\omega(t)$ is adapted to the filtration $\sigma(u_{0,N}, \mathcal{F}_t)$.
    \end{itemize}
    \end{definition}
    To solve this stochastic equation, we use a Da-Prato-Debussche's trick which consists of splitting the solution as follows $u=v+z$, where 
    \begin{equation} \label{linear-stoch-eq}
\left\{
\begin{aligned}
dz &= i~\Delta z~dt+\sqrt{\alpha}\sum_{n=1}^N a_ne_n(x)d\mathcal{B}_n(t),\\
z(0)&=0.
\end{aligned}
\right.
\end{equation}

\begin{equation}
\left\{
\begin{aligned}\label{eq in v}
    \partial_t v &=F(v)= i \left( \Delta v - P_N\left( e^{\beta |v + z|^2} (v + z) \right) \right)- \alpha P_N\left((1-\Delta)^{-\frac{1}{2}}\left(|(1-\Delta)^{\frac{1}{2}}(v+z)|^\delta(1-\Delta)^{\frac{1}{2}}(v+z)    \right)\right) \\
    &\quad\quad \quad\quad - \alpha \left( Ce^{\gamma C_1\lVert P_N(e^{\beta\abs{v+z}^2}(v+z))\rVert_{L^2}^2}\left((v+z)+P_N(e^{\beta\abs{v+z}^2}(v+z))\right) 
     +C_2P_N(e^{\gamma\abs{v+z}^2}(v+z))\right), \\
    v(0) &= P_N u_0.
\end{aligned}
\right.
\end{equation}

The equation(\ref{linear-stoch-eq}) is a linear stochastic equation whose unique solution is given by the stochastic convolution 
\begin{align}
z_{\alpha}(t)=\sqrt{\alpha}\int_0^te^{(t-\tau)i\Delta}d\mathcal{W}_N(\tau).
\end{align}
A basic application of Ito's formula and Doob's inequality gives us that
 for $\mathbb{P}-$almost sure $w\in \Omega$,~~$\forall T>0$,
\begin{align}\label{estimate on z}
\sup_{t\in[0,T]}\lVert z^w_{\alpha}(t)\rVert_{L^2}^2\leq C_{\alpha}(w,T).
\end{align}
Hence, for $\mathbb{P}-$ almost all $w\in \Omega, z_\alpha^w\in C_t(\mathbb{R}_t,C^\infty(E_N)).$\\
 Let us now consider the second problem \eqref{eq in v}.
Since we are in finite dimension ,it is easy to see that $F\in C^{\infty}(E_N,E_N)$, then thanks to the Cauchy Lipschitz theorem, the problem has a local-time smooth solution.
Now, the next task is to show that our solution $v$ is global in time $\mathbb{P}$-almost surely.
\begin{proposition}\label{Gwp of v}
    The local solution $v$ constructed above exists globally in time, $\mathbb{P}-$almost surely.
\end{proposition}
\begin{proof}
    By taking the inner product with $v$, we have :
    \begin{align*}
        \frac{1}{2}\frac{d}{dt}\lVert v\rVert^2_{L^2}&= \langle v,i\Delta v\rangle-\langle v,i(v+z)e^{\beta\abs{v+z}^2}\rangle-\alpha~ Ce^{\gamma C_1\lVert P_N(e^{\beta\abs{v+z}^2}(v+z))\rVert_{L^2}^2}\langle v,v+z\rangle\\&\quad-\alpha ~Ce^{\gamma C_1\lVert P_N(e^{\beta\abs{v+z}^2}(v+z))\rVert_{L^2}^2}\langle v,(v+z)e^{\beta\abs{v+z}^2}\rangle-\alpha C_2\langle v,(v+z)e^{\gamma\abs{v+z}^2}\rangle 
        \\&\quad-\alpha\left\langle v,P_N\left((1-\Delta)^{-\frac{1}{2}}\left(|(1-\Delta)^{\frac{1}{2}}(v+z)|^\delta(1-\Delta)^{\frac{1}{2}}(v+z)    \right)\right) \right\rangle.
        \end{align*}
        By using the fact that $\langle v, i\Delta v\rangle=0$ and $\langle (v+z), i ~(v+z)e^{\beta\abs{v+z}^2}\rangle=0$, we can substitute $-z$ in the second term on the right and by taking $\gamma >\beta$; we have 
    \begin{align*}
        \frac{1}{2}\frac{d}{dt}\lVert v\rVert^2_{L^2}
        &\leq \frac{\alpha}{2}Ce^{\gamma C_1\lVert P_N(e^{\beta\abs{v+z}^2}(v+z))\rVert_{L^2}^2}\left(\lVert z\rVert_{L^2}^2-\lVert v\rVert_{L^2}^2\right)\\&\quad+ (1+\alpha C_2)\int_M|z||v+z|e^{\gamma\abs{v+z}^2}
        -\alpha C_2\int_M|v+z|^2e^{\gamma\abs{v+z}^2}\\
        &\quad
         +\alpha ~Ce^{\gamma C_1\lVert P_N(e^{\beta\abs{v+z}^2}(v+z))\rVert_{L^2}^2}\left( 
    {\int_M |z| |v+z| e^{\beta \abs{v+z}^2} - \frac{1}{2}\int_M |v+z|^2 e^{\beta \abs{v+z}^2}} 
\right)\\
&\quad-\frac{\alpha}{2}Ce^{\gamma C_1\lVert P_N(e^{\beta|v+z|^2}(v+z))\rVert_{L^2}^2}\sum_{p\in \mathbb{N}}\frac{\beta^p}{p!}\lVert v+z\rVert_{L^{2p+2}}^{2p+2}+\alpha C_N\lVert v\rVert_{L^2}^2+\alpha\lVert v+z\rVert_{L^{2+2\delta}}^{2+2\delta}.
    \end{align*}
    
    In order to absorb the term ${\int_M |z| |v+z| e^{\beta \abs{v+z}^2} ~~\text{by}~~ \int_M |v+z|^2 e^{\beta \abs{v+z}^2}} $, we use the generalized Young inequality. To achieve this, we consider the function  $f_\lambda:
        \mathbb{R}^+\to \mathbb{R}^+$ such that $f_\lambda(x)=xe^{\lambda{x}^2}$, $\lambda>0$. The function $f_\lambda$ is continuous and strictly increasing, therefore  it defines a bijection. We want to define a strictly convex application $\Phi:\mathbb{R}^+\to\mathbb{R}^+$ such that $\Phi^\lambda(f_\lambda(x))=x^2e^{\lambda x^2}$. \\ It is easy to see that there exists a unique $C^1$ and strictly convex application $\Phi^\lambda:\mathbb{R}^+\to \mathbb{R}^+$ solution of the functional equation $\Phi^\lambda(f_\lambda({x}))=x^2e^{\lambda{x}^2}$; In fact, since $f_\lambda$ defines a bijection, by defining the application $\Phi^\lambda$ such that
        $\Phi^\lambda(x)=xf_\lambda^{-1}(x)$, It is easy to see that $\Phi^\lambda(f_\lambda({x}))=x^2e^{\lambda{x}^2}$ and a basis computation of $(\Phi^\lambda)^{''}$ show that $(\Phi^\lambda)^{''}>0$ which ensure the strict convexity. \\
         Let us set $\Phi_{\alpha C_2}^\lambda=\alpha C_2\Phi^\lambda$ and $\Phi_{\alpha C_2}^{\lambda^*}$ the Legendre transformation of $\Phi_{\alpha C_2}^\lambda$. By using the fact that $\alpha\in(0,1)$ and the generalized Young inequality $(xy \leq \Phi_{\alpha C_2}^\lambda(x) + \Phi_{\alpha C_2}^{\lambda^*}(y))$, we will then have:
        \begin{align*}
\int_M ( 1+ \alpha C_2) |z| \cdot |v + z| e^{\gamma |v + z|^2} &\leq \int_M  (1+C_2)|z| \cdot |v + z| e^{\gamma |v + z|^2}\\ 
&\leq \int_M \Phi_{\alpha C_2}^{\gamma^*} \left( (1+C_2) |z| \right) + \int_M \Phi_{\alpha C_2}^\gamma \left( |v + z| e^{\gamma |v + z|^2} \right)\\
&= \int_M \Phi_{\alpha C_2}^{\gamma^{*}} \left( (1+C_2) |z| \right) + \alpha C_2\int_M |v + z|^2 e^{\gamma |v + z|^2}.
\end{align*}
Now, since we have $\sup_{t\in[0,T]}\lVert z^w_{\alpha}(t)\rVert_{L^2}^2\leq C_{\alpha}(w,T)$ for $\mathbb{P}-$ almost sure $w\in \Omega$ and $\forall T>0$, using the fact that we are in finite dimensionality, we obtain \[
\| z^w(t) \|_{L^\infty(M)} \leq C_N \| z^w(t) \|_{L^2} \leq C_{N, \omega, T,\alpha}. 
\]
Therefore we have using the growth of $\Phi^\lambda$ and the fact $$-\frac{\alpha}{2}Ce^{\gamma C_1\lVert P_N(e^{\beta|v+z|^2}(v+z))\rVert_{L^2}^2}\sum_{p\in \mathbb{N}}\frac{\beta^p}{p!}\lVert v+z\rVert_{L^{2p+2}}^{2p+2}+\alpha\lVert v+z\rVert_{L^{2+2\delta}}^{2+2\delta}\leq 0.$$
\begin{align*}
\frac{1}{2}\frac{d}{dt}  \| v \|^2 &\leq \quad\frac{\alpha}{2}Ce^{\gamma C_1\lVert P_N(e^{\beta\abs{v+z}^2}(v+z))\rVert_{L^2}^2} \left( \| z \|_{L^2}^2 - \| v \|_{L^2}^2\right)\\
&\quad \quad+\int_M \Phi^{\gamma*}_{\alpha C_2}((1+C_2) |z|) 
+\frac{\alpha}{2}~ Ce^{\gamma C_1\lVert P_N(e^{\beta\abs{v+z}^2}(v+z))\rVert_{L^2}^2}\int_M \Phi_1^{\beta*}(2 |z|)dx+\alpha C_N\lVert v\rVert_{L^2}^2\\
& \leq \frac{\alpha}{2}Ce^{\gamma C_1\lVert P_N(e^{\beta\abs{v+z}^2}(v+z))\rVert_{L^2}^2} \left( \underbrace{ \text{Vol}(M) \Phi_1^{\beta*}(2C_{N, \omega, T,\alpha})}_{{C'_{N, \omega, T,\alpha}}}+ \| z \|_{L^2}^2 - \| v \|_{L^2}^2\right)\\
& \quad+\text{Vol}(M) \Phi^{\gamma*}_{\alpha C_2}((1+C_2)C_{N, \omega, T,\alpha})+\alpha C_N\lVert v\rVert_{L^2}^2 .
\end{align*}
We then have
\begin{itemize}
    \item either $\lVert v^w(t)\rVert_{L^2}^2\leq \lVert z^w(t)\rVert_{L^2}^2+C'_{N, \omega, T,\alpha},$\\
    
    \item Or $\lVert v^w(t)\rVert_{L^2}^2> \lVert z^w(t)\rVert_{L^2}^2+C'_{N, \omega, T,\alpha}.$
\end{itemize}
For the first case, using the estimate on the boundedness of $z$, we  have the claim.\\
For the second case, we have
\begin{align*}
     \frac{\alpha}{2}Ce^{\gamma C_1\lVert P_N(e^{\beta\abs{v+z}^2}(v+z))\rVert_{L^2}^2} \left( \underbrace{ \text{Vol}(M) \Phi_1^{\beta*}(2C_{N, \omega, T,\alpha})}_{{C'_{N, \omega, T,\alpha}}}+ \| z \|_{L^2}^2 - \| v \|_{L^2}^2\right)<0.
\end{align*}

Thus, we obtain 
\begin{align}\label{estimate on v}
\frac{d}{dt}\lVert v\rVert_{L^2}^2\leq  2\text{Vol}(M) \Phi^{\gamma*}_{\alpha C_2}(C_{N, \omega, T,\alpha})+2\alpha C_N\lVert v\rVert_{L^2}^2.
\end{align}
Then by applying the Gronwall's lemma, we obtain
\[
\|v(t)\|_{L^2}^2 
\le \|v(0)\|_{L^2}^2 e^{2\alpha C_Nt}
+ \frac{2\text{Vol}(M) \Phi^{\gamma*}_{\alpha C_2}(C_{N, \omega, T,\alpha})}{2\alpha C_N}\left(e^{2\alpha C_Nt}-1\right).
\]

Hence, \(\|v(t)\|_{L^2}\) remains finite for every finite time,
which concludes the proof.
\end{proof}
Therefore, we have the existence of the global solution $u_{\alpha}\in C_t(\mathbb{R}^+;E_N)$ of the problem (\ref{eq3}).
\subsection*{Uniqueness and continuity} By using the fact that $F\in C^{\infty}(E_N,E_N)$ and the mean value theorem, it is easy to prove the uniqueness and continuity with respect to the initial data. 
\subsection*{Adaptability} We have that the solution $z_{\alpha}(t)$ is adapted to $\mathcal{F}_t$ and since $v$ is constructed by using the fixed point theorem, then $v$ is also adapted, therefore $u$ is adapted to $\sigma(u_{0,N},\mathcal{F}_t)$.\\
Let us denote by $u_{\alpha}(t,u_{0,N})$ the unique stochastic solution to (\ref{eq3}).

\subsection{The Markov semi-group and stationnary measure}
Let us define the transition probability $P_t^{\alpha,N}(u_0,\Gamma)=\mathbb{P}(u_{\alpha}(t,P_Nu_0)\in \Gamma)$ with $u_0\in L^2$ and $\Gamma \in \text{Bor}(L^2)$ .\\
The Markov semi groups are given by:
\begin{align*}
    \mathcal{B}_t^{\alpha,N} f(u_0) &= \int_{L^2} f(v) P_t^{\alpha,N}(u_0, dv) \quad C_b(L^2)\to C_b(L^2);\quad 
    \mathcal{B}_t^{\alpha,N *} \lambda(\Gamma) = \int_{L^2} \lambda(du_0) P_t^{\alpha,N}(u_0, \Gamma) \quad P(L^2) \to P(L^2).
\end{align*}
From continuity of the solution with respect to initial data, Feller property holds: $
\mathcal{B}_{t}^{\alpha,N} C_b(L^2) \subset C_b(L^2).$

Let us define the associated global flow by
$
\phi_N^t: 
P_Nu_0 \to \phi_N^t P_Nu_0 \quad E_N\to E_N
,$
where \(\phi_N^t(P_Nu_0) := u(t,u_{0,N})\) represents the solution to (\ref{Galerkin_Trudinger}) starting at \(P_Nu_0\) and  set the corresponding Markov groups:       
$\phi_N^t f(v) = f(\phi_N^t(P_N v)) \quad C_b(L^2) \to C_b(L^2);\quad
\phi_N^{t*} \lambda(\Gamma) = \lambda(\phi_N^{-t}(\Gamma)) \quad P(L^2) \to P(L^2).
$
\begin{proposition}
    Let $u_{0,N}$ be a random variable in $E_N$ independent of $\mathcal{F}_t$ such that $\mathbb{E}M(u_{0,N})<\infty$ and $\mathbb{E}E(u_{0,N})<\infty$. Let u be the solution to (\ref{eq3}) starting at $u_{0,N}$. Then we have
    \begin{align}\label{estimate1}
      \mathbb{E}M(u)+\alpha \int_0^t\mathbb{E}\mathcal{M}(u)d\tau &= \mathbb{E}M(u_{0,N})+\alpha\frac{A^0_N}{2}t,\\ \label{estimate2}
      \mathbb{E}E(u)+\alpha \int_0^t\mathbb{E}\mathcal{E}(u)d\tau &\leq \mathbb{E}E(u_{0,N})+\frac{\alpha}{2}A^0_Nt+\frac{\alpha}{2}A^1_Nt+\frac{\alpha}{2}C_M(2\beta+1)A_N^{\frac{1}{2}}\int_0^t\mathbb{E}\left(\int_M(1+|u|^2)e^{\beta|u|^2}dx\right).
    \end{align}
\end{proposition}
\begin{proof}
    Since $u\in C_t(\mathbb{R}^+;E_N)$ , then by applying the Ito's formula we have for $F(u)=M(u)$ or $E(u)$:
    
$$dF(u) = F'(u; du) +\frac{1}{2}\sum_{|n| \leq N} F''(u;\sqrt{\alpha}~ a_ne_n, \sqrt{\alpha}~a_ne_n) \, dt.$$ 
For $F=M(u)=\frac{1}{2}\lVert u\rVert_{L^2}^2$, since\\ $M'\left(u;i \big((\Delta-1) u - P_N ((e^{\beta\abs{u}^2}-1)u)\right)=\left\langle u,i \big((\Delta-1) u - P_N ((e^{\beta\abs{u}^2}-1)u)\right\rangle=0.$
\begin{align*}
\text{Thus}~~~M'(u; du) = -\alpha \mathcal{M}(u) \, dt + \sqrt{\alpha} \sum_{|n| \leq N} a_n(u, e_n) \, d\beta_n ~~\text{and}~~M''(u;\sqrt{\alpha}~a_ne_n,\sqrt{\alpha}~a_ne_n)=\alpha|a_n|^2.
\end{align*}
Therefore, integrating and taking the expectation while using the vanishing property of the martingale term, we obtain $$\mathbb{E}M(u)+\alpha \int_0^t\mathbb{E}\mathcal{M}(u)d\tau = \mathbb{E}M(u_{0,N})+\alpha\frac{A^0_N}{2}t.$$
For $F(u)=E(u)=\frac{1}{2}\lVert u\rVert_{L^2}^2+\frac{1}{2}\lVert \nabla u\rVert_{L^2}^2+\frac{1}{2\beta}\sum_{p\in \mathbb{N}^*}\frac{\beta^{p+1}}{(p+1)!}\lVert u\rVert_{L^{2p+2}}^{2p+2}$, since
\begin{align*}
E'\left(u~;i((\Delta-1) u - P_N ((e^{\beta\abs{u}^2}-1)u)\right)&=\left\langle (1-\Delta) u +(e^{\beta\abs{u}^2}-1)u~;~i\left((\Delta-1) u - P_N ((e^{\beta\abs{u}^2}-1)u)\right)\right\rangle \\
&=0
\end{align*}
we then have
\begin{align}\label{0}
    E'(u,du)&=-\alpha\mathcal{E}(u)dt+ \sqrt{\alpha} \sum_{|n| \leq N} a_n\left((1-\Delta) u +(e^{\beta\abs{u}^2}-1)u, e_n\right) \, d\beta_n \quad \text{and}\\ \label{1}
    E''(u;\sqrt{\alpha}~a_ne_n,\sqrt{\alpha}~a_ne_n)&\leq \alpha|a_n|^2\left( \lVert e_n\rVert^2_{L^2}+\lVert e_n\rVert^2_{\dot{H}^1}+\int_M(2\beta|u|^2+1)e^{\beta|u|^2}|e_n|^2dx\right).
\end{align} 
Now by using the Sogge's inequality \cite{Sogge1988} $\left(\lVert e_n\rVert_{L^\infty(M)}\leq C_M'\lambda_n^{\frac{1}{4}}\right)$, we obtain
\begin{align*}
E''(u;\sqrt{\alpha}~a_ne_n,\sqrt{\alpha}~a_ne_n)\leq \alpha|a_n|^2\left(1+\lambda_n+C_M(2\beta+1)\lambda_n^{\frac{1}{2}}\int_M(1+|u|^2)e^{\beta|u|^2}dx\right).
\end{align*}
Therefore by integrating and taking the expectation, while using the vanishing property of the martingale term, we obtain
\begin{align*}
\mathbb{E}E(u)+\alpha \int_0^t\mathbb{E}\mathcal{E}(u)d\tau \leq \mathbb{E}E(u_0)+\frac{\alpha}{2}A^0_Nt+\frac{\alpha}{2}A^1_Nt+\frac{\alpha}{2}C_M(2\beta+1)A_N^{\frac{1}{2}}\int_0^t\mathbb{E}\left(\int_M(1+|u|^2)e^{\beta|u|^2}dx\right).
\end{align*}
\end{proof}
\subsection*{Existence of Stationnary measure}
\leavevmode
\begin{theorem}
For any \( N \geq 2 \) and \( \alpha \in (0,1) \), there is a stationary measure $\mu_N^{\alpha}$
and we have the following estimates
\begin{align}\label{estimate3}
\int_{L^2} \mathcal{M}(u) \mu_{N}^{\alpha}(du) &=\frac{A_N^0}{2}\leq \frac{A^0}{2}\\ \label{estimate4}
\int_{L^2} \mathcal{E}(u) \mu_{N}^{\alpha}(du) &\leq C_5, 
\end{align}
where $C$ does not depend on $\alpha$ and $N$.

\end{theorem}
\begin{proof}
Let $B_R$ be the closed ball of $E_N$ with center $0$ and raduis $R.$ By using the Bogoliubov-Krylov argument, let us find the measure $\lambda$ such that the sequence $\{\Bar{\lambda_t}\}_{t\geq 0}$ is tight where $\Bar{\lambda_t}=\frac{1}{t}\int_0^t\mathcal{B}_{\tau}^{\alpha,N,*}\lambda ~d\tau.$\\
Let $R>0$, since we are in finite dimensional $E_N$, all the norm are equivalents and the ball $B_R(E_N)$ is compact in $E_N$, It is easy to see by using the Chebyshev's inequality, the estimates (\ref{estimate1}~-~\ref{4}) and by choosing  
 $\lambda=\delta_0$ the dirac measure concentrated at 0; that $\bar{\lambda}_t(B_R^c)\leq \frac{C_N}{R^2}$.\\
Therefore, we have the tightness of $\{\Bar{\lambda_t}\}_{t\geq 0}$, then by Prokhorov theorem there is subsequence $ \{t_n\}_{n\in \mathbb{N}}\subset \tau_+$ and measure $\mu_N^{\alpha}$ supported on $E_N$  such that $\bar{\lambda}_{t_n} \rightharpoonup \mu_N^{\alpha}$ weakly in $P(L^2)$ and $\mu_N^{\alpha}$ is invariant.\\
Let us now prove the estimates (\ref{estimate3}) and (\ref{estimate4})

\subsection*{Estimates} Let $\chi \in C^{\infty}$ be a smooth function such that
$\chi(x) = 1 ~\text{for } x \in [0,1], \text{and}\\
0  ~\text{for } x \in [2,\infty)$
 and let $\chi_R(x)=\chi(\frac{x}{R})$.\\
Since the map $u \to \lVert u\rVert_{L^2}^2\chi_R(\lVert{u}\rVert_{L^2})$ is continuous and bounded on $L^2$, then we have by using the fact $\lVert u\rVert_{L^2}^2\leq \mathcal{M}(u)+C_N$
\begin{align*}
\int_{L^2}\lVert u\rVert_{L^2}^2\chi_R(\lVert{u}\rVert_{L^2}) \bar{\lambda}_{t_n}(du)\leq \int_{L^2}\lVert u\rVert_{L^2}^2\bar{\lambda}_{t_n}(du)\leq C_N+\frac{1}{t_n}\int_0^{t_n}\mathbb{E}\mathcal{M}(u) d\tau\leq C_N+\frac{A^0_N}{2}.
\end{align*}
By passing to the limit $t_n \to \infty,R\to \infty$ and using the Fatou's lemma, we obtain $\mathbb{E}M(u_{0,N})<\infty.$
Therefore using (\ref{estimate1}), and the fact that $\mu_{N}^{\alpha}$ is invariant, we obtain: $\int_{L^2} \mathcal{M}(u) \mu_{N}^{\alpha}(du) =\frac{A_N^0}{2}\leq \frac{A^0}{2}.$\\
On the other hand, we have by using the fact that we are in finite dimensionality\\ $E(u)=\frac{1}{2}\lVert u\rVert_{L^2}^2+\frac{1}{2}\lVert \nabla u\rVert_{L^2}^2+\frac{1}{2\beta}\sum_{p\in \mathbb{N}^*}\frac{\beta^{p+1}}{(p+1)!}\lVert u\rVert_{L^{2p+2}}^{2p+2}\leq C_N\mathcal{M}(u)+C'_N$ and then 
\begin{align*}
\int_{L^2}E(u) \mu_{N}^{\alpha}(du)=\int_{L^2}\left(\frac{1}{2}\lVert u\rVert_{L^2}^2+\frac{1}{2}\lVert \nabla u\rVert_{L^2}^2+\frac{1}{2\beta}\sum_{p\in \mathbb{N}^*}\frac{\beta^{p+1}}{(p+1)!}\lVert u\rVert_{L^{2p+2}}^{2p+2}\right)\mu_{N}^{\alpha}(du)<\infty,  ~~\text{thus}~~\mathbb{E}E(u_{0,N})<\infty.
\end{align*}
Let us remark also that 
\begin{align}\label{2}
    \int_M(1+|u|^2)e^{\beta|u|^2}dx&
    \leq 2e^\beta |M|+2\int_Me^{\beta|u|^2}|u|^2dx
    \leq 2e^\beta|M|+2\sum_{p\in\mathbb{N}^*}\frac{\beta^p}{p!}\lVert u\rVert_{L^{2p+2}}^{2p+2}.
\end{align}
Since the term $ C_M(2\beta+1)A^{\frac{1}{2}}\sum_{p\in\mathbb{N}^*}\frac{\beta^p}{p!}\lVert u\rVert_{L^{2p+2}}^{2p+2}$ can easily be reabsorbed by the term $\frac{C}{2}\sum_{p\in\mathbb{N}^*}\frac{\beta^p}{p!}\lVert u\rVert_{L^{2p+2}}^{2p+2}$ because the constant $C$ is large enough;
then, by using (\ref{estimate2}) and the invariant measure, we have:\\$\int_{L^2} \mathcal{E}_0(u) \mu_{N}^{\alpha}(du) \leq C'_5$ and it follows that $\int_{L^2}\mathcal{E}(u)\mu_{N}^\alpha(du)\leq C_5.$
\end{proof}
On the other hand, a simple application of Ito's formula to $F_R(u)=\lVert{u}\rVert_{L^2}^2\left(1-\chi_R(\lVert{u}\rVert_{L^2}^2)\right)$ by using the estimate \eqref{estimate3} gives us the estimate:
    
    \begin{align}\label{estimate5}
   \forall R>0,\quad \int_{L^2} \mathcal{M}(u)(1-\chi_R(\lVert{u}\rVert_{L^2}^2)) \mu_{N}^{\alpha}(du) \leq C_6R^{-1} \quad \text{ where $C_6$ is independent of $\alpha,N.$}
    \end{align}
In the following, we pass the inviscid limit $\alpha \to 0$.
\subsection{Inviscid Limit}
\leavevmode
\begin{proposition}
    Let $N\geq 2$, there is a measure $\mu_N$ that is invariant under the flow $\phi_N^t$ and satisfies the following estimates:
    \begin{align}\label{estimate6}
     \int_{L^2} \mathcal{M}(u) \mu_{N}(du) &=\frac{A^0_N}{2}\\
     \label{estimate8}
\int_{L^2} \mathcal{E}_0(u) \mu_{N}(du) &\leq C_7, 
\end{align} 
\text{which gives the crucial estimate}
\begin{align}\label{crucial estimate}
\int_{L^2} \bigg(e^{\gamma C_1\lVert P_N(e^{\beta\abs{u}^2}u)\rVert_{L^2}^2}\left( \lVert u\rVert_{H^1}^2+\|P_N( u e^{\beta |u|^2}) \|_{L^2}^2 \right) +\sum_{p \in \mathbb{N}} \frac{\gamma^p}{p!} \lVert  u^{p+1} \rVert_{H^1}^2+\lVert u\rVert_{W^{1,2+\delta}}^{2+\delta}   \bigg) \mu_N(du) \leq C_7'.
\end{align}
Furthermore, we have also the estimate
\begin{equation}\label{estimate7}
   \int_{L^2}E(u)\mathcal{E}_0(u)\mu_N(du)\leq C_8 \quad \text{where $C_8$ does not depend on $N$}.
\end{equation}
\end{proposition}
\begin{proof}
\leavevmode
\subsection*{Existence}
Thanks to the estimate (\ref{estimate3} or
 \ref{estimate4}), we have the weak compactness of any sequence $(\mu_N^{\alpha} )_{\alpha\in (0,1)}$ with respect to the topology of $L^2$, therefore there exists a subsequence $(\mu_N^{\alpha_j})_j$ and a measure $\mu_N$ such that $(\mu_N^{\alpha_j})_j$ converges to $\mu_N$. 
\subsection*{Estimates}
Since $\mu_N^\alpha$ and $ \mu_N $ are supported on $E_N$ and we are actually working in a finite dimensional space, the functions $\mathcal{M}, \mathcal{E}_0$ are continuous. Hence, the estimate \eqref{estimate8} follow from \eqref{estimate3},\eqref{estimate4}, the lower semicontinuity and boundary from below of $\mathcal{E}_0(u)$.\\Let us prove the estimate \eqref{estimate6}. Since we are in finite dimensionality, $\chi_R(\lVert u\rVert_{L^2}^2)\mathcal{M}$ is continuous and bounded, we have 
\begin{align*}
\frac{A_N^0}{2} - \int_{L^2} (1 - \chi_R(\lVert{u}\rVert_{L^2}^2)) \mathcal{M}(u) \mu_{N}^{\alpha_j} (du) = \int_{L^2} \chi_R(\lVert{u}\rVert_{L^2}^2) \mathcal{M}(u) \mu_{N}^{\alpha_j} (du).
\end{align*}
Now, by using \eqref{estimate5}, we will have
\begin{align*}
\frac{A_N^0}{2} - C_7 R^{-1} \leq \int_{L^2} \chi_R(\lVert{u}\rVert_{L^2}^2) \mathcal{M}(u) \mu_{N}^{\alpha_j} (du) .
\end{align*}
Leaving $j\to \infty$ and $R\to \infty$, we arrive at $\frac{A_N^0}{2}\leq \int_{L^2} \mathcal{M}(u) \mu_{N}(du) $.\\
On the other hand, since $\mathcal{M}$ is lower semi continuous and bounded from below, we have
\begin{align*}
    \int_{L^2}\mathcal{M}(u)\mu_N(du)\leq \int_{L^2}\mathcal{M}(u)\mu_N^{\alpha_j}(du)=\frac{A_N^0}{2}.
\end{align*}
 Let us now prove the estimate \eqref{estimate7}. By applying the Ito's formula to $F.G $ where $F ~\text{or}~ G=E(u)$ or $M(u)$, and $u$ is the stationary solution to \eqref{eq3}, we have:
\begin{align}
    dFG &= F(u) G'(u;du)+G(u)F'(u,du)+ \frac{1}{2} \sum_{m=0}^{N}2F'(u,\sqrt{\alpha}~a_me_m)G'(u,\sqrt{\alpha}~a_me_m)\label{Ito's formula}\\
    &\quad+ \frac{1}{2} \sum_{m=0}^{N} \left(  F(u) G''(u,\sqrt{\alpha}~ a_me_m, \sqrt{\alpha}~a_me_m)+G(u) F''(u, \sqrt{\alpha}~a_me_m, \sqrt{\alpha}a_me_m) \right).\nonumber
  \end{align}  
    \begin{align}
   dE^2&= -2\alpha E(u)\mathcal{E}(u)+ 2\sqrt{\alpha} \sum_{m=0}^{N} E(u)E'(u, a_m e_m) d\beta_m + \frac{\alpha}{2} \sum_{m=0}^{N}  2 |a_m|^2\left( E(u) E''(u, e_m, e_m) +  (E'(u, e_m))^2\right) .\nonumber
\end{align}
Now, by using the estimate \eqref{1}and \eqref{2} combined with the Sogge's inequality\cite{Sogge1988}, we have:
\begin{align*}
    \mathbb{E}E^2(u)+2\alpha\int_0^t&\mathbb{E}E(u)\mathcal{E}(u)d\tau\leq \mathbb{E}E^2(u_0)+\alpha\left( A_N^0+ A_N^1+C_{M,\beta} A_N^{\frac{1}{2}}\right)\int_0^t\mathbb{E}E(u)d\tau\\&+\alpha A_N^{\frac{1}{2}}C'_{M,\beta}\int_0^t\mathbb{E}\left(E(u)\left(\int_Me^{\beta|u|^2}|u|^2dx\right)\right)d\tau+\alpha(A^1_N+A^0_N)\int_0^t\mathbb{E}\mathcal{E}_0(u)d\tau.
\end{align*}
Since $\mathcal{E}(u)\geq \mathcal{E}_0(u)-\mathcal{M}(u)-\tilde{C} $, we then obtain
\begin{align}\label{estimate energy-energy}
  &\mathbb{E}E^2(u)+2\alpha\int_0^t\mathbb{E}E(u)\mathcal{E}_0(u)d\tau\leq \mathbb{E}E^2(u_0)+\alpha\left( A_N^0+ A_N^1+C_{M,\beta} A_N^{\frac{1}{2}}\right)\int_0^t\mathbb{E}E(u)d\tau\\
  &+\alpha A_N^{\frac{1}{2}}C'_{M,\beta}\int_0^t\mathbb{E}\bigg(E(u)\bigg(\int_Me^{\beta|u|^2}|u|^2\bigg)\bigg)d\tau+2\alpha\int_0^t\mathbb{E}E(u)\mathcal{M}(u)d\tau+\alpha(A^1_N+A^0_N)\int_0^t\mathbb{E}\mathcal{E}_0(u)d\tau+2\alpha \tilde{C}t\nonumber.  
\end{align}
In order the control $\mathbb{E}E(u)\mathcal{M}(u)$, let us apply the Ito's formula \eqref{Ito's formula} with $F=M(u)$ and $G=E(u)$.
\begin{align}\label{estimate mass-energy}
    \mathbb{E}M(u)E(u)&+\alpha\int_0^t\mathbb{E}E(u)\mathcal{M}(u)d\tau+\alpha\int_0^t\mathbb{E}M(u)\mathcal{E}(u)d\tau\leq \mathbb{E}M(u_0)E(u_0)+\alpha C\int_0^t\mathbb{E}\mathcal{E}_0(u)d\tau \\&+\frac{\alpha}{2}\int_0^t(A_N^0+A_N^1+C_{\beta,M}A_N^{\frac{1}{2}})\mathbb{E}M(u)d\tau
    +\frac{\alpha}{2}C'_{\beta,M}A_N^{\frac{1}{2}}\int_0^t\mathbb{E}M(u)\left(\int_Me^{\beta|u|^2}|u|^2dx\right)d\tau\nonumber.
\end{align}
Since $\mathcal{E}(u)\geq \mathcal{E}_0(u)-\mathcal{M}(u)-\tilde{C} $, we then obtain:
\begin{align}\label{estimate mass-energy'}
    &\mathbb{E}M(u)E(u)+\alpha\int_0^t\mathbb{E}E(u)\mathcal{M}(u)d\tau+\alpha\int_0^t\mathbb{E}M(u)\mathcal{E}_0(u)d\tau\leq \mathbb{E}M(u_0)E(u_0)+\alpha C\int_0^t\mathbb{E}\mathcal{E}_0(u)d\tau \\&+\frac{\alpha}{2}\int_0^t(A_N^0+A_N^1+C_{\beta,M}A_N^{\frac{1}{2}})\mathbb{E}M(u)d\tau
    +\frac{\alpha}{2}C'_{\beta,M}A_N^{\frac{1}{2}}\int_0^t\mathbb{E}M(u)\left(\int_Me^{\beta|u|^2}|u|^2dx\right)d\tau \nonumber\\
&+\alpha\int_{0}^t\mathbb{E}M(u)\mathcal{M}(u)d\tau+\alpha \tilde{C}t\nonumber.
\end{align}

We can easily control the term $\mathbb{E}M(u)\mathcal{M}(u)$ by applying the Ito's formula with $F=G=M(u)$,
\begin{align}\label{estimte mass-mass}
    \mathbb{E}M^2(u)+2\alpha\int_0^t\mathbb{E}M(u)\mathcal{M}(u)d\tau\leq \mathbb{E}M^2(u_0)+3\alpha A_N^0\int_0^t\mathbb{E}M(u)d\tau.
\end{align}
Using the invariance measure $\mu_N^\alpha$ and replacing \eqref{estimte mass-mass} in \eqref{estimate mass-energy'} then \eqref{estimate mass-energy'} in \eqref{estimate energy-energy} while using the fact $\mathbb{E}_{\mu_N^\alpha}\mathcal{E}_0(u)\leq C_7$ and noticing that the term $C'_{\beta,M}A^{\frac{1}{2}}~\mathbb{E}(M(u)+E(u))\left(\int_Me^{\beta|u|^2}|u|^2dx\right)$ can be reabsorbed by $\mathbb{E}E(u)\mathcal{E}_0(u)$, we obtain 
\begin{align}
   \int_{L^2}E(u)\mathcal{E}_0(u)\mu_N^\alpha(du)\leq C_8.\label{Est_E.Cal_E} 
\end{align}
 
The lower semi-continuity and bounded below property of $E(u)\mathcal{E}_0(u)$ and the convergence $\mu_N^{\alpha_j}\to \mu_N$ give us the claim.
\subsection*{Invariance}
It is enough to show only the invariance under $\phi^t_N$ for $t>0$; because for $t<0$, by using the invariance for positives times, 
$\mu_N(\Gamma)=\mu_N(\phi^t_N\Gamma)=\mu_N(\phi^{2t}_N\phi_N^{-t}\Gamma)=\mu_N(\phi_N^{-t}\Gamma)$ which is the aim.\\
Now the proof of the invariance for positives times is summarized in the following diagram
\[
\begin{tikzcd}
\mathcal{B}^{\alpha_j,N*}_t \mu_{N}^{\alpha_j} \arrow[r,equal, "\text{(I)}"] \arrow[d, "\text{(III)}"'] & \mu_{N}^{\alpha_j} \arrow[d, "\text{(II)}"] \\
\Phi_N^{t*} \mu_N \arrow[r,equal, "\text{(IV)}"'] & \mu_N
\end{tikzcd}
\]
Our goal is to prove the equality (IV) and to achieve this, we just need to prove the convergence (III). For that, let $f: L^2 \to \mathbb{R}$ be a lipschitz function that is also bounded by $1$, we have 
$$
    (\mathcal{B}_t^{\alpha_j,N*}\mu_N^{\alpha_j},f)-(\phi_N^{t*}\mu_N,f)=(\mu_N^{\alpha_j},\mathcal{B}_t^{\alpha_j,N}f)-(\mu_N,\phi_N^tf)
    =(\mu_N^{\alpha_j},\mathcal{B}_t^{\alpha_j,N}f-\phi_N^tf)-(\mu_N-\mu_N^{\alpha_j},\phi_N^tf)
    = A_1-A_2.
$$
By using the Feller property of $\phi_N^t$ and the boundedness of $f$, we see that $A_2 \to 0$ as $j\to \infty$ and
$$\abs{A_1}\leq \int_{B_R(L^2)}\abs{\phi_N^tf(u_0)-\mathcal{B}_t^{\alpha_j,N}f(u_0)}\mu_N^{\alpha_j}(du_0)+2\mu_N^{\alpha_j}(L^2\setminus B_R(L^2))=A_3+A_4.$$
We have 
\begin{align*}
    A_3&= \int_{B_R(L^2)}\left|\left(\int_{L^2}f(v)P_t^{\alpha_j,N}(u_0,dv)\right)-\phi_N^tf(u_0) \right|\mu_N^{\alpha_j}(du_0)\\
    &=\int_{B_R(L^2)}\left|\left( \int_{\Omega}f(u_{\alpha_j}(t,P_Nu_0))d\mathbb{P}\right)- f(\phi_N^tP_Nu_0) \right|\mu_N^{\alpha_j}(du_0)\\
    &\leq \int_{B_R(L^2)}\left( \int_{\Omega}\left|(f(u_{\alpha_j}(t,P_Nu_0))- f(\phi_N^tP_Nu_0))\right|d\mathbb{P}\right) \mu_N^{\alpha_j}(du_0)\\
    &\leq  \int_{B_R(L^2)}\left( \int_{S_r}\left|(f(u_{\alpha_j}(t,P_Nu_0))- f(\phi_N^tP_Nu_0))\right|d\mathbb{P}\right) \mu_N^{\alpha_j}(du_0)\\
    &\quad \quad+ \int_{B_R(L^2)}\left( \int_{S_r^c}\left|(f(u_{\alpha_j}(t,P_Nu_0))- f(\phi_N^tP_Nu_0))\right|d\mathbb{P}\right) \mu_N^{\alpha_j}(du_0)\\
    &\leq C_f\int_{B_R(L^2)}\mathbb{E}\left( \lVert u_{\alpha_j}(t,P_Nu_0)-\phi_N^tP_Nu_0\rVert_{L^2}\mathbf{1}_{S_r}\right)\mu_N^{\alpha_j}(du_0)+2\int_{B_R(L^2)}\mathbb{E}\mathbf{1}_{S_r^c}~~\mu_N^{\alpha_j}(du_0)
\end{align*}
by using the Lipschitz property.\\
Now let us consider for $r>0$,
$$ S_r=\left\{ \omega\in \Omega|\max\left(\sup_{\tau\leq t}\left|\sqrt{\alpha_j}\sum_{n=0}^Na_n\int_0^\tau(u,e_n)d\mathcal{B}_n(\tau') \right|,\sup_{\tau\leq t}\lVert z_{\alpha_j}(\tau)\rVert_{L^2}\right)\leq r\sqrt{\alpha_j t} \right\}.$$
Let us compute $\mathbb{E} \mathbf{1}_{S_r^c}$. We have by using the Doob's inequality \\
$$\mathbb{E}\sup_{\tau\leq t}\left|\sqrt{\alpha_j}\sum_{n=0}^Na_n\int_0^\tau(u,e_n)d\mathcal{B}_n(\tau) \right|^2\leq 4\alpha_j\sum_{n=0}^Na_n^2\int_0^t\mathbb{E} (u,e_n)^2d\tau\leq \alpha_jt A_0\mathbb{E}\lVert u\rVert_{L^2}^2\leq C\alpha_jt .$$\\
We have also $\mathbb{E}\sup_{\tau\leq t}\lVert z_{\alpha_j}(\tau)\rVert^2_{L^2}\leq C\alpha_jt$. \\
According to Chebyshev's inequality, we have:\\
$$\mathbb{E} \mathbf{1}_{S_r^c}=\mathbb{P}\left\{ w| \max\left(\sup_{\tau\leq t}\left|\sqrt{\alpha_j}\sum_{n=0}^Na_n\int_0^t(u,e_n)d\mathcal{B}_n(\tau') \right|,\sup_{\tau\leq t}\lVert z_{\alpha_j}(\tau)\rVert_{L^2}\right)\geq r\sqrt{\alpha_j t}\right\}\leq \frac{C\alpha_j t}{r^2\alpha_jt}=\frac{C}{r^2}.$$
We need to prove now the following statement 
\begin{lemma}
    We have for any $R>0$ and $r>0$ 
    \begin{align}
\sup_{u_0 \in B_R(L^2)} \mathbb{E}\left( \lVert \phi_N^t P_N u_0 - u_{\alpha_j}(t, P_N u_0)\rVert_{L^2}\mathbf{1}_{S_r} \right) \to 0 \text{ as } j \to \infty.
\end{align}
\end{lemma}
\begin{proof}
    Let us set $w_{j}=u-v_{j}=\phi_N^tP_Nu_0-v_{j}(t,P_Nu_0)$ with $w_j=w_{\alpha_j}$. We have 
    \begin{align*}
  \partial_t &w_{j}=i\left[\Delta w_{j}-P_N\left( e^{\beta|u|^2} u- e^{\beta|v_{j} + z_{j}|^2}(v_{j} + z_{j})\right)  \right]+\alpha_jP_N\left((1-\Delta)^{-\frac{1}{2}}\left(|(1-\Delta)^{\frac{1}{2}}(v_j+z_j)|^\delta(1-\Delta)^{\frac{1}{2}}(v_j+z_j)    \right)\right)\\
  &+\alpha_j\left( Ce^{\gamma C_1\lVert P_N(e^{\beta\abs{v_j+z_j}^2}(v_j+z_j))\rVert_{L^2}^2}\left((v_{j}+z_{j})+P_N(e^{\beta\abs{v_{j}+z_{j}}^2}(v_{j}+z_{j})\right)+C_2P_N\left( e^{\gamma |v_{j} + z_{j}|^2} (v_{j} + z_{j}) \right) 
     \right)
     \end{align*}
Since $\mathbb{E}\lVert z_{j}\rVert_{L^2}\to 0 ~~\text{as}~~ j \to \infty,$ we just need to prove that $\sup_{u_0 \in B_R(L^2)} \mathbb{E}\left( \lVert w_{j}\rVert_{L^2}\mathbf{1}_{S_r} \right) \to 0 \text{ as } j \to \infty.$\\
By using the fact that:
\[
u\,e^{\beta |u|^{2}} - u_j\,e^{\beta |u_j|^{2}}
= \sum_{p\ge 0} \frac{\beta^{p}}{p!}\Big( |u|^{2p}u - |v_j+z_j|^{2p}(v_j+z_j) \Big)
= w_j\sum_{p\ge 0} f_{2p}(u,v_j) \;+\; z_j \sum_{p\ge 0}  g_{2p}(v_j,z_j)
\]
where $f_{2p}$ and $g_{2p}$ are polynomials of degree $2p$ in the given variables and the series $F(u,v_j)=\sum_{p\ge 0} f_{2p}(u,v_j)$ and $ G(v_j,z_j)=\sum_{p\ge 0}  g_{2p}(v_j,z_j)$ are convergent and we obtain by taking the inner product with $w_j$,
\begin{align}\label{equation in w_j}
    \frac{1}{2}\partial_\tau \lVert w_j\rVert_{L^2}^2&\leq \lVert w_j\rVert_{L^2}^2\lVert F(u,v_j)\rVert_{L^{\infty}_{t,x}}+\lVert w_j\rVert_{L^2}\lVert z_j\rVert_{L^2}\lVert G(v_j,z_j)\rVert_{L^{\infty}_{t,x}}+\alpha_j \tilde{C}_N\lVert w_j\rVert_{L^2}\lVert u_j\rVert_{L^2}^{1+\delta}\\ \nonumber
    & +\alpha_j\lVert w_j\rVert_{L^2}\lVert u_j\rVert_{L^2}\left[C\exp{\left(\gamma C_1\exp{\left(2\beta\lVert u_j\rVert_{L^\infty_{t,x}}^2\right)}\lVert u_j\rVert_{L^2}^2\right)}\left(1+e^{\beta\lVert u_j\rVert_{L^\infty_{t,x}}^2}\right)+C_2e^{\gamma\lVert u_j\rVert_{L^\infty_{t,x}}^2}\right].
\end{align}
Thanks to the definition of $S_r$, we have $\sup_{\tau\leq t}\lVert z_j(\tau)\rVert_{L^2}\leq r\sqrt{\alpha_j t}$ with $\alpha_j\in (0,1)$ and $\alpha_j\to 0$ which shows that $z_j$ is uniformly bounded in $w,j,\tau\in[0,t]$, but on the other hand the estimate of $v_j$ given by \eqref{estimate on v}  is not uniform in $j$, because of the blow ups as $j\to \infty$ due to the Legendre transformation $\Phi_{\alpha_j C_2}^{\gamma*}(|z_j|)=\alpha_jC_2\Phi^{\gamma*}(\frac{|z_j|}{\alpha_j C_2})$ and the super exponential growth of $\Phi^{\gamma*}$. However, by applying the Ito's formula to $M=\frac{1}{2}\lVert u_j\rVert_{L^2}^2$ with $u_j=z_j+v_j$, we have 
\[
dM = \alpha_j\left( \frac{A^0_N}{2} - \mathcal{M}(u_j)\right)\,d\tau 
   + \sqrt{\alpha_j} \sum_{|n|\leq N} a_n (u_j, e_n) \, d\beta_n.
\]
Since $\mathcal{M}\geq -C_N$ and by definition of $S_r$, the martingale term is controlled uniformly in 
$\omega, \alpha_j$, we then obtain: 
$$
\frac{1}{2} \|u_j(\tau)\|_{L^2}^2  \leq \frac{1}{2}\lVert u_0\rVert_{L^2}^2 
+ \alpha_j \frac{A^0_N}{2}t +\alpha_jC_N t+ r \sqrt{\alpha_j t} 
\leq C^N_{R,t,r} \quad \text{uniformly in}~ j, w,\tau \in[0,t],u_0\in B_R(L^2).
$$
Therefore since $z_j,v_j$ and $u_j$ are bounded uniformly in $j, w,\tau \in[0,t],u_0\in B_R(L^2)$, a simple Young inequality to \eqref{equation in w_j} gives $\partial_\tau \lVert w_j\rVert_{L^2}^2\leq C_6 \lVert w_j\rVert_{L^2}^2+C_7 \alpha_j$ and by Gronwall's lemma with $w_j(0)=0$, we have:
$$
\sup_{u_0 \in B_R} \| w_j(t) \|_{L^2}^21_{S_r} 
\leq C_7 t e^{C_6 t} \alpha_j 
\xrightarrow[j \to \infty]{} 0,
~\text{hence}~ 
\sup_{u_0 \in B_R} \mathbb{E}\, \| w_j(t;u_0) \|_{L^2}1_{S_r} 
\leq \mathbb{E} \sup_{\tilde{u}_0 \in B_R} \| w_j(t;\tilde{u}_0) \|_{L^2}1_{S_r}
\to 0.
$$

\end{proof}
Which completes the proof.
\end{proof}

\section{Infinite-dimensional limit and Global theory}

Our goal now is to exploit the result established in the inviscid limit. We will use them to take the infinite-dimensional limit $N\to \infty$ and establish the global theory claimed in our main theorem. In what follow we assume that $\gamma >4\beta$.
\subsection{Existence of global solution} For any $k\in \mathbb{N} $, let us set the following spaces :
\begin{align}
X_k = L^2([-k, k], H^1) \cap \left\{H^1([-k, k], H^{ - 1}) + H^1([-k, k], L^2)\right\} = X_k^1 + X_k^2;
\end{align}
\begin{align}\label{space Y}
Y_k = L^2([-k, k], H^{1^ -}) \cap C([-k, k], H^{(- 1)^-}+H^{0^-}).
\end{align}
By using the classical Sobolev embedding, we see that \( H^1([-k, k], H^{-1}) \hookrightarrow C^{\frac{1-}{2}}([-k, k], H^{-1}) \), hence\\ \( X^1_k \hookrightarrow C^\frac{1-}{2}([-k, k], H^{-1}) \).  
We also see that \( X^2_k \hookrightarrow C^{\frac{1-}{2}}([-k, k], L^2) \). Thus, we have that \( X_k^1 \) and \( X_k^2 \) are compactly embedded in \( C([-k, k], H^{(- 1)^-}) \) and \( C([-k, k], H^{0^-}) \), respectively.  
We see also that both \( X_k^1 \) and \( X_k^2 \) are also compactly embedded in \( L^2([-k, k], H^{1^-}) \) (Aubin-Lions), therefore, we obtain that \( X_k \) is compactly embedded in \( Y_k \). We will need to split the space $Y_k$ as follows:
\begin{align}
    Y_k^1 &= L^2([-k, k], H^{1^ -}) \cap C([-k, k], H^{(- 1)^-}),\\
    Y_k^2 &= L^2([-k, k], H^{1^ -}) \cap C([-k, k], H^{0^-}).
\end{align}
As we will see it below, $Y_k^1$ will manage the linear part of equation (more precisely, the Laplace term), and $Y_k^2$ the nonlinear one, in the convergence procedure.

Let us denote by \( \lambda_N^k \) the distributions of the processes \( (u_N(t))_{t \in [-k, k]} \), and these are seen as random variables valued in \( C([-k, k], H^{({-1})^-}) \). Since, the measure $\mu_N$ is invariant, then we have:
\begin{equation}\label{lifthing}
\mu_N = \lambda^k_N|_{t=t_0} \quad \text{for any} \quad t_0 \in [-k, k].
\end{equation}
\begin{proposition}
    Let $\theta \in (0,\frac{\gamma}{2})$. We have that  
\begin{align}
\int_{Y_k} \left(\| u \|^2_{X_k} +\lVert e^{\theta|u|^2}\rVert^2_{L^2_tH^1_x}+\lVert e^{\theta|u|^2}u\rVert^2_{L^2_tH^1_x}\right)\lambda^k_N(du) &\leq C'k, \label{estimate9}\\
    \mathbb{E}\,\|\partial_t u_N^k\|_{L^2([-k,k];H^{-1})}^2 &\leq C''k.\label{dt in L2}
\end{align}
where $C'$ and $C''$ do not depend on \( N \). 
\end{proposition}
\begin{proof}
From the estimate (\ref{crucial estimate}), We derive that  
\[
   \mathbb{E}  \| u_N \|_{L^2([-k,k],H^1)}^2 =  \int_{-k}^k \mathbb{E}\| u_N(\tau) \|_{H^1}^2 d\tau \leq 2kC_2.
\]

By writing $u_N^k:=u_N|t\in[-k,k]$ in the integral formulation:
\begin{align}\label{formulation u^k_N}
u_N^{k}(t) = u_{0,N} + i \int_0^t (\Delta-1) u_N^{k}(\tau) d\tau -i\int_0^t P_N ((e^{\beta\abs{u_N^{k}(\tau)}^2}-1)u_N^{k}(\tau))~d\tau
\end{align}

By using the estimate \eqref{crucial estimate}, we have the following estimates
\begin{align*}
\mathbb{E}\left\|\int_0^t    (\Delta-1) u_N^k d\tau \right\|_{H^1([-k,k],H^{- 1})}^2 &\leq C_3 \mathbb{E}\left(\int_{-k}^k\left(\partial_t\int_0^t\lVert(1-\Delta) u^k_N\rVert_{H^{-1}}d\tau\right)^2dt
\right)\leq C_4 \mathbb{E}\left(\int_{-k}^k \| u_N^k \|_{H^1}^2 d\tau   \right)
\leq C_5k 
\end{align*}
\begin{align*}
    \mathbb{E}\left\lVert\int_0^t P_N ((e^{\beta\abs{u_N^k}^2}-1)u_N^k)~d\tau \right\rVert_{H^1([-k,k],L^2)}^2&\leq 2\mathbb{E}\left(\left\lVert\int_0^t u_N^k~d\tau \right\rVert_{H^1([-k,k],L^2)}^2+\left\lVert\int_0^t P_N (e^{\beta\abs{u_N^k}^2}u_N^k)~d\tau \right\rVert_{H^1([-k,k],L^2)}^2   \right)\\
    &\leq 2C_3\mathbb{E}\left(\int_{-k}^k\lVert u^k_N\rVert^2_{L^2}+\lVert P_N(e^{\beta\abs{u_N^k}^2}u_N^k) \rVert^2_{L^2}dt\right)
    \leq C_6k 
\end{align*}
In particular, we obtain \eqref{dt in L2}.

Summarizing the calculations above, we obtain 
\begin{align}\label{estimate on X_k}
\int_{Y_k} \| u \|^2_{X_k} \lambda^k_N(du) \leq Ck.
\end{align}
On the other hand, to control the $L^2_tH^1_x$ norm of $e^{\theta|u|^2}$, we use the fact$\int_{L^2}\sum_{p\in \mathbb{N}}\frac{\gamma^p}{p!}\lVert u^{p+1}\rVert_{H^1}^2\mu_N(du)\leq C,$ contained in \eqref{crucial estimate}.

Therefore, we have, by using this estimate and the Young inequality,
$$\int_{Y_k}\lVert e^{\theta|u|^2}\rVert^2_{L^2_tL^2_x}~d\lambda^k_N\lesssim_{\theta,k,|M|} 1+\int_{Y_k}\int_{-k}^k\int_M e^{\gamma|u|^2}|u|^2dx~dt~d\lambda_N^k \leq C'k.$$
Next,
\begin{align*}
\int_{Y_k}\lVert \nabla \left(e^{\theta|u|^2}\right)\rVert^2_{L^2_tL^2_x}~d\lambda^k_N= &\leq C_{1,\theta}\int_{Y_k}\int_{-k}^k\int_M |\nabla u|^2 e^{\gamma |u|^2}dx~dt~d\lambda_N^k\\
&\leq C_{1,\theta}\int_{-k}^k\int_{L^2}\sum_{p\in \mathbb{N}}\frac{\gamma^p}{p!}\lVert u^{p+1}\rVert_{H^1}^2\leq C_2k.
\end{align*}
In similar way, we can also show that $\int_{Y^k}\lVert e^{\theta|u|^2}u\rVert^2_{L^2_tH^1_x}d\lambda^k_N \leq Ck$.\\
This finishes the proof.
\end{proof}
\begin{remark}\label{2'}
   \begin{itemize}
       \item By using the estimate \eqref{estimate on X_k} and the compactness of $X_k$ in $Y_k$,  we have the existence of a subsequence $(\lambda^k_{N_j})_j$ and the measure $\lambda^k$ such that $\lambda^k_{N_j}\rightharpoonup \lambda^k$ in $P(Y_k)$. Thanks to the Skorokhod representation theorem, there exist $Y_k$ valued random variables $\widetilde{u}_{N_j}^k$ and $\widetilde{u}^k$, all defined in the same probability space that we denote again by $(\Omega,\mathcal{F}, \mathbb{P})$, such that  
\begin{enumerate}   
    \item $\widetilde{u}_{N_j}^k$ and $\widetilde{u}^k$ are distributed by $\lambda^k_{N_j}$ and $\lambda^k$, respectively. \\
    \item $\widetilde{u}_{N_j}^k \to \widetilde{u}^k$ ~~$\mathbb{P}$- almost surely in $ Y_k$ as $j\to \infty$
\end{enumerate}
       \item Similar argument, combined with the estimate \eqref{estimate9}, shows that the sequences $\left\{\int_0^te^{\theta|\widetilde{u}^k_N|^2}\,d\tau\right\}_N$ and $\left\{\int_0^te^{\theta|\widetilde{u}^k_N|^2}\widetilde{u}^k_N\,d\tau\right\}_N$ are compact in $C([-k,k],H^{1^-})$ for all $0<\theta< \gamma/2$, and in particular in the space $C([-k,k],L^p_x)$ for all $p>1$ $\mathbb{P}$- almost surely. As such, up to a subsequence that we can identify with $\widetilde{u}^k_N$, we have the bounds  (for $\theta<\frac{\gamma}{2}$): 
       \begin{align}
           \sup_{t\in[-k,k]}\|\int_0^te^{\theta|\widetilde{u}^k_N|^2}\,d\tau\|_{L^p_x} &\leq C_{p,k}^\omega\quad \forall\, p>1, \label{Bound_CtLp_expterm}\\
           \sup_{t\in[-k,k]}\|\int_0^te^{\theta|\widetilde{u}^k_N|^2}\widetilde{u}^k_N\,d\tau\|_{H^{1-\kappa}_x} &\leq C_{\kappa,k}^\omega\quad \forall\, \kappa>0.\label{Bound_CtLp_expterm2}
       \end{align}
   \end{itemize} 
\end{remark}

It following the estimate \eqref{crucial estimate} that   any limiting process $\tilde{u}^k$ belongs $\mathbb{P}$-a.s to $L^2([-k,k];H^1)\cap L^r([-k,k];L^q)\cap L^{2+\delta}_t([-k,k],W_x^{1,2+\delta})$ $  \forall r,q\in[1,\infty).$\\

Next, we want to show that $\widetilde{u}^k$ is a strong solution of (\ref{Intro_Trudinger})~~$\mathbb{P}$-a.s on $t[-k,k]$. To do this, we pass to the limit in the following integral form of the Galerkin projection 
\begin{align}\label{Integ_formulat}
\widetilde{u}_{N_j}^{k}(t) = \widetilde{u}_{0,N_j} + i \int_0^t (\Delta-1) \widetilde{u}_{N_j}^{k}(\tau) d\tau -i\int_0^t P_{N_j} ((e^{\beta\abs{\widetilde{u}_{N_j}^{k}(\tau)}^2}-1)\widetilde{u}_{N_j}^{k}(\tau))~d\tau.
\end{align}
In the light of the compactness established above, we see that the linear terms in \eqref{Integ_formulat} converges in $C_tH^{(-1)^-}$. Let us establish the convergence of the nonlinear term in $C_tH^{(-1)^-}$. It suffices to establish such convergence in  $C_tL^2_x$.
\begin{lemma}\label{3'}
   Let us denote by $F(u)=ue^{\beta|u|^2}$ and $F_{N_j}=P_{N_j}F$. We have that:\\
   $$\int_0^tF_{N_j}(\widetilde{u}_{N_j}^k)\,d\tau=\int_0^tP_{N_j}\left(\widetilde{u}^k_{N_j}e^{\beta|\widetilde{u}^k_{N_j}|^2}  \right)\,d\tau\to \int_0^t\widetilde{u}^ke^{\beta|\widetilde{u}^k|^2}\,d\tau=\int_0^tF(\widetilde{u}^k)\,d\tau\quad \quad~\text{in~~ $C_tL^2_{x}$}\ \ \mathbb{P}\text{-a.s} ~~\text{as}~~j\to \infty.$$-
\end{lemma}
\begin{proof}
    Let us write
    \begin{align}
        \int_0^tP_{N_j}e^{\beta|\tilde{u}^k_{N_j}|^2}\tilde{u}^k_{N_j}-e^{\beta|\tilde{u}^k|^2}\tilde{u}^k\,d\tau=P_{N_j}\int_0^te^{\beta|\tilde{u}^k_{N_j}|^2}\tilde{u}^k_{N_j}-e^{\beta|\tilde{u}^k|^2}\tilde{u}^k\,d\tau-(1-P_{N_j})\int_0^te^{\beta|\tilde{u}^k|^2}\tilde{u}^k\,d\tau.
    \end{align}
    
    Now, we use Hölder's inequality repetitively in $t$, and in $x$, to obtain 
 \begin{align}
     \lVert P_{N_j}\int_0^te^{\beta|\tilde{u}^k_{N_j}|^2}\tilde{u}^k_{N_j}-&e^{\beta|\tilde{u}^k|^2}\tilde{u}^k\,d\tau\rVert_{C_tL^2_{x}}\leq\\
     &C\sup_t\lVert\int_0^t \left(e^{\beta|\tilde{u}^k_{N_j}|^2}(|\tilde{u}^k_{N_j}|^2+1)+ e^{\beta|\tilde{u}^k|^2} (|\tilde{u}^k|^2+1)\right)^2\,d\tau\,\rVert_{L^2_{x}}^\frac{1}{2} \lVert \tilde{u}^k_{N_j}-\tilde{u}^k\rVert_{L^2_tL^4_x}\\
     &\leq C'\sup_t\lVert\int_0^t  e^{2\beta^+|\tilde{u}^k_{N_j}|^2}+ e^{2\beta^+|\tilde{u}^k|^2} \,d\tau\,\rVert_{L^2_{x}}^\frac{1}{2} \lVert \tilde{u}^k_{N_j}-\tilde{u}^k\rVert_{L^2_tL^4_x}
 \end{align}
By using the estimate \eqref{Bound_CtLp_expterm} and the condition $2\beta <\frac{\gamma}{2}$
, we have the boundedness of the first term in the RHS. We arrive 
\begin{align}
    \lVert P_{N_j}\int_0^te^{\beta|\tilde{u}^k_{N_j}|^2}\tilde{u}^k_{N_j}-e^{\beta|\tilde{u}^k|^2}\tilde{u}^k\,d\tau\rVert_{C_tL^2_{x}}\leq C_{k}^\omega \lVert \tilde{u}^k_{N_j}-\tilde{u}^k\rVert_{L^2_tL^4_x}
\end{align}
and the convergence follows using the convergence of $\tilde{u}^k_{N_j}\to \tilde{u}^k$ in $L^2_tH^{1^-}_x$.\\
It remains to use the reverse Poincaré inequality and the estimate \eqref{Bound_CtLp_expterm2} to find

\begin{align}
    \|(1-P_{N_j})\int_0^te^{\beta|\tilde{u}^k|^2}\tilde{u}^k\,d\tau\|_{C_tL^2_x}\leq {N_j}^{-1^{-}}\|\int_0^te^{\beta|\tilde{u}^k|^2}\tilde{u}^k\,d\tau\|_{C_tH^{1^{-}}_x}\leq C{N_j}^{^{-1^-}}.
\end{align}
This concludes the proof.
\end{proof}
\begin{remark}
 \begin{enumerate}
     \item Below, we will show that uniqueness holds in the class of the constructed solutions;
     \item Since $k$ was arbitrary, we can use a diagonal argument to obtain global solutions (i.e., existing on the full time interval $(-\infty,+\infty)$) ;
     \item In the following, we will show improved regularity properties of the solutions.
 \end{enumerate}
\end{remark}

\subsection{Invariance law of the solution $u(t)$}
\leavevmode\\
For the invariance of the law of \( u(t) \), let us denote by \( \lambda \) the distribution of the process \( u = (u(t))_{t \in \mathbb{R}} \) then there exist a subsequence $\lambda_{N_j}$ such that $\lambda_{N_j}\to \lambda$. 
For this subsequence in the \( N \)-parameter that produced \( \lambda \), we can extract a subsequence, using the Prokhorov theorem, that produces a measure \( \mu \) as a weak limit point of \( (\mu_N) \). 
Passing to the limit along this subsequence in the relation \eqref{lifthing}, we see that \( \mu = \lambda |_{t = t_0} \) for any \( t_0 \in \mathbb{R} \). 
This establishes that $ \mu $ is an invariant law for $ u $.
\subsection{Regularity of the solution}

From the bound \eqref{crucial estimate}, the weak lower semicontinuity of convex functions of the norms, noticing the compactness of $u_N$ and $|u_N|^q$, for all $q\geq 2$, in $\cap_{p>1}L^p\cap_{\kappa>0}H^{1-\kappa}$, we obtain that
\begin{align}
    \int_{L^2}\mathcal{E}_0(u)\mu(du)=C <+\infty,
\end{align}
and by invariance,
\begin{align}
    \int_{L^2}\int_0^t\mathcal{E}_0(u(\tau,v))d\tau\mu(dv)=Ct<+\infty.
\end{align}
where $u(t,v)$ is the value of the process $u(t)$ knowing that $u(0)=v$.
In particular, $\mathbb{P}$-almost surely the limiting process $u$ satisfies
\begin{align}
    h^\omega(t):=\sum_{p\in \mathbb{N}}\frac{\gamma^p}{p!}\lVert u^\omega(t)^{p+1}\rVert_{H^1}^2 &<\infty\label{estimate uniquenesss}\\
    \int_0^th^\omega(\tau)d\tau=\int_0^t\sum_{p\in \mathbb{N}}\frac{\gamma^p}{p!}\lVert (u^\omega)^{p+1}(t)\rVert_{H^1}^2\,d\tau & \leq C_t<\infty, \label{integ:h}
\end{align}
(where, besides finiteness, the dependence of $C_t$ on $t$ is not known explicitly).\\
Also, by $\mathbb{P}-$ almost sure weak convergence (Estimate \eqref{crucial estimate} + Skorokhod) and the energy and mass conservation in finite-dimension, we obtain
\begin{align}
E({u}^\omega(t)) &\leq \liminf_{j\to\infty}E({u}^\omega_{N_j}(t))=\liminf_{j\to\infty}E({u}^\omega_{N_j}(0)) \\
M({u}^\omega(t)) &= \lim_{j\to\infty}M({u}^\omega_{N_j}(t))=\lim_{j\to\infty}M({u}^\omega_{N_j}(0))=M(u^\omega(0)).
\end{align}

We obtain that for a.a $\omega\in\Omega,  u^\omega\in L^{\infty}_tH^1_x\cap L^{2p+2}_x~~\forall  p\in \mathbb{N}$.\\\\
We want to show now that that $\mathbb{P}$-almost surely $u^\omega\in C(\mathbb{R},H^1).$ For that, we need this following lemma.
\begin{lemma}\label{uniqueness lemma}
    For $\gamma>4\beta$, for $\mathbb{P}$-almost all $\omega$, for any $t>0$, there exists $C_{3,t}^w<\infty$ such that:
 \begin{align}
     \int_0^t\lVert (e^{\beta^+|u^\omega|^2}-1)\rVert_{\dot{H}^1}^2d\tau\leq C_{3,t}^\omega.
 \end{align}
\end{lemma}
\begin{proof}

We have \begin{align*}
 \lVert (e^{\beta^+|u^\omega|^2}-1)\rVert_{\dot{H}^1}=\left\lVert  \nabla\left(\sum_{k\in \mathbb{N}^*}\frac{({\beta^+})^k}{k!}|u^\omega|^{2k}\right)\right\rVert_{L^2}\leq \sum_{k\in \mathbb{N}^*}\frac{(\beta^+)^k}{k!}\left\lVert \nabla |u^\omega|^{2k}\right\rVert_{L^{2}}.
 \end{align*}
Let $2k=p+1$ i.e $p=2k-1$, by using the estimate \eqref{estimate uniquenesss}, 
\begin{align*}
\sum_{p \in \mathbb{N}} \frac{\gamma^p}{p!} \lVert \nabla (u^\omega)^{p+1}(t) \rVert_{L^2}^2 = h^\omega(t)<\infty \quad \mathbb{P}~ a.e ~~,\text{we have}~~~\forall p\in \mathbb{N}, \lVert \nabla (u^\omega)^{p+1}(\tau)\rVert_{L^2}^2\leq \frac{h^\omega(\tau)p!}{\gamma^p} \quad \mathbb{P}~ a.s.
\end{align*}
We then obtain:
\begin{align*}
\left\lVert \nabla |u^\omega|^{2k}\right\rVert_{L^{2}}\lesssim \sqrt{h^\omega(\tau)}\left(\frac{(2k-1)!}{\gamma^{2k-1}}\right)^\frac{1}{2}=\frac{\sqrt{(2k-1)!}\gamma^\frac{1}{2}}{\gamma^k}\sqrt{h^\omega(\tau)}
\end{align*}
And, over all, we obtain:
\begin{align}
 \lVert (e^{\beta^+|u^\omega|^2}-1)\rVert_{\dot{H}^1}\lesssim \sqrt{h^\omega(\tau)}\gamma^\frac{1}{2}\sum_{k\in\mathbb{N^*}}\left(\frac{\beta^+}{\gamma}\right)^k\frac{\sqrt{(2k-1)!}}{k!} . 
\end{align}
Using the basic D'Alembert's convergence criterion, It is easy to see that for $\gamma>2\beta$, the series $\sum_{k\in\mathbb{N^*}}\left(\frac{\beta^+}{\gamma}\right)^k\frac{\sqrt{(2k-1)!}}{k!}$ converges. Since $\gamma>4\beta$, the condition is fulfilled. Then $ \lVert (e^{\beta^+|u^\omega(\tau)|^2}-1)\rVert_{\dot{H}^1}\leq C \sqrt{h^\omega(\tau)} \quad \mathbb{P}$- a.a. $\omega$. It remains to use the integrability of the function $h^\omega$ in \eqref{integ:h}.
\end{proof}

Using Duhamel's formulation $u(t)=e^{it(\Delta-1)}u_0-i\int_0^t e^{i(t-\tau)(\Delta-1)}\left(e^{\beta|u(\tau)|^2}-1)u(\tau)\right)d\tau$, we have the following, by applying Kato's Ponce inequality \text{with}~$ \frac{1}{2}=\frac{1}{2+\delta}+\frac{1}{f(\delta)}=\frac{1}{2}+\frac{1}{\infty}$ and the Sobolev embedding \\ $W^{1,2+\delta}\hookrightarrow L^\infty$ in dimension 2 (without loss of generality, assume that $t<t'$):
\begin{align*}
\lVert u(t)-u(t')\rVert_{H^1}&\leq \lVert (S(t)-S(t'))u_0\rVert_{H^1}+\bigg(C(\delta)+\sqrt{f(\delta)}\bigg)\int_t^{t'} \lVert\nabla(e^{\beta|u|^2}-1)\rVert_{L^2}\lVert u(\tau)\rVert_{W^{1,2+\delta}}d\tau\\
&\leq \lVert (S(t)-S(t'))u_0\rVert_{H^1}+\bigg(C(\delta)+\sqrt{f(\delta)}\bigg)\lVert\nabla(e^{\beta|u|^2}-1)\rVert_{L^2([t,t'];L^2_x)}\|u\|_{L^2([t,t'];W^{1,2+\delta}_x)}.
\end{align*}
Thanks to the strong continuity of $S(\cdot)=e^{i\cdot(\Delta-1)}$ on $H^1$, Lemma \ref{uniqueness lemma} and the fact that $u\in L^{2+\delta}_{loc}W^{1,2+\delta}~\mathbb{P}$ a.e, we arrive at $u^\omega\in C(\mathbb{R},H^1)$. Therefore, we obtain $u\in C(\mathbb{R},H^1)\cap L^{2+\delta}_{loc}W^{1,2+\delta}_x$ for $\mathbb{P}$-a.s.
\subsection{Uniqueness of the solution }\label{1}
\leavevmode\\
Let $u_1, u_2$ be two solutions of \eqref{Intro_Trudinger} starting at $u_0\in \text{Supp}(\mu)$. We have :\\
\begin{equation*}
\begin{cases}
    \partial_t u_1 = i\left((\Delta-1) u_1 - (e^{\beta |u_1|^2} - 1)u_1\right), \\
    \partial_t u_2 = i\left((\Delta-1) u_2 - (e^{\beta |u_2|^2} - 1)u_2\right).
\end{cases}
\end{equation*}
Let $u=u_1-u_2$; by taking the equation in $u$ and performing the inner product with u, {while using Green's formula and the homogeneous Dirichlet boundary condition}, we have :\\
\begin{align*}
    \frac{1}{2}\frac{d}{dt}\lVert u\rVert_{L^2}^2&\leq C_1\left\langle\left((e^{\beta^+|u_1|^2}-1)+(e^{\beta^+|u_2|^2}-1)\right)|u|,|u|\right\rangle
\end{align*}

We observe that, in the spirit of applying Gronwall's lemma, we nearly need to control $L^\infty_tL^\infty_x$ (actually exponential integrability in time of the $L^\infty_x$ norm) of the solution. However, we know that the $ L^\infty_x $ norm is not controlled by $ H^1_x $ in dimension 2, which makes us unable to use our $C_tH^1_x$ property, leaving the uniqueness problem very complicated, since Gronwall's lemma fails in this case. On the other hand, even though our $L^{2+\delta}_tW^{1,2+\delta}_x$ bound can control the $L^{2+\delta}_tL^\infty_x$-norm of the solution, for our matter we are in need of a much more time integrability condition.  However, in dimension 2, all the $ L^p $ norms (for $ p < \infty $) are controlled by $ H^1 $. Thus, following Yudovich's argument, we can establish the uniqueness. To achieve this, we apply the abstract version of the Yudovich argument that we have developed in theorem \eqref{yudowich theorem}. In order of ideas, it suffices to show that $$\forall t_0>0,~~\int_0^{t_0}\left(\lVert \nabla(e^{\beta^+|u_1^\omega|^2}-1)\rVert_{L^2}+\lVert \nabla(e^{\beta^+|u_2^\omega|^2}-1)\rVert_{L^2}\right)\left(1+\lVert u_1^\omega\rVert_{H^1}+\lVert u_2^\omega\rVert_{H^1}\right)d\tau<\infty.$$
Since $u_1^\omega,u_2^\omega\in L^2_{loc}(\mathbb{R},H^1)$ and $\int_0^t\lVert \nabla(e^{\beta^+|u_i^\omega|^2}-1)\rVert_{L^2}^2d\tau \leq C_{3,t}^{\omega,i}$  with $i=1,2$  $\mathbb{P}-a.a $ according to lemma \eqref{uniqueness lemma}, we have the claim by applying theorem \eqref{yudowich theorem}.
\subsection{Continuity of the flow with respect to the initial data in $H^1$}\label{1'}
\leavevmode
Let $u_1, u_2$ be two solutions of \eqref{Intro_Trudinger} starting at $u_1^0, u_2^0\in \text{Supp}(\mu)$. We have :
\begin{equation*}
\begin{cases}
    \partial_t u_1 = i\left((\Delta-1) u_1 - (e^{\beta |u_1|^2} - 1)u_1\right), \\
    \partial_t u_2 = i\left((\Delta-1) u_2 - (e^{\beta |u_2|^2} - 1)u_2\right).
\end{cases}
\end{equation*}
Let $u=u_1-u_2$; by taking the equation in $u$ and performing the inner product with $u$ while using {the Green's formula and the homogeneous Dirichlet boundary condition} and the Yudovich's argument, we have :
\small
\begin{align}\label{estimate_cont}
     \sup_{t\in [T_0,T_1]}\lVert u(t) \rVert_{L^2}^{2} &\lesssim \left(\sqrt{1-\lambda}  \int_{T_0}^{T_1} \left( 1+\lVert u_1\rVert_{{\dot{H}^1}}+\lVert u_2\rVert_{\dot{H}^1}\right) \left(\lVert (e^{\beta^+|u_1|^2}-1)\rVert_{\dot{H}^1}+\lVert (e^{\beta^+|u_2|^2}-1)\rVert_{\dot{H}^1} \right) d\tau +\lVert u(T_0)\rVert_{L^2}^{2(1-\lambda)} \right)^{\frac{1}{1-\lambda}}.
\end{align}
Passing $\lVert u(T_0)\rVert_{L^2}\to 0$ implies for any $\lambda\in (0,1)$\\
$$\lim_{\lVert u(T_0)\rVert_{L^2}\to 0}\sup_{t\in [T_0,T_1]}\lVert u(t) \rVert_{L^2}^{2}\lesssim \left(\sqrt{1-\lambda}  \int_{T_0}^{T_1} \left( 1+\lVert u_1\rVert_{{\dot{H}^1}}+\lVert u_2\rVert_{\dot{H}^1}\right) \left(\lVert (e^{\beta^+|u_1|^2}-1)\rVert_{\dot{H}^1}+\lVert (e^{\beta^+|u_2|^2}-1)\rVert_{\dot{H}^1} \right) d\tau \right)^{\frac{1}{1-\lambda}}.$$
and by letting $\lambda\to 1^- $,we obtain: $$\lim_{\lVert u(T_0)\rVert_{L^2}\to 0}\sup_{t\in [T_0,T_1]}\lVert u(t) \rVert_{L^2}=0.$$
Also by taking the equation in $u$ and performing the inner product with $-\Delta u$, {while using Green's formula and the homogeneous Dirichlet boundary condition}, we have:
\begin{align*}
    \frac{1}{2}\frac{d}{dt}\lVert\nabla u\rVert_{L^2}^2&\lesssim\bigg\langle\left(|\nabla(e^{\beta^+|u_1|^2}-1)|+|\nabla (e^{\beta^+|u_2|^2}-1)| \right)|u|,|\nabla u|\bigg\rangle +\bigg\langle\left((e^{\beta^+|u_1|^2}-1)+ (e^{\beta^+|u_2|^2}-1) \right)|\nabla u|,|\nabla u|\bigg\rangle
\end{align*}
Let $0<\epsilon<\delta$, by applying the Holder's inequality with $1=\frac{1}{2}+\frac{1}{2+\epsilon}+\frac{1}{f(\epsilon)}$ where $f(\epsilon)=\frac{4+2\epsilon}{\epsilon}$ and the Sobolev's embedding with precise constant $H^1\hookrightarrow L^{f(\epsilon)}$, we have:  
\begin{align*}
\frac{d}{dt}\lVert\nabla u\rVert_{L^2}^2
\leq C'\sqrt{f(\epsilon)}\bigg(\lVert \nabla(e^{\beta^+|u_1|^2}-1)\rVert_{L^2}+\lVert \nabla(e^{\beta^+|u_2|^2}-1)\rVert_{L^2}  \bigg)\lVert \nabla u\rVert_{L^2}\lVert \nabla u\rVert_{L^{2+\epsilon}}
\end{align*}
Now, by interpolating $L^{2+\epsilon}$ between $L^2$ and $L^{2+\delta}$, we have $\lVert v\rVert_{L^{2+\epsilon}}\leq \lVert v\rVert_{L^2}^{1-\theta}\lVert v\rVert_{L^{2+\delta}}^{\theta}$ with $\theta=\frac{\epsilon (4+2\delta)}{\delta(4+2\epsilon)}$ and then 
\begin{align*}
\frac{d}{dt}\lVert\nabla u\rVert_{L^2}^2
\leq C'\sqrt{f(\epsilon)}\bigg(\lVert \nabla(e^{\beta^+|u_1|^2}-1)\rVert_{L^2}+\lVert \nabla(e^{\beta^+|u_2|^2}-1)\rVert_{L^2}  \bigg)\lVert \nabla u\rVert_{L^2}^{2-\theta}\lVert \nabla u\rVert_{L^{2+\delta}}^{1+\theta}.
\end{align*}
Since $\epsilon$ is sufficiently close to 0, we have $1+\theta\leq 1+\frac{\delta}{2}$ and by integrating, we obtain
\begin{align*}
  \sup_{t\in [T_0,T_1]}\lVert \nabla u(t) \rVert_{L^2} &\leq \left(C_1\theta \sqrt{f(\epsilon) } \int_{T_0}^{T_1} \left( 1+\lVert u\rVert_{W^{1,2+\delta}}^{1+\frac{\delta}{2}}\right) \left(\lVert (e^{\beta^+|u_1|^2}-1)\rVert_{\dot{H}^1}+\lVert (e^{\beta^+|u_2|^2}-1)\rVert_{\dot{H}^1} \right) d\tau +\lVert \nabla u(T_0)\rVert_{L^2}^{\theta} \right)^{\frac{1}{\theta}} 
\end{align*}
Passing $\lVert \nabla u(T_0)\rVert_{L^2}\to 0$ implies for any $0<\epsilon<\delta$ :
$$\sup_{t\in [T_0,T_1]}\lVert \nabla u(t) \rVert_{L^2} \leq \left(C_1\theta \sqrt{f(\epsilon) } \int_{T_0}^{T_1} \left( 1+\lVert u\rVert_{W^{1,2+\delta}}^{1+\frac{\delta}{2}}\right) \left(\lVert (e^{\beta^+|u_1|^2}-1)\rVert_{\dot{H}^1}+\lVert (e^{\beta^+|u_2|^2}-1)\rVert_{\dot{H}^1} \right) d\tau \right)^{\frac{1}{\theta}} $$
   By letting $\epsilon\to 0$, we have $\theta \to 0$ and $\theta \sqrt{f(\epsilon)}\to 0$, and since $u\in L^{2+\delta}_{loc}W^{1,2+\delta}$ and $\int_0^t\lVert(e^{\beta^+|u|^2}-1)\rVert_{\dot{H}^1}^2\leq C_{3,t}^\omega$ $\mathbb{P}$~a.a according lemma \eqref{uniqueness lemma}, we have:
$$\lim_{\lVert \nabla u(T_0)\rVert_{L^2}\to 0}\sup_{t\in [T_0,T_1]}\lVert \nabla u(t) \rVert_{L^2}=0.$$

\subsection{On the size of data}
Consider the random variables $u^N:\Omega\to L^2$ whose laws are $\mu^N$. We can consider the random variables $\mathcal{K}_N,\, \mathcal{V}_N,\, \mathcal{M}_N :\Omega\to \R$ as
\begin{align}
  \mathcal{K}_N &= \left\langle u_N,P_N\left((1-\Delta)^{-\frac{1}{2}}\left(|(1-\Delta)^{\frac{1}{2}}u_N|^\delta(1-\Delta)^{\frac{1}{2}}u_N    \right)\right)\right\rangle=:\langle u_N,g_{\delta,N}\rangle\\
  \mathcal{V}_N &={Ce^{\gamma C_1\lVert P_N(e^{\beta\abs{u_N}^2}u_N)\rVert_{L^2}^2}}\lVert u_N\rVert_{L^2}^2+{Ce^{\gamma C_1\lVert P_N(e^{\beta\abs{u_N}^2}u_N)\rVert_{L^2}^2}}\sum_{p\in\mathbb{N}}\frac{\beta^p}{p!}\lVert u_N\rVert_{L^{2p+2}}^{2p+2}+C_2\sum_{p\in \mathbb{N}}\frac{\gamma^p}{p!}\lVert u_N\rVert_{L^{2p+2}}^{2p+2}\\
  \mathcal{M}_N &=\mathcal{K}_N+\mathcal{V}_N.
\end{align}
We have already established a.s convergence result of $u_N$ to $u$ as $N\to\infty$. From results established below, we have the convergence of $\mathcal{K}_N,\, \mathcal{V}_N,\, \mathcal{M}_N$ as random variables. Denoting there respective (random variables) limits by $\mathcal{K},\, \mathcal{V},\, \mathcal{M}$, we find the following forms:
\begin{align}
  \mathcal{K} &=\langle u,g_{\delta}\rangle \quad\quad g_\delta:=L^2-\lim_{N\to\infty} g_{\delta,N}\\
  \mathcal{V} &={Ce^{\gamma C_1\lVert e^{\beta\abs{u}^2}u\rVert_{L^2}^2}}\lVert u\rVert_{L^2}^2+{Ce^{\gamma C_1\lVert e^{\beta\abs{u}^2}u\rVert_{L^2}^2}}\sum_{p\in\mathbb{N}}\frac{\beta^p}{p!}\lVert u\rVert_{L^{2p+2}}^{2p+2}+C_2\sum_{p\in \mathbb{N}}\frac{\gamma^p}{p!}\lVert u\rVert_{L^{2p+2}}^{2p+2}\\
  \mathcal{M} &=\mathcal{K}+\mathcal{V}.
\end{align}
Note that although $\mathcal{K}$ could not be explicitly expressed in terms of $u$ due to the lack of compactness in $H^1$, the quantity $\mathcal{V}$ does.
\subsubsection{Nontriviality.}
\begin{proposition} Recall the condition $\gamma>4\beta$. 
 Then we have
\begin{align}
            \Bbb E \left(\mathcal{K}+\mathcal{V}\right)&=\frac{A^0}{2},\label{eq:M}
 \end{align}
 where $\mathcal{K}$ and $\mathcal{V}$ are the modified kinetic and potential energies, defined in \eqref{Modif_Kin_En} and \eqref{Modif_Pot_En}.
\end{proposition}
\begin{proof}
Throughout the proof, $u_N$ is a stationary process distributed by $\mu_N$, in view of the Skorokhod representation. We will rely on the following observations
\begin{itemize}
      \item Due to \eqref{crucial estimate}, $\mu_N$ converges weakly to $\mu$ in $H^{1-\kappa}$ and $L^p$ for any $\kappa>0$ and $p\geq 1$ ;
    \item Their Skorokhod representations satisfy : $u_N$ converges a.s to $u$ in $H^{1-\kappa}$ and in $L^p$ for any $\kappa>0$ and $p\geq 1$.
\end{itemize} 
Let us recall 
\begin{align*}
\mathcal{M}(u_N)&={Ce^{\gamma C_1\lVert P_N(e^{\beta\abs{u_N}^2}u_N)\rVert_{L^2}^2}}\lVert u_N\rVert_{L^2}^2+{Ce^{\gamma C_1\lVert P_N(e^{\beta\abs{u_N}^2}u_N)\rVert_{L^2}^2}}\sum_{p\in\mathbb{N}}\frac{\beta^p}{p!}\lVert u_N\rVert_{L^{2p+2}}^{2p+2}+C_2\sum_{p\in \mathbb{N}}\frac{\gamma^p}{p!}\lVert u_N\rVert_{L^{2p+2}}^{2p+2}\\ &\quad \quad+\left\langle u_N,P_N\left((1-\Delta)^{-\frac{1}{2}}\left(|(1-\Delta)^{\frac{1}{2}}u_N|^\delta(1-\Delta)^{\frac{1}{2}}u_N    \right)\right)\right\rangle\\&=\mathcal{V}_N(u_N)+\mathcal{K}_N(u_N).
\end{align*}
Let us first notice that the process $f_{\beta,N}:=e^{\beta|u_N|^2}u_N$ is compact in $L^2$. Since $\delta<1$ then the condition $\gamma>4\beta$ is taken so that we obtain $2\beta(1+\delta/2)<\gamma$. Therefore,
\begin{align}
    \|\nabla\left(e^{\beta|u_N|^2}u_N\right)\|_{L^{1+\frac{\delta}{2}}}^{1+\frac{\delta}{2}}&\leq C_\delta\int_M|\nabla u_N|^{1+\frac{\delta}{2}}|u_N|^{2+\delta}e^{\beta(1+\delta/2)|u_N|^2} +e^{\beta(1+\delta/2)|u_N|^2}|\nabla u_N|^{1+\frac{\delta}{2}}\\
    &\leq C_{\delta,|M|}+ C^1_\delta\int_M |\nabla u_N|^{2+\delta}+e^{2\beta^+(1+\delta/2)|u_N|^2}(|u_N|^{2}+|\nabla u_N|^2)\,dx \leq C_{\delta,|M|}+ C_\delta^2 \mathcal{E}(u_N).
\end{align}
We invoque \eqref{crucial estimate} to obtain,
\begin{align}
    \mathbb{E} \|\nabla\left(e^{\beta|u_N|^2}u_N\right)\|_{L^{1+\frac{\delta}{2}}}^{1+\frac{\delta}{2}}\leq C_{\delta,|M|}',
\end{align}
where $C_{\delta,|M|}'$ does not depend on $N.$ We use the Rellich-Kondrachov theorem to conclude that $e^{\beta|u_N|^2}u_N$ is compact in $L^{\frac{4+2\delta}{2-\delta}-}\subset L^2$. Since $u_N$ converges to $u$ in $L^2$ $\mathbb{P}$-a.s, we have that for almost all $\omega$, $u_N^\omega$ converges to $u^\omega$ a.s in $x$, up to a subsequence denoted again by $u_N^\omega$. Using the continuity of $z\mapsto e^{\beta|z|^2}z$, we can identify the a.s limit, then the $L^2$-limit, of $e^{\beta|u_N|^2}u_N$ as $e^{\beta|u|^2}u$.\\
 Denoting $g_{\delta,N}=P_N(1-\Delta)^{-\frac{1}{2}}|(1-\Delta)^\frac{1}{2} u_N|^\delta(1-\Delta)^\frac{1}{2} u_N$, we see, using \eqref{crucial estimate}, that 
\begin{align}
    \mathbb{E} \|(1-\Delta)^{\frac{1}{2}} g_{\delta,N}\|_{L^\frac{2+\delta}{1+\delta}}^{\frac{2+\delta}{1+\delta}}\leq \mathbb{E} \|u_N\|_{{W}^{1,2+\delta}}^{2+\delta} \leq \mathbb{E}\mathcal{E}(u_N)\leq C,
\end{align}

which shows compactness of $g_{\delta,N}$ in $L^r$, for $r<1+\frac{4}{\delta}$, by virtue of the  Rellich-Kondrachov theorem. Hence this quantity is compact in $L^2$ since $\delta<1$. We then identify a Skorokhod limit in $L^2$ that we denote $g_\delta$. The quantity $\mathcal{K}(u)$ and $\mathcal{V}(u)$, being the a.s. limits of $\mathcal{K}_N(u_N)$ and $\mathcal{V}_N(u_N)$, 
Fatou's Lemma allows to show the easy side of the needed equality : 
\begin{align}
    \mathbb{E}\left(\mathcal{K}+\mathcal{V}\right)\leq \frac{A^0}{2}. 
\end{align}
The reverse inequality is more difficult. 
 Let $M\leq N$, $R>0$, we have for $\lambda=\beta ~\text{or}~ \gamma$. Using the estimate \eqref{Est_E.Cal_E}, we have that
\begin{align*}
&\mathbb{E} \sum_{p \in \mathbb{N}} \frac{\lambda^p}{p!}\|u_N\|_{L^{2p+2}}^{2p+2}
= \mathbb{E} \sum_{p \in \mathbb{N}}\frac{\lambda^p}{p!} {\||u_N|^{p+1}\|_{L^2}^2}
= \mathbb{E} \sum_{p \in \mathbb{N}} \frac{\lambda^p}{p!}\left( \|P_{ M} |u_N|^{p+1} \|_{L^2}^2
+ \|(1-P_{M}) |u_N|^{p+1} \|_{L^2}^2\right) \\
 &\leq
 \mathbb{E} \sum_{p \in \mathbb{N}} \frac{\lambda^p}{p!}\|P_{ M} |u_N|^{p+1}\|_{L^2}^2 \left(1_{\{\lVert u_N\rVert_{L^2}\leq R\}}+1_{\{\lVert u_N\rVert_{L^2}\geq R\}}  \frac{\|u_N\|_{L^2}^2}{\lVert u_N\rVert_{L^2}^2}  \right)
+M^{-1}\mathbb{E}\sum_{p\in \mathbb{N}}\frac{\lambda^p}{p!}\langle |\nabla u_N|^2,|u_N|^{2p}\rangle \\
&\leq \mathbb{E}1_{\{\lVert u_N\rVert_{L^2}\leq R\}} \sum_{p \in \mathbb{N}} \frac{\lambda^p}{p!}\|P_{ M} |u_N|^{p+1}\|_{L^2}^2  
+ R^{-2} \mathbb{E}1_{\{\lVert u_N\rVert_{L^2}\geq R\}}  E(u_N) \sum_{p \in \mathbb{N}} \frac{\lambda^p}{p!}\|P_{ M} |u_N|^{p+1}\|_{L^2}^2+M^{-1}C_6 \\
&\leq \mathbb{E}1_{\{\lVert u_N\rVert_{L^2}\leq R\}} \sum_{p \in \mathbb{N}} \frac{\lambda^p}{p!}\|P_{ M} |u_N|^{p+1}\|_{L^2}^2  
+ R^{-2} C_8+M^{-1}C_6.
\end{align*}
Let $A={\{\lVert u_N\rVert_{L^2}+\lVert P_N(e^{\beta|u_N|^2}u_N)\rVert_{L^2}\leq R\}}$, recall the estimate \eqref{Est_E.Cal_E}. We also have
\begin{align*}
    \mathbb{E} ~C&e^{\gamma C_1\lVert P_N(e^{\beta\abs{u_N}^2}u_N)\rVert_{L^2}^2}\sum_{p \in \mathbb{N}} \frac{\beta^p}{p!}\|u_N\|_{L^{2p+2}}^{2p+2}\leq M^{-1}C_6+ \mathbb{E}~Ce^{\gamma C_1\lVert P_N(e^{\beta\abs{u_N}^2}u_N)\rVert_{L^2}^2}1_{A} \sum_{p \in \mathbb{N}} \frac{\lambda^p}{p!}\|P_{ M} |u_N|^{p+1}\|_{L^2}^2\\
    & \quad\quad\quad\quad\quad\quad+\mathbb{E}~Ce^{\gamma C_1\lVert P_N(e^{\beta\abs{u_N}^2}u_N)\rVert_{L^2}^2}\left(1_{A}^c \frac{\lVert u_N\rVert_{L^2}^2+\lVert P_N(e^{\beta|u_N|^2}u_N)\rVert_{L^2}^2}{\lVert u_N\rVert_{L^2}^2+\lVert P_N(e^{\beta|u_N|^2}u_N)\rVert_{L^2}^2}\right)\sum_{p \in \mathbb{N}} \frac{\lambda^p}{p!}\|P_{ M} |u_N|^{p+1}\|_{L^2}^2\\
    &\quad\quad\quad\quad\quad\quad\quad\quad\quad \quad\quad\quad\leq M^{-1}C_6+4R^{-2}C_8+\mathbb{E}~Ce^{\gamma C_1\lVert P_N(e^{\beta\abs{u_N}^2}u_N)\rVert_{L^2}^2}1_{A} \sum_{p \in \mathbb{N}} \frac{\lambda^p}{p!}\|P_{ M} |u_N|^{p+1}\|_{L^2}^2.
\end{align*}
Summarizing,
\begin{align*}
     \mathbb{E}\mathcal{V}_N(u_N)&\leq  
    \mathbb{E} ~Ce^{\gamma C_1\lVert P_N(e^{\beta\abs{u_N}^2}u_N)\rVert_{L^2}^2}\left(\lVert u_N\rVert_{L^2}^2+\sum_{p \in \mathbb{N}} \frac{\beta^p}{p!}\|P_{ M} |u_N|^{p+1}\|_{L^2}^2\right)1_{\{\lVert u_N\rVert_{L^2}+\lVert P_N(e^{\beta|u_N|^2}u_N)\rVert_{L^2}\leq R\}}\\
    &\quad\quad+C_2\mathbb{E}\left(\sum_{p\in\mathbb{N}}\frac{\gamma^p}{p!}\|P_{ M} |u_N|^{p+1}\|_{L^2}^2\right)1_{\{\lVert u_N\rVert_{L^2}\leq R\}}+C'(R^{-2}+M^{-1}).
\end{align*}
 Now, recall that $g_{\delta,N}=P_N(1-\Delta)^{-\frac{1}{2}}|(1-\Delta)^\frac{1}{2} u_N|^\delta(1-\Delta)^\frac{1}{2} u_N.$ Let $r>1$ such that $r(1+\delta)<2+\delta$, we have:
 \begin{align*}
   |\langle u_N,g_{\delta,N}\rangle|^r&\lesssim \lVert u_N\rVert_{L^2}^r\lVert |(1-\Delta)^{\frac{1}{2}}u_N|^\delta(1-\Delta)^{\frac{1}{2}} u_N\rVert^r_{H^{-1}}  \\
   &\lesssim  \lVert u_N\rVert_{L^2}^r\lVert |(1-\Delta)^{\frac{1}{2}}u_N|^\delta(1-\Delta)^{\frac{1}{2}} u_N\rVert_{L^{1+\frac{1}{1+\delta}}}^r \\
   &\lesssim  \lVert u_N\rVert_{L^2}^r\lVert (1-\Delta)^{\frac{1}{2}}u_N\rVert_{L^{2+\delta}}^{(1+\delta)r}\\
   &\lesssim  \sum_{p\in \mathbb{N}}\frac{\beta^{p}}{p!}\lVert u_N\rVert_{L^{2p+2}}^{2p+2}+\lVert (1-\Delta)^{\frac{1}{2}}u_N\rVert_{L^{2+\delta}}^{2+\delta}
 \end{align*}
 Recall that $\mathbb{E}\sum_{p\in \mathbb{N}}\frac{\beta^{p}}{p!}\lVert u_N\rVert_{L^{2p+2}}^{2p+2}+\lVert (1-\Delta)^{\frac{1}{2}}u_N\rVert_{L^{2+\delta}}^{2+\delta} \leq C$. Denote $B=\{\|u_N\|_{L^2}+\|g_{\delta,N}\|_{L^2}\leq R\}$. Using the Markov inequality, we arrive at
 \begin{align}
     \Bbb E \langle u, g_{\delta,N}\rangle &=   \Bbb E \langle u, g_{\delta,N}\rangle 1_{B}+\text{Complement}\\
     &\leq \Bbb E \langle u, g_{\delta,N}\rangle 1_{B}+R^{-r}C.
 \end{align}
Overall, we obtain

\begin{align*}
     \mathbb{E}\mathcal{V}_N(u_N) &+\Bbb E\langle u_N, g_{\delta,N}\rangle =\\
     \frac{A^0_N}{2}&\leq  
    \mathbb{E} ~Ce^{\gamma C_1\lVert P_N(e^{\beta\abs{u_N}^2}u_N)\rVert_{L^2}^2}\left(\lVert u_N\rVert_{L^2}^2+\sum_{p \in \mathbb{N}} \frac{\beta^p}{p!}\|P_{ M} |u_N|^{p+1}\|_{L^2}^2\right)1_{A}\\
    &\quad\quad+C_2\mathbb{E}\left(\sum_{p\in\mathbb{N}}\frac{\gamma^p}{p!}\|P_{ M} |u_N|^{p+1}\|_{L^2}^2\right)1_{\{\lVert u_N\rVert_{L^2}\leq R\}}+\Bbb E\langle u_N,g_{\delta,N}\rangle 1_B+ C'(R^{-2}+M^{-1})+CR^{-r}.
\end{align*}

 We can pass to the limit $N\to\infty$, then $M\to\infty$ and $R\to\infty$ to obtain the desired reverse inequality, and then the following equality  $$\int_{L^2}\bigg({Ce^{\gamma C_1\|e^{\beta|u|^2}u\|_{L^2}^2}}\lVert u\rVert_{L^2}^2+{Ce^{\gamma C_1\|e^{\beta|u|^2}u\|_{L^2}^2}}\sum_{p\in\mathbb{N}}\frac{\beta^p}{p!}\lVert u\rVert_{L^{2p+2}}^{2p+2}+C_2\sum_{p\in \mathbb{N}}\frac{\gamma^p}{p!}\lVert u\rVert_{L^{2p+2}}^{2p+2}+\langle u, g_\delta\rangle\bigg)\mu(du)=\frac{A^0}{2},$$
 where $g_\delta$ is an appropriate $\Bbb P$-a.s. limit in $L^2$ of $g_{\delta,N}$, as it can easily be seen that $g_{\delta,N}$ is bounded in $L^{\frac{2+\delta}{2}}(\Omega,W^{1,\frac{2+\delta}{1+\delta}})$ (one can then combine the Prokhorov and Skorokhod  theorems to satisfy the claim).
\end{proof}

\subsubsection{Scaled schemes, redifinition of the measure and size of the data}
We start from the identity established above, that is, 
 \begin{align}\mathbb{E}\mathcal{M}=\frac{A^0}{2}.
\end{align}
Recall that the number $A^0$ was defined from the $L^2$-expansion of the Brownian motion involved in the fluctuation-dissipation scheme as
\begin{align}
    A^0=\lim_{N\to\infty}\sum_{j\leq N}|a_j|^2.
\end{align}
Now, considering scaled versions of this scheme in which the coefficients $a_j$ are all multiplied by a positive number $\lambda$, we will ultimately obtain a family of measures $\{\mu^\lambda\}$ and corresponding modified energies $\{\mathcal{M}^\lambda\}$ satisfying similar properties proven above, with respective identities
\begin{align}
    \mathbb{E}\mathcal{M}^\lambda=\frac{A^0_\lambda}{2}=\frac{\lambda^2A^0}{2}.
\end{align}
Redefine $\mu$ as the following convex combination
\begin{align}
    \mu=\sum_{n=1}^\infty\frac{\mu^{\lambda_n}}{2^n},
\end{align}
for a sequence $\lambda_n\to\infty$ as $n\to\infty$. We see easily that $\mu$ define a stationary process satisfying the equation, and by construction we have the following
\begin{remark}\label{Rmk_size}
  For any $R>0$, for $\lambda_n>>1$, $\mathbb{P}(\mathcal{M}^{\lambda_n}>R)>0$.  
\end{remark}
\section{Further properties of the constructed measure $\mu$}
In this section, we establish the following three properties regarding the non-degeneracy of the measure. We also define a statiscal ensemble (of full measure) invariant under the flow. We assume that the coefficients $(a_m)$ entering the definition of the noise are nonzero and decreasing.

\begin{enumerate}
    \item {\bfseries Property 1: Large modified potential energy.} One drawback of the modified energy is that its kinetic part is not explicit in $u$. Therefore, it does not directly transfer "large size property" of Remark \ref{Rmk_size} above to $u$. Exploiting the fact that the potential part is explicit in $u$, we perform an argument that solves this issue.
    \item {\bfseries Property 2: Non atomicity and absolute continuity of the mass functional.} This property rules atoms and, more generally, positive charge of level sets of the mass.
    \item {\bfseries Property 3: Dimensionality of at least two.} This, in particular, rules out any one dimensional family of solutions.
\end{enumerate}

\subsection{Large (explicit) modified potential energy}
\begin{lemma}
    For almost all $\omega\in\Omega$, there is $N^\omega$ such that
    \begin{align}
        \mathcal{V}^\omega\geq -1+\frac{\mathcal{K}^\omega-1}{C(N^\omega)},
    \end{align}
    where $C(\cdot)$ is a non random function such that $\lim_{x\to\infty}C(x)=\infty$.
\end{lemma}
\begin{proof}
 On the one hand, write, using finite-dimensionality, that for any $N$
    \begin{align}
        \mathcal{K}^\omega_N=\langle u^N,(1-\Delta)^{-\frac{1}{2}}|(1-\Delta)^{\frac{1}{2}}u^N|^\delta(1-\Delta)^{\frac{1}{2}}u^N\rangle\leq C_1(N)\|u^N\|_{L^2}^2\|u^N\|_{L^2}^\delta \leq C(N)\mathcal{V}^\omega_N.
    \end{align}
    On the other hand, by almost sure convergence of $\mathcal{K}^\omega_N$ to $\mathcal{K}^\omega$, there is $N_0^\omega$ such that for all $N\geq N_0^\omega$,
    \begin{align}
        \mathcal{K}^\omega_N\geq \mathcal{K}^\omega-1.
    \end{align}
    Then, for $N\geq N_0^\omega$,
     \begin{align}
        C(N)\mathcal{V}^\omega_N\geq \mathcal{K}^\omega-1.
    \end{align}
    Thanks to the a.s. convergence $\mathcal{K}^\omega_N\to\mathcal{K}^\omega$ as $N\to\infty$, we can increase $N_0^\omega$ if needed, to have $\mathcal{V}^\omega_N\leq \mathcal{V}^\omega+1$ when $N\geq N_0^\omega$. Whence follows the lemma.
\end{proof}
\begin{proposition}
   The support of $\mu$ contains large $\mathcal{V}(u)$ data.
\end{proposition}
\begin{proof}
    Split $\Omega$ into : $\Omega_1=\{\mathcal{V}>\mathcal{K}\}$ and $\Omega_2=\{\mathcal{K}\geq\mathcal{V}\}$. We already have that $\PP(\mathcal{M}\geq R)>0$ for all $R>0$ (Remark \ref{Rmk_size}).
    \begin{enumerate}
        \item Let $\omega\in\Omega_1\cap\{\mathcal{M}\geq R\}$, we have that $2\mathcal{V}^\omega\geq R$.
        \item Let $\omega\in\Omega_2\cap\{\mathcal{M}\geq R\}$, we have that $2\mathcal{K}^\omega\geq R$, and then
        \begin{align}
             \mathcal{V}^\omega\geq -1+\frac{R-2}{2C(N^\omega)}\to\infty,\quad R\to\infty.
        \end{align}
    \end{enumerate}
\end{proof}
\subsection{Absolute continuity property}
\begin{proposition}
    \begin{enumerate}
        \item The measure $\mu$ has no atom at $0$ :
        \begin{align}
            \mu(\|u\|_{L^2}\leq \delta)\leq C\delta \quad\quad \text{as $\delta\to 0$};
        \end{align}
        \item The distribution of the mass has a density with respect to the Lebesgue measure on $\R$.
    \end{enumerate}
\end{proposition}
\begin{proof}
We follow a local time argument introduced by Shirikyan \cite{Shirikyan2011}. We derive the following identity (see \cite{Shirikyan2011} for details) for $u:=u_N^\alpha$, stationary solution to \eqref{eq3},
\begin{align}
    &\mathbb{E}\int_{\Gamma}1_{(a,\infty)} (g(\|u\|_{L^2}^2))\left\{g'(\|u\|_{L^2}^2)\left(\frac{A_0^N}{2}-\mathcal{M}_N(u)\right)+g''(\|u\|_{L^2}^2)\sum_{|m|\leq N}|a_m|^2|(u,e_m)|^2\right\}da \label{Identity_Shir1}\\
    &+\sum_{|m|\leq N}a_m^2\mathbb{E}\left(1_\Gamma(g(\|u\|_{L^2}^2))|g'(\|u\|_{L^2}^2)|^2|(u,e_m)|^2\right)=0.\nonumber
\end{align}
We then obtain
\begin{align}
    \mathbb{E}\int_{\Gamma}1_{(a,\infty)} (g(\|u\|_{L^2}^2))\left\{g'(\|u\|_{L^2}^2)\left(\frac{A_0^N}{2}-\mathcal{M}_N(u)\right)+g''(\|u\|_{L^2}^2)\sum_{|m|\leq N}|a_m|^2|(u,e_m)|^2\right\}da
   \leq 0.\label{Id_Shir2}
\end{align}
For $a>b>0$, $\Gamma=[a,b]$, choose $g$ in $C^2(\R)$ such that $g(x)=\sqrt{x}$ for $x\geq a$ and $g(x)=0$ for $x\leq 0$. We obtain
\begin{align}
    \mathbb{E}\int_a^b1_{(a,\infty)} (\|u\|_{L^2})\left\{ \frac{A_0^N-2\mathcal{M}_N(u)}{4\|u\|_{L^2}}-\frac{1}{4\|u\|_{L^2}^3}\sum_{|m|\leq N}|a_m|^2|(u,e_m)|^2\right\}da\leq 0.
\end{align}
We find tha
\begin{align}
 \mathbb{E}\int_a^b \frac{1_{(a,\infty)} (\|u\|_{L^2})}{\|u\|_{L^2}^3}\left(A_N\|u\|_{L^2}^2-\sum_{|m|\leq N}a_m^2|(u,e_m)|^2\right)da \leq 2(b-a)\mathbb{E}\left(\frac{\mathcal{M}_N(u)}{\|u\|_{L^2}}\right).
\end{align}
Since all $a_m$ are nonzero and decreasing, set $\theta^N=A_0^N-|a_0|^2$. Note that $|a_1|^2<\theta^N<A_0$ for all $N\geq 2$. We obtain the lower bound
\begin{align}
   A_N\|u\|_{L^2}^2-\sum_{|m|\leq N}a_m^2|(u,e_m)|^2\geq \theta^N\|u\|_{L^2}^2\geq |a_1|^2\|u\|_{L^2}^2. 
\end{align}
We arrive at
\begin{align}
    \mathbb{E}\int_a^b \frac{1_{(a,\infty)} (\|u\|_{L^2})}{\|u\|_{L^2}}da \leq \frac{2}{|a_1|^2}(b-a)\mathbb{E}\left(\frac{\mathcal{M}_N(u)}{\|u\|_{L^2}}\right).
\end{align}
Let $\delta>0$, we find
\begin{align}
    \delta^{-1}\mathbb{E}\int_a^b 1_{(a,\delta)} (\|u\|_{L^2})da\leq \mathbb{E}\int_a^b \frac{1_{(a,\delta)} (\|u\|_{L^2})}{\|u\|_{L^2}}da \leq \frac{2}{|a_1|^2}(b-a)\mathbb{E}\left(\frac{\mathcal{M}_N(u)}{\|u\|_{L^2}}\right).
\end{align}
Hence
\begin{align}
    \frac{1}{b-a}\int_a^b \mathbb{P}(a<\|u\|_{L^2}<\delta)da \leq \frac{2\delta}{|a_1|^2}\mathbb{E}\left(\frac{\mathcal{M}_N(u)}{\|u\|_{L^2}}\right).
\end{align}
Pass to the limit $a\to 0$, then $b\to 0$, we obtain
\begin{align}
    \mathbb{P}(0<\|u\|_{L^2}<\delta) \leq \frac{2\delta}{|a_1|^2}\mathbb{E}\left(\frac{\mathcal{M}_N(u)}{\|u\|_{L^2}}\right).
\end{align}
Now, due to the nondegeneracy of the noise, a routine argument show that the noise is non degenerate $\PP(u=0)=0$. We arrive at 
\begin{align}
    \mathbb{P}(0\leq\|u\|_{L^2}<\delta) \leq \frac{2\delta}{|a_1|^2}\mathbb{E}\left(\frac{\mathcal{M}_N(u)}{\|u\|_{L^2}}\right).
\end{align}
Invoking Lemma \ref{lem:MN_estimate} (see below), we arrive at the claimed non atomicity at $0$ : 
\begin{align}
    \mathbb{P}(0\leq\|u\|_{L^2}<\delta) \leq C\delta.
\end{align}
We now turn to the proof of absolute continuity of $\|u\|_{L^2}$. Let's take $g(x)=x$ in the identity \eqref{Identity_Shir1} :
\begin{align}
    \mathbb{E}\left(1_\Gamma(\|u\|_{L^2}^2)\sum_{|m|\leq N}a_m^2|(u,e_m)|^2\right)\leq A_0^N\ell(\Gamma)\leq A_0\ell(\Gamma).
\end{align}
Now, let $M<N$
\begin{align}
    \sum_{|m|\leq N}a_m^2|(u,e_m)|^2\geq\sum_{|m|\leq M}a_m^2|(u,e_m)|^2 &\geq a_M^2\sum_{|m|\leq M}|(u,e_m)|^2\\
    &=a_M^2\left(\sum_{|m|\leq N}|(u,e_m)|^2-\sum_{M<|m|\leq N}|(u,e_m)|^2\right)\\
    &\geq a_M^2\left(\|u\|_{L^2}^2-\lambda_M^{-1}\|u\|_{H^1}^2\right).
\end{align}
Consider the set $S_\epsilon=\{\|u\|_{L^2}\geq\epsilon,\ \ \|u\|_{H^1}\leq \epsilon^{-1}\}$. Therefore, if $\mu$ is conditioned on the set $S_\epsilon$, we obtain
\begin{align}
    a_M^2(\epsilon-\lambda_M^{-1}\epsilon^{-1})\mu(\|u\|_{L^2}^2\in\Gamma)\leq A_0\ell(\Gamma).
\end{align}
Now, for all $\epsilon>0$, we can find $N>>1$ such that there is $M<N$ satisfying $\lambda_M^{-1}\leq \frac{\epsilon^2}{2}$; so that $\epsilon-\lambda_M^{-1}\epsilon^{-1}\geq \frac{\epsilon}{2}$. We then obtain that : for all $\epsilon>0$, there are $N_0$ and $M$ such that for all $N>N_0$, 
\begin{align}
   \mu_N^\alpha(\{\|u\|_{L^2}^2\in\Gamma\}\cap S_\epsilon)\leq \frac{2A_0}{\epsilon a_M^2}\ell(\Gamma).
\end{align}
Passing to the limits $\alpha\to 0,\ \ N\to\infty$:
\begin{align}
   \mu_N^\alpha(\{\|u\|_{L^2}^2\in\Gamma\}\cap S_\epsilon)\leq \frac{2A_0}{\epsilon a_M^2}\ell(\Gamma).
\end{align}
On the other hand, using the non-atomicity bound above and Chebyshev inequality,  we have that
\begin{align}
    \mu(S_\epsilon^c)\leq \mu(\|u\|_{L^2}\leq \epsilon)+\mu(\|u\|_{H^1}\geq \epsilon^{-1})\leq C\epsilon.
\end{align}
Overall, for all $\epsilon>0$
\begin{align}
   \mu(\|u\|_{L^2}^2\in\Gamma)\leq C\epsilon+\frac{2A_0}{\epsilon}\ell(\Gamma).
\end{align}
This shows the absolute continuity property w.r.t. Lebesgue measure on $\R$.
\end{proof}
\begin{lemma}\label{lem:MN_estimate}
Let $u:=u^\alpha_N$ be a stationary solution to \eqref{eq3}. We have the following estimate :
\begin{equation}
\mathbb{E}\left[\frac{\mathcal{M}_N(u)}{\|u\|_{L^2}}\right]
\leq
C_1\,\mathbb{E}\mathcal{E}_0(u)+C_2\leq C,
\end{equation}
where $C$ is independent of $(\alpha,N)$.
\end{lemma}
\begin{proof}
    Let us recall that \begin{align*} 
   \mathcal{M}_N(u_N)&={Ce^{\gamma C_1\lVert P_N(e^{\beta\abs{u_N}^2}u_N)\rVert_{L^2}^2}}\lVert u_N\rVert_{L^2}^2+{Ce^{\gamma C_1\lVert P_N(e^{\beta\abs{u_N}^2}u_N)\rVert_{L^2}^2}}\sum_{p\in\mathbb{N}}\frac{\beta^p}{p!}\lVert u_N\rVert_{L^{2p+2}}^{2p+2}+C_2\sum_{p\in \mathbb{N}}\frac{\gamma^p}{p!}\lVert u_N\rVert_{L^{2p+2}}^{2p+2}\\ &\quad \quad+\left\langle u_N,P_N\left((1-\Delta)^{-\frac{1}{2}}\left(|(1-\Delta)^{\frac{1}{2}}u_N|^\delta(1-\Delta)^{\frac{1}{2}}u_N    \right)\right)\right\rangle.
\end{align*}
When we divide by $\lVert u_N\rVert_{L^2}^2$, the first term will produce a term ${Ce^{\gamma C_1\lVert P_N(e^{\beta\abs{u_N}^2}u_N)\rVert_{L^2}^2}}\lVert u_N\rVert_{L^2}$ which can be control by using the Legendre transformation with the strict convex function such that $\Phi_\lambda(e^{\lambda x^2})=e^{\lambda x^2}x^2$. Indeed the generalized Young inequality yields:
\begin{align*}
    {e^{\gamma C_1\lVert P_N(e^{\beta\abs{u_N}^2}u_N)\rVert_{L^2}^2}}\lVert u_N\rVert_{L^2}&\leq \Phi_\lambda({e^{\gamma C_1\lVert P_N(e^{\beta\abs{u_N}^2}u_N)\rVert_{L^2}^2}})+\Phi_\lambda^*(\lVert u_N\rVert_{L^2})\\
    &\leq {e^{\gamma C_1\lVert P_N(e^{\beta\abs{u_N}^2}u_N)\rVert_{L^2}^2}}\lVert P_N(e^{\beta\abs{u_N}^2}u_N)\rVert_{L^2}^2+\Phi_\lambda^*(\lVert u_N\rVert_{L^2})
\end{align*}
 Let us consider
\[
\Phi_\lambda(y)=\frac{1}{\lambda}y\log y,\qquad y\geq 1,
\]
where $\lambda>0$ will be chosen later. Notice that
\[
\Phi_\lambda(e^{\lambda x^2})
=
\frac{1}{\lambda}e^{\lambda x^2}\lambda x^2
=
x^2e^{\lambda x^2}.
\]
We first compute the Legendre transform of $\Phi_\lambda$. By definition,
\[
\Phi_\lambda^*(s)
=
\sup_{y\geq 1}
\left\{
sy-\frac{1}{\lambda}y\log y
\right\}.
\]
Let
\[
F_s(y)=sy-\frac{1}{\lambda}y\log y .
\]
A direct computation gives
\[
F_s'(y)
=
s-\frac{1+\log y}{\lambda}.
\]
Hence, the critical point satisfies
\[
\log y=\lambda s-1,
\]
which gives
\[
y_s=e^{\lambda s-1}.
\]
For $s\geq \lambda^{-1}$, we have $y_s\geq1$, and therefore the supremum is
attained at $y_s$. Consequently,
\[
\begin{aligned}
\Phi_\lambda^*(s)
&=
s e^{\lambda s-1}
-\frac{1}{\lambda}e^{\lambda s-1}(\lambda s-1)\\
&=
\frac{1}{\lambda}e^{\lambda s-1}.
\end{aligned}
\]
For $0\leq s<\lambda^{-1}$, the maximum is attained at the boundary point
$y=1$, which yields $\Phi_\lambda^*(s)=s$. Therefore,
\[
\Phi_\lambda^*(s)
=
\begin{cases}
s, & 0\leq s\leq \lambda^{-1},\\[2mm]
\dfrac{1}{\lambda}e^{\lambda s-1},
& s\geq \lambda^{-1}.
\end{cases}
\]
In particular, using the elementary inequality
\[
\lambda s\leq \varepsilon s^2+\frac{\lambda^2}{4\varepsilon},
\qquad \varepsilon>0,
\]
we obtain
\[
\Phi_\lambda^*(s)
\lesssim_{\lambda,\varepsilon}
e^{\varepsilon s^2}.
\]
We now choose $\varepsilon>0$ such that
\[
\varepsilon<\gamma .
\]
Hence,
\[
\Phi_\lambda^*(\|u_N\|_{L^2})
\lesssim_{\lambda,\varepsilon}
e^{\varepsilon\|u_N\|_{L^2}^{2}}.
\]

It remains to relate this quantity to the potential energy. Expanding the
exponential, we have
\[
e^{\varepsilon\|u_N\|_{L^2}^{2}}
=
\sum_{p=0}^{\infty}
\frac{\varepsilon^p}{p!}
\|u_N\|_{L^2}^{2p}.
\]
Since $M$ is compact, Hölder's inequality gives, for every $p\geq0$,
\[
\|u_N\|_{L^2}^{2p+2}
\leq
|M|^p
\|u_N\|_{L^{2p+2}}^{2p+2}.
\]
Therefore,
\[
\begin{aligned}
e^{\varepsilon\|u_N\|_{L^2}^{2}}
&=
\sum_{p=0}^{\infty}
\frac{\varepsilon^p}{p!}
\|u_N\|_{L^2}^{2p}\lesssim e^\varepsilon+\sum_{p=0}^\infty\frac{(\varepsilon |M|)^p}{p!}\lVert u_N\rVert_{L^{2p+2}}^{2p+2}.
\end{aligned}
\]
Choosing $\varepsilon$ sufficiently small so that
\[
\varepsilon |M|\leq \gamma ,
\]
we finally obtain
\[
{e^{\gamma C_1\lVert P_N(e^{\beta\abs{u_N}^2}u_N)\rVert_{L^2}^2}}\lVert u_N\rVert_{L^2}
\lesssim {e^{\gamma C_1\lVert P_N(e^{\beta\abs{u_N}^2}u_N)\rVert_{L^2}^2}}\lVert P_N(e^{\beta\abs{u_N}^2}u_N)\rVert_{L^2}^2
+\sum_{p=0}^{\infty}
\frac{\gamma^p}{p!}
\|u_N\|_{L^{2p+2}}^{2p+2}\lesssim \mathcal{E}_0(u).
\]
for the second term, we have to control the term 
\begin{align*}
    {Ce^{\gamma C_1\lVert P_N(e^{\beta\abs{u_N}^2}u_N)\rVert_{L^2}^2}}{\sum_{p\in\mathbb{N}}\frac{\beta^p}{p! \lVert u_N\rVert_{L^2} }\lVert u_N\rVert_{L^{2p+2}}^{2p+2}}
\end{align*}
Let us first remark that 
\begin{align*}
    \lVert u_N\rVert_{L^{2p+2}}\lesssim\lVert u\rVert_{L^2}^\frac{1}{2p+2}\lVert u\rVert_{L^{2(2p+1)}}^{1-\frac{1}{2p+2}}
\end{align*}
Thus 
\begin{align*}
    \lVert u_N\rVert_{L^{2p+2}}^{2p+2}&\lesssim\lVert u_N\rVert_{L^2}\lVert u_N\rVert_{L^{2(2p+1)}}^{2p+1}=\lVert u_N \rVert_{L^2}\lVert u^{p+1}\rVert_{L^{\frac{2(2p+1)}{p+1}}}^\frac{2p+1}{p+1}\\
    &\lesssim \lVert u_N\rVert_{L^2}\left(\frac{2(2p+1)}{p+1}\right)^\frac{2p+1}{2p+2}\lVert \nabla u^{p+1}\rVert_{L^2}^\frac{2p+1}{p+1}\\
    &\lesssim  \lVert u_N \rVert_{L^2}\left( (\sqrt{2})^{4p+2}+\lVert \nabla u^{p+1}\rVert_{L^2}^2 \right)
\end{align*}
Therefore 
\begin{align*}
  {e^{\gamma C_1\lVert P_N(e^{\beta\abs{u_N}^2}u_N)\rVert_{L^2}^2}}{\sum_{p\in\mathbb{N}}\frac{\beta^p}{p! \lVert u_N\rVert_{L^2} }\lVert u_N\rVert_{L^{2p+2}}^{2p+2}}&\lesssim \Phi({e^{\gamma C_1\lVert P_N(e^{\beta\abs{u_N}^2}u_N)\rVert_{L^2}^2}})+\Phi^*(1) + {e^{\gamma C_1\lVert P_N(e^{\beta\abs{u_N}^2}u_N)\rVert_{L^2}^2}} \lVert \nabla u_N^{p+1}\rVert_{L^2}^2\\
  &\lesssim \Phi^*(1)+ {e^{\gamma C_1\lVert P_N(e^{\beta\abs{u_N}^2}u_N)\rVert_{L^2}^2}}\lVert P_N(e^{\beta\abs{u_N}^2}u_N)\rVert_{L^2}^2\\
  &\qquad+{e^{\gamma C_1\lVert P_N(e^{\beta\abs{u_N}^2}u_N)\rVert_{L^2}^2}} \lVert \nabla u_N^{p+1}\rVert_{L^2}^2\\
  &\lesssim \Phi^*(1)+\mathcal{E}_0(u)
\end{align*}
Also for the third term, we have
\begin{align*}
    {\sum_{p\in\mathbb{N}}\frac{\gamma^p}{p! \lVert u_N\rVert_{L^2} }\lVert u_N\rVert_{L^{2p+2}}^{2p+2}} \lesssim C+\sum_{p\in \mathbb{N}}\frac{\gamma^p}{p!}\lVert \nabla u_N^{p+1}\rVert_{L^2}^2\lesssim C+\mathcal{E}_0(u)
\end{align*}
For the last term, since we are in two-space dimension, $L^{1+\frac{1-\delta}{1+\delta}}\hookrightarrow H^{-1}$, We have:
\begin{align*}
  \left\langle u_N,P_N\left((1-\Delta)^{-\frac{1}{2}}\left(|(1-\Delta)^{\frac{1}{2}}u_N|^\delta(1-\Delta)^{\frac{1}{2}}u_N    \right)\right)\right\rangle&\leq \lVert u_N\rVert_{L^2}\lVert u_N\rVert_{H^1}^{1+\delta} \leq \mathcal{E}_0(u)+C. 
\end{align*}
\end{proof}

\subsection{Dimensionality of the measure }
\begin{proposition}\label{prop:dimension_measure}
 The measure $\mu$ is of at least two-dimension in the sense that
\[
\mu(K)=0
\]
for every compact set $K\subset H^{1}(\mathbb{T})$ 
with Hausdorff dimension  $\dim_{H}(K)<2.$
\end{proposition}

Let us recall the Krylov estimate for stochastic processes (see Section $7.9$ \cite{KuksinShirikyan2012}). Consider the following $d$-dimensional process :
\begin{align}
    dx =b_tdt+\sigma_tdB_t
\end{align}
where $\sigma$ is a $d\times m$-matrix and $B$ is a $m$-dimensional Brownian motion. For all bounded measurable function $f: \R^d\to\R$, we have 
\begin{align}
    \mathbb{E}\int_0^1f(x_t)(\text{det}(\sigma\sigma^T))^\frac{1}{d}dt\leq C_d\|f\|_{L^d}\mathbb{E}\int_0^1|b_t|dt.\label{Krylov_estimate}
\end{align}
We will work on the Galerkin fluctuation-dissipation problem \eqref{eq3} in order to obtain uniform bounds in $(\alpha,N)$ on the relevant estimate. Then, the passage to the limits follows usual arguments. To simplify the notation, we write $u:=u^\alpha_N$.
We define the set 
\begin{align}
    O=\{u\in L^2\ :\ (i\Delta u,|u|^2u)=0\}.
\end{align}
Consider the following vector
\begin{align}
    X(u)=\begin{cases}
        \left(\begin{array}{c}
     M(u)\\ 
     F(u)
\end{array}\right) \quad \text{if $u\in O$} \\ \\
\left(\begin{array}{c}
     M(u)\\
     E(u)
\end{array}\right) \quad \text{if $u\in O^c$}
    \end{cases}
\end{align}
where $M(u)$ and $E(u)$ are the mass and energy functionals, and 
\begin{align}
    F(u)=\frac{1}{4}\int_M |u|^4dx.
\end{align}
Set 
\begin{align}
    \mathcal{F}(u)=\langle\nabla_{\bar u}F,\mathcal{L}u\rangle.
\end{align}
We note that $\mathcal{F}(u)\lesssim \mathcal{E}(u)$.

We wish to apply the Krylov estimate to the vector $X=X(u)$. We will split it following $O$ and $O^c$.

\begin{itemize}
    \item \textbf{On the set $O$}. First remark that 
\begin{align}
    \langle\nabla_{\bar u}F,i(\Delta -e^{\beta|u|^2}+1)u\rangle_{|_{O}}=\langle |u|^2u,i(\Delta -e^{\beta|u|^2}+1)u\rangle_{|_{O}}=0.
\end{align}
Hence
\begin{align}
    dX &=\alpha\left(\begin{array}{c}
        -\mathcal{M}(u)+\frac{A_0^N}{2} \\
         -\mathcal{F}(u)+\frac{1}{2}\sum_{|m|\leq N}|a_m|^2\langle u^2,e_m^2\rangle
    \end{array}\right)dt
    +\sqrt{\alpha}\sum_{|m|\leq N}a_m
    \left(\begin{array}{c}
        \langle u,e_m \rangle \\
         \langle |u|^2u,e_m \rangle
    \end{array}\right)d\beta_m \\
    &=: \alpha b_Xdt+\sigma_X\cdot (d\beta_1\cdots d\beta_m)^T.
\end{align}
In order to apply Krylov's estimate, we easily observe that 
\begin{align}
    \mathbb{E}|b_X|< C,
\end{align}
where $C$ is independent of $(\alpha,N)$.
We then are left the analysis of the determinant of the matrix 
\begin{align}
    \sigma_X\sigma_X^*=\alpha\left(\begin{array}{cc}
       \sum_{|m|\leq N}|a_m|^2|\langle u,e_m\rangle|^2  & \sum_{|m|\leq N}|a_m|^2\langle u,e_m\rangle\langle |u|^2u,e_m\rangle \\
        \sum_{|m|\leq N}|a_m|^2\langle u,e_m\rangle\langle |u|^2u,e_m\rangle & \sum_{|m|\leq N}|a_m|^2|\langle |u|^2u,e_m\rangle|^2
    \end{array}\right).
\end{align}
Remark that this matrix is real. Let $\gamma=(\gamma_1\ \gamma_2)\in\R^2$. We have that
\begin{align}
    \gamma \sigma_X\sigma_X^*\gamma^T=\alpha\sum_{|m|\leq N}|a_m|^2\left(\gamma_1u+\gamma_2|u|^2u,e_m\right)^2.
\end{align}
Hence, since all $a_m$ are nonzero, the matrix is singular if and only if 
\begin{align}
    \gamma_1u(x)+\gamma_2|u(x)|^2u(x)=0\quad \forall x.
\end{align}
This situation holds only if : $(\gamma_1=\gamma_2=0)$ or $u=\text{Const}$.
    \item \textbf{On the set $O^c$}. In this setting, on one hand we have, from estimate \eqref{crucial estimate},  that
    \begin{align}
        \mathbb{E}|b_X| \leq C,
    \end{align}
    where $C$ is independent of $(\alpha,N).$ On the other hand, the matrix is given by
    \begin{align}
        \sigma_X\sigma_X^*=\alpha\left(\begin{array}{cc}
       \sum_{|m|\leq N}|a_m|^2|\langle u,e_m\rangle|^2  & \sum_{|m|\leq N}|a_m|^2\langle u,e_m\rangle\langle -\Delta u+(e^{\beta|u|^2}-1)u,e_m\rangle \\
        \sum_{|m|\leq N}|a_m|^2\langle u,e_m\rangle\langle -\Delta u+(e^{\beta|u|^2}-1)u,e_m\rangle & \sum_{|m|\leq N}|a_m|^2|\langle -\Delta u+(e^{\beta|u|^2}-1)u,e_m\rangle|^2
    \end{array}\right).
    \end{align}
    Similarly to the analysis above, singularity of the matrix leads to the equation:
    \begin{align}
        \gamma_1u(x)+\gamma_2( -\Delta u(x)+(e^{\beta|u(x)|^2}-1)u(x))=0\quad \forall x,
    \end{align}
    $(\gamma_1,\gamma_2)\in\R^2$.
    Let us consider the complementary two cases :
    \begin{enumerate}
        \item {\bfseries case 1: $\gamma_2=0$.} 
This implies : $(\gamma_2=0=\gamma_1)$ or $u=\text{Const}$.
        \item {\bfseries case 2: $\gamma_2\neq 0$.}
        Taking the (real) scalar product against $i|u|^2u$, we find
        \begin{align}
            \langle |u|^2u,i\Delta u\rangle =0,
        \end{align}
        which is impossible since $u$ belongs to $O^c.$
    \end{enumerate}
    Summary : on $O^c$, we also have that : $(\gamma_1=\gamma_2=0)$ or $u=\text{Const}$.
\end{itemize}
Overall, we arrive at the following 
\begin{lemma}\label{Lem:nondeg:det}
    \begin{align}
        \text{det}(\sigma_X\sigma_X^*)=0 \Leftrightarrow u=\text{Const}.
    \end{align}
\end{lemma}
\begin{proposition}  Let $u_N^\alpha$ be distributed by $(\mu_N^\alpha$. We have that
    \begin{align}
        \PP(X(u_N^\alpha)\in \Gamma)\leq C \ell_2(\Gamma)\quad \forall \Gamma\in\text{Bor}(\R^2),
    \end{align}
    where $C$ is independent of $(\alpha,N)$ and $\ell_2$ is the Lebesgue measure on $\R^2$. 
\end{proposition}
\begin{proof}
    Let us write
    \begin{align}
       \PP(X(u_N^\alpha)\in \Gamma) = \PP(X(u_N^\alpha)\in \Gamma\cap B(0,\epsilon))+\PP(X(u_N^\alpha)\in \Gamma\cap(\R^2\backslash B(0,\epsilon))) =: (1)+(2).
    \end{align}
    We have that
    \begin{align}
        (1) \leq \PP(M(u_N^\alpha)\leq \epsilon)\leq C\sqrt{\epsilon}.
    \end{align}
    Now, for $\min(M(u)^2+F(u)^2,M(u)^2+E(u)^2)\geq \epsilon^2$, we have by Lemma \ref{Lem:nondeg:det} and the continuity of the determinant that $\text{det}(\sigma_X\sigma_X^*)\geq c_\epsilon>0$. Taking $f=1_{\Gamma\cap(\R^2\backslash B(0,\epsilon))}$ in Krylov's estimate \eqref{Krylov_estimate}, we find
    \begin{align}
       c_\epsilon \PP(X(u_N^\alpha)\in \Gamma\cap(\R^2\backslash B(0,\epsilon)))\leq \mathbb{E}1_{\Gamma\cap(\R^2\backslash B(0,\epsilon))}(X(u_N^\alpha))(\text{det}(\sigma\sigma^T))^\frac{1}{d}dt\leq C\ell_2(\Gamma).
    \end{align}
    Hence,
    \begin{align}
        \PP(X(u_N^\alpha)\in \Gamma\cap(\epsilon,\infty))\leq C_\epsilon\ell_2(\Gamma).
    \end{align}

    Overall,
    \begin{align}
        \PP(X(u_N^\alpha)\in \Gamma)\leq C\sqrt{\epsilon}+C_\epsilon\ell_2(\Gamma),
    \end{align}
    where $C$ and $C_\epsilon$ do not depend on $(\alpha,N)$. Therefore, one finds a non-decreasing function $\rho:\R_+\to\R_+$ such that $\rho(0)=0$ satisfying
    \begin{align}
        \PP(X(u_N^\alpha)\in \Gamma)\leq \rho(\ell_2(\Gamma)).
    \end{align}
\end{proof}
Using the Portmanteau theorem and Ulam's regularity theorem, we obtain that
\begin{corollary}\label{Cor:dim}  Let $u$ be distributed by $\mu$. We have that
    \begin{align}
        \PP(X(u)\in \Gamma)\leq \rho(\ell_2(\Gamma))\quad \forall \Gamma\in\text{Bor}(\R^2).
    \end{align}
\end{corollary}
To finalize the proof of Proposition \ref{prop:dimension_measure}, we invoke Corollary \ref{Cor:dim}; it is clear that $X$ is locally Lipshitz, and since such maps do not increase Hausdorff dimension of compact sets, we obtain the claim.  
\subsection{Statistical ensemble}
 \begin{proposition}
We can construct a flow $\phi_t$ on $\Sigma$ ie $\phi_t:\Sigma \to \Sigma \quad u_0\mapsto u^\omega(t)$    
\end{proposition}
\begin{proof}
  First, we closely follow the construction of the Kuskin–Shirikyan dynamic flow within the framework of 2D Euler dynamics \cite{KuksinShirikyan2012}.   Let
\begin{align*}
K=\left\{
u\in C(\mathbb{R},H^1)
\cap
L^{2+\delta}_{\mathrm{loc}}(\mathbb{R};W^{1,2+\delta})
\;:\;
\left(e^{\left(\frac{\gamma}{2}-\right)|u|^2}-1\right)
\in L^2_{\mathrm{loc}}(\mathbb{R};H^1)
\right\}.
\end{align*}

We can endowed $K$ with a metric \(d_K\)
such that \((K,d_K)\) is a complete metric space. Let us define
\[
d_K(u,v)
:=
d_{K_1}(u,v)
+
\sum_{T=1}^{\infty}2^{-T}
\frac{
\left\|
e^{c|u|^2}
-
e^{c|v|^2}
\right\|_{L^2([-T,T];H^1)}
}{
1+
\left\|
e^{c|u|^2}
-
e^{c|v|^2}
\right\|_{L^2([-T,T];H^1)}
},
\]
where $c=\frac{\gamma}{2}-$ with $\gamma>4\beta$ and 
$$d_{K_1}(u,v)=\sum_{T=1}^{\infty}2^{-T}
\frac{
\displaystyle
\sup_{|t|\le T}\|u(t)-v(t)\|_{H^1}
+
\|u-v\|_{L^{2+\delta}([-T,T];W^{1,2+\delta})}
}{
\displaystyle
1+
\sup_{|t|\le T}\|u(t)-v(t)\|_{H^1}
+
\|u-v\|_{L^{2+\delta}([-T,T];W^{1,2+\delta})}
}$$
Let $(u_n)_{n\in \mathbb{N}}$ a cauchy sequence in $K$, It is easy to see that $\forall T\geq 1, u_n$ converge to $u$ in $C([-T,T];H^1)\cap L^{2+\delta}([-T,T];W^{1,2+\delta})$ and $e^{c|u_n|^2}$ converge to $g$ in $L^2([-T,T];H^1)$. Since $u_n$ converge to $u$ in $C([-T,T];H^1)$ and $H^1\hookrightarrow L^2$, then
the $L^2$- convergence implies the existence of an almost everywhere convergent subsequence, we may extract a subsequence  denoted by \(u_{n_k}\) such that
\[
u_{n_k}(t,x)\to u(t,x)
\qquad\text{a.e. on }[-T,T]\times M.
\]
Consequently,
\[
e^{c|u_{n_k}(t,x)|^2}\to e^{c|u(t,x)|^2}
\qquad\text{a.e.}
\]
On the other hand, since
\[
e^{c|u_{n_k}|^2}\to g
\quad\text{in }L^2([-T,T];H^1),
\]
we may assume, after extraction of a further subsequence, that
\[
e^{c|u_{n_k}(t,x)|^2}\to g(t,x)
\qquad\text{a.e.}
\]
Hence,
\[
g=e^{c|u|^2}
\qquad\text{a.e.}
\]
Let us recall the equation
\begin{equation}\label{E}
\left\{
\begin{aligned}
\partial_t u
&=
i\Big((\Delta-1)u
+\big(e^{\beta |u|^2}-1\big)u\Big),\\
u(0)&=u_0.
\end{aligned}
\right.
\end{equation}
Let us define 
\begin{align}
    X=\overline{\{u\in K, \exists \{u_{N_k}\}_k ~\text{satisfying} \eqref{Integ_formulat} ~~\text{such that}~~ u_{N_k}\sim \lambda_{N_k}~~\text{and} ~~u_{N_k}\to u ~~\text{in}~~Y\}}
\end{align}
where $\lambda_{N_k}$ is the lifted measure defined in \eqref{lifthing} and $Y $ is defined in \eqref{space Y}. It should be noted that any element of $X$ is solution of \eqref{E}: for the case $u\in X \setminus \partial X$, (see subsection 4.1) for the passage to the limit as $N_k\to \infty.$
Now, consider the case $u\in \partial X$: it can be checked that any sequence $\{u_n\}_n$ of  solution of \eqref{E} that converge to an element $u$ in $K$, then $u$ is also solution of \eqref{E}.
Indeed we have
\[
u_n(t)
=
u_{0,n}
+
\int_0^t
i(\Delta-1)u_n(\tau)\,d\tau
+
\int_0^t
i\big(e^{\beta |u_n(\tau)|^2}-1\big)
u_n(\tau)\,d\tau .
\]

Since \(u_n\to u\) in \(K\), by following the same arguments as in Lemma \eqref{3'}, we obtain the convergence of
\[
\int_0^t
\big(e^{\beta |u_n(\tau)|^2}-1\big)
u_n(\tau)\,d\tau
\longrightarrow
\int_0^t
\big(e^{\beta |u(\tau)|^2}-1\big)
u(\tau)\,d\tau  \qquad \text{in}\quad C_tH^{{-1}-}.
\]
which gives the claim.\\
Using Remark \ref{2'} and Skorokhod representation theorem, we have $\lambda(X)=1$.\\
Let us define linear continuous operator
\begin{align}
    \Pi:X&\to H^1\\
    u&\mapsto u(0)\nonumber
\end{align}
This operator is injective by the uniqueness property of the solution (Subsection \ref{1}), thus the operator 
\begin{align}
    \Pi: X&\to \Pi(X) \qquad \text{is bijective}\\
    u&\mapsto u(0)\nonumber
\end{align}
Since the solution $u$ is global in time, we can define:
\begin{align*}
    \theta_t: X\to X \quad u(\cdot)\mapsto u(\cdot+t)
\end{align*}
Let us set $\Sigma=\Pi(X)$, by the $H^1$ continuity with respect to the initial data (Subsection \ref{1'}),  we can then define the group of homeomorphisms:
\begin{align*}
    \phi_t=\Pi\circ\theta_t\circ\Pi^{-1}: \Sigma\to \Sigma
\end{align*}
such that $\{\phi_tu_0, t\in \mathbb{R}\}$ is the a unique global solution of \eqref{E}. It follows that $\mu(\Sigma)=1$ (and, by construction, the set $\Sigma$ is a invariant by $\phi_t$.)
\end{proof}
\section*{Acknowledgments}
The research of M. Sy is funded by the Alexander von Humboldt foundation under the “German Research Chair programme” financed by the Federal Ministry of Education and Research (BMBF). The research of F. G. Longmou-Moffo is funded by the German Academic Exchange Service (DAAD) through a research cooperation AIMS-Rwanda/Bielefeld University. 
\bibliography{ref}
\bibliographystyle{abbrv}
\end{document}